\documentclass[aop, preprint]{imsart}

\usepackage{mathrsfs}
\RequirePackage{amsthm,amsmath,amsfonts,amssymb}
\RequirePackage[numbers]{natbib}
\usepackage{enumitem}

\startlocaldefs

\makeatletter
\newtheorem{theorem}{Theorem}

\newtheorem{assumption}[theorem]{Assumption}

\newtheorem{corollary}[theorem]{Corollary}

\newtheorem{definition}[theorem]{Definition}
\newtheorem{example}{Example}

\newtheorem{lemma}[theorem]{Lemma}

\newtheorem{proposition}[theorem]{Proposition}
\newtheorem{remark}[theorem]{Remark}

\numberwithin{equation}{section}
\numberwithin{example}{section}
\numberwithin{theorem}{section}

\RequirePackage[normalem]{ulem}
\RequirePackage{color}\definecolor{RED}{rgb}{1,0,0}\definecolor{BLUE}{rgb}{0,0,1}

\def\bR{\mathbb R}
\def\bN{\mathbb N}
\def\bE{\mathbb E}
\def\ac{\mathcal{AC}}

\def\calF{\mathcal F}
\def\calP{\mathcal P}

\newcommand{\norm}[1]{\| #1 \| }

\endlocaldefs

\begin{document}

\begin{frontmatter}
\title{A weak convergence approach to large deviations for stochastic approximations}
\runtitle{Large deviations for stochastic approximations}

\begin{aug}
\author[A]{\fnms{Henrik}~\snm{Hult}\ead[label=e1]{hult@kth.se}},
\author[A]{\fnms{Adam}~\snm{Lindhe}\ead[label=e2]{adlindhe@kth.se}}
\author[A,B]{\fnms{Pierre}~\snm{Nyquist}\ead[label=e3]{pnyquist@chalmers.se}}
\and
\author[A]{\fnms{Guo-Jhen}~\snm{Wu}\ead[label=e4]{gjwu@kth.se}}
\address[A]{Department of Mathematics, KTH, 100 44 Stockholm, Sweden\printead[presep={,\ }]{e1,e2,e4}}
\address[B]{Department of Mathematical Sciences, Chalmers University of Technology and University of Gothenburg, 412 96 Gothenburg, Sweden\printead[presep={,\ }]{e3}}

\end{aug}

\begin{abstract}
The theory of stochastic approximations form the theoretical foundation for studying convergence properties of many popular recursive learning algorithms in statistics, machine learning and statistical physics. Large deviations for stochastic approximations provide asymptotic estimates of the probability that the learning algorithm deviates from its expected path, given by a limit ODE, and the large deviation rate function gives insights to the most likely way that such deviations occur. 
In this paper we prove a large deviation principle for general stochastic approximations with state-dependent Markovian noise and decreasing step size. Using the weak convergence approach to large deviations, we generalize previous results for stochastic approximations and identify the appropriate scaling sequence for the large deviation principle. We also give a new representation for the rate function, in which the rate function is expressed as an action functional involving the family of Markov transition kernels. Examples of learning algorithms that are covered by the large deviation principle include stochastic gradient descent, persistent contrastive divergence and the Wang-Landau algorithm.


\end{abstract}



\end{frontmatter}


\section{Introduction}
\label{sec:intro}
Stochastic approximation (SA) algorithms, first introduced by Robbins and Monro in the 1950's \cite{robbins-monroe}, has become one of the most important classes of stochastic numerical methods. Originally aimed at finding the the root of a continuous function given noisy observations, SA is now a fundamental tool in a range of areas such as statistics, optimization, electrical engineering, and machine learning, to mention but a few. Within the latter, the importance of SA algorithms is illustrated by the fact that a specific subclass of methods---stochastic gradient descent (SGD) methods---is central to the training of deep learning methods, and in reinforcement learning the standard methods (Q-learning and temporal-difference-learning) are variants of SA. The general class that is SA algorithms with state-dependent noise (see below for the definition) therefore constitute a rich and important family of stochastic recursive algorithms. In addition to the examples already mentioned (SGD, reinforcement learning), this class also includes persistent contrastive divergence, adaptive Markov chain Monte-Carlo (MCMC) and extended ensemble algorithms such as the Wang-Landau algorithm. 
%

The theory of SA stems from the pioneering work of Robbins and Monro \cite{robbins-monroe} and Kiefer and Wolfowitz \cite{kiefer-wolfowitz}, and remains an active research area within probability theory. This is in part due to the many and diverse applications of SA algorithms where, due to the complex nature of the systems under considerations, different variants of the original Robbins-Monro algorithm are needed. In turn, developing the theoretical foundation for SA algorithms, such as, e.g., convergence results, central limit theorems, concentration results and results on deviations, is of fundamental importance; monographs covering many of the standard results of the theory include \cite{borkar,Duflo,kusyin}. In this work, we add to the theoretical understanding of SA algorithms by developing general large deviation results for the associated stochastic processes. 

The basic SA algorithm with state-dependent noise considers a stochastic process $\{X_k\}_{k \in\mathbb{N}}$ on a probability space $(\Omega,\mathcal{F},P)$, with an associated noise sequence $\{Y_k\}_{k \in \mathbb{N}}$. The process $\{X_k\}_{k \in\mathbb{N}}$ is assumed to satisfy the recursion,  
\[
    X_{k+1} = X_k + \varepsilon_k g(X_k,Y_{k+1}), \qquad  k\geq 0,
\]
where $X_0=x_0\in\mathbb{R}^{d_1}$, $Y_0=y_0\in\mathbb{R}^{d_2}$, $g:\mathbb{R}^{d_1} \times\mathbb{R}^{d_2}\to \mathbb{R}^{d_1}$, and $\{\varepsilon_k\}$ is a sequence of step-sizes. The noise sequence $\{Y_k\}_{k \in \mathbb{N}}$ is state-dependent in such a way that
\[
    P(Y_{k+1}\in A| X_k,Y_k) = \rho_{X_k}(Y_k,A), \qquad A \in \mathcal{B}(\mathbb{R}^{d_2}),
\]
with $\rho_x(y,\cdot)$ being a probability measure on the Borel sets of $\mathbb{R}^{d_2}$, for any $x\in\mathbb{R}^{d_1}$ and $y\in\mathbb{R}^{d_2}$.

For many variants of SA, i.e., variants of the recursion defining $\{ X_k \} _ {k \in \mathbb{N}}$, it is possible to establish convergence and characterize the corresponding limit. The two main techniques for this are based on martingale theory and ordinary differential equations (ODEs), respectively. The latter approach, first introduced by Ljung in \cite{Ljung77} and referred to as the ODE method, is a powerful method for  
studying convergence of SA (see, e.g., \cite{borkar, kusyin}), that relies on the idea that, for large $k$, the stochastic approximation essentially follows a limit ODE. The following is a brief outline of the approach. Assume that, for each $x \in \mathbb{R}^{d_1}$, the transition kernel, $\rho_x$, admits a unique invariant distribution $\pi_x$. We can then rewrite the recursion for $X_{k+1}$ as
\begin{align*}
    \frac{X_{k+1}-X_k}{\varepsilon_k} = \bar g(X_k) + [g(X_k, Y_{k+1})- \bar g(X_k)],
\end{align*}
where $\bar g(x) = \int g(x,y) \pi_x(dy)$. Under appropriate conditions, ensuring that the influence of the noise $g(X_k, Y_{k+1})- \bar g(X_k)$ is small, for large $k$ and small $\varepsilon_k$, the algorithm will approximately follow the solution to the limit ODE
\begin{align}
\label{eq:limitODE}
    \dot x(t) = \bar g(x(t)).
\end{align}
Limit points of an SA algorithm can therefore be described as the forward limit set of \eqref{eq:limitODE}. 

Because of the inherent randomness of an SA algorithm, even for large $k$ there is a positive probability that the corresponding stochastic process deviates from a neighborhood of a limit point. To make an analogy, we can think of an SA algorithm approaching a point of convergence as (the algorithm) learning. If it then starts to deviate from such a point, we think of it as forgetting. From the mathematical point of view, large deviation theory describes the rate at which an SA algorithm deviates from a neighborhood of a limit point and characterizes the most likely trajectories along which such deviations occur. That is, to use the analogy again, large deviations provides insights into how an SA algorithm forgets and the rate at which it happens. Therefore, understanding large deviations for SA is a natural and very useful complement to any convergence analysis.
%

When proving convergence for an SA algorithm, a difficulty arises from the possibility that it diverges, in the sense that $|X_k| \to \infty$ with some positive probability. A simple and useful method to exclude the possibility of such a divergence is to project the SA updates on a compact set $C$. This amounts to considering the projected recursion
\[
    X_{k+1} = \text{proj}_C\left[X_k + \varepsilon_k g(X_k,Y_{k+1})\right], \qquad  k\geq 0,
\]
where $\text{proj}_C$ denotes the projection onto $C$; see e.g.\ \cite{DupKush87PTRF} for some initial large deviation results in this setting when $C$ is convex. Convergence of SA can also be proved using a more intricate analysis that involves projections on an increasing sequence of of compact sets as in \cite{andrieu}. In this paper we focus on the behavior of the algorithm close to a point of convergence, and therefore do not need to consider such projected recursions.  

The existing literature on large deviations for SA can be divided into works that consider constant step size, where $\varepsilon_k \equiv \varepsilon > 0$ does not depend on $k$, and works that consider decreasing step size ($\varepsilon_k \to 0$ as $k \to \infty$). For constant step size, the theory was first developed by Freidlin for dynamical systems in continuous time with noise that does not depend on the state \cite{fre76, fre78}. These results were generalized by Iscoe, Ney and Nummelin to Markov-additive processes in continuous and discrete time \cite{iscneynum}. The most general results are obtained by Dupuis in \cite{dup2}, where he considers discrete time systems with state-dependent noise. 
The results rely on the existence of an appropriate limiting Hamiltonian and the rate function is given by an action functional, where the local rate function is the Fenchel-Legendre transform of the limiting Hamiltonian; see Section \ref{sec:const_step} for additional details on the development of large deviations principles for SA with constant step size.

For SA with decreasing step size the first results are obtained by Kushner \cite{kus16}, who considers step size sequences of the form $\varepsilon_n = (n+1)^{-\rho}$, $\rho \in (0,1]$, update functions $g(x,y) = b(x) + y$, with $b(\cdot)$ Lipschitz continuous, and $\{Y_n\}$ a sequence of iid centered Gaussian variables. Kushner identifies the correct scaling sequence for the large deviations principle, under the assumption of the existence of an appropriate limiting time-dependent Hamiltonian. In \cite{dupkus2}, Dupuis and Kushner generalize the results of \cite{kus16}. Therein, they consider step-size sequences satisfying $\varepsilon_n \geq 0$,  $\sum_n \varepsilon_n = \infty$, $\varepsilon_n \to 0$, update functions of the form $g(x,y) = \bar{b}(x) + b(x,y)$, with additional assumptions on $b$ and the noise sequence $\{ Y_n \}$. Similar to Kushner in \cite{kus16}, Dupuis and Kushner also assume the existence of an appropriate limiting Hamiltonian; see Section \ref{sec:decr_step} for more details.

Throughout the existing literature, the large deviation results for SA are obtained by identifying a Hamiltonian $H(x,\alpha)$, that sometimes can be interpreted as a limiting log-moment generating function, and defining the local rate function as the convex conjugate of $H(x,\alpha)$. A problem with this approach is that the Hamiltonian is implicitly defined as a limit and its relation to the underlying dynamics, such as the family of transition kernels $\{\rho_x(y, dz), x \in \mathbb{R}^{d_1}, y \in \mathbb{R}^{d_2}\}$, can only be established in some special cases. 

In contrast, in this paper, we generalize the results of \cite{dupkus2}, for a decreasing step size sequence, to include state-dependent noise---which is needed in many applications, see Section \ref{sec:appl} for some examples---and the local rate function is expressed in terms of the family of transition kernels, as opposed to as the convex conjugate of a limiting Hamiltonian. In addition, our assumptions are more general, as the update function $g$ need not be bounded in $x$. The main challenges in proving the new results arise in proving the lower bound of the Laplace principle (see Theorem \ref{thm:main}), in part due to the state-dependent noise. To overcome this, a key ingredient in our approach is a novel use of local ergodicity at the time-scale of the noise process; more details are provided in Sections \ref{sec:main} and \ref{sec:lower}. We also remark that from a technical point of view, the setting of fixed step size is somewhat easier than the setting considered here. Analogous results to the ones developed in this paper can be obtained with minor modification of the weak-convergence techniques used here.  However, in the interest of keeping the paper at a reasonable length, the results for SA with fixed step size are left for future work. 

To summarize, the main contributions of this paper are the following. The weak-convergence approach is used to prove a large deviations principle (e.g.\ Laplace principle) for SA with state-dependent noise, see Theorem \ref{thm:main}. The result generalize existing results for decreasing step-size sequences and identifies the correct scaling sequence. A new representation of the rate function is provided, formulated in terms of the underlying dynamics, that does not rely on the assumption of a limiting Hamiltonian. Several examples from statistics, statistical physics and machine learning are shown to satisfy the assumptions of the large deviations principle.

The remainder of the paper is organized as follows. In Section \ref{sec:prel} we provide the preliminaries needed for the paper: notations and definitions (Section \ref{sec:notation}); large deviations and the Laplace principle (Section \ref{sec:LD}); the SA algorithm under consideration (Section \ref{sec:model}); the assumptions used throughout the paper (Section \ref{sec:ass}); a heuristic derivation of (the correct form of) the rate function appearing in the main result. Next, in Section \ref{sec:main} we state the main result Theorem \ref{thm:main}, a Laplace principle for the SA algorithm defined in Section \ref{sec:model}. In this section we also study the associated rate function: establish its continuity properties, provide alternative representations, establish a connection to a limiting Hamiltonian. For ease of comparison, a more detailed literature review is provided in Section \ref{sec:litRev}, when the setup and main results of this paper have already been stated. The proof of Theorem \ref{thm:main} is divided into two parts, carried out in Section \ref{sec:upper} (Laplace upper bound) and Section \ref{sec:lower} (Laplace lower bound), respectively. Section \ref{sec:limit} is dedicated to the proof of Theorem \ref{thm:limit}, a general version of a result (Theorem \ref{thm:convX}) used in Section \ref{sec:lower}.

\section{Preliminaries}
\label{sec:prel}
%
\subsection{Notation and definitions}
\label{sec:notation}
We use the conventional notation,  $\mathbb{N} = \{1,2,\dots\}$ and $\mathbb{N}_0 = \{0,1,2, \ldots\}$. Throughout the paper $\{\varepsilon_k\}_{k \in \mathbb{N}_0}$ denotes the sequence of step sizes (learning rate) of our SA algorithm. 
Given such a sequence, define the intermediate times $t_0=0$, $t_n=\sum_{k=1}^{n}\varepsilon_k$, and let $\mathbf{m}(t)=\max\{n:t_n\leq t\}$ be the maximum number of iterations that occurs before time $t$. Note that $\mathbf{m}(t_n)=n$. 

For $T > 0$, the space $C([0,T]:\mathbb{R}^{d})$ consists of $\mathbb{R}^{d}$-valued continuous functions defined on $[0,T]$ and $C_x([0,T]:\mathbb{R}^d)$ is the subspace of continuous functions starting at $x$ at time $0$. The space $C([0,T]:\mathbb{R}^d)$ is equipped with the supremum norm $\|f\|_\infty = \sup_{s,t \in [0,T]}\|f(s)-f(t)\|$, where $\|\cdot\|$ is the Euclidean norm on $\mathbb{R}^d$. For $x,y\in\mathbb{R}^d$, their inner product is denoted $\langle x,y\rangle$. 

Given a Polish space $\mathcal{X}$, with Borel $\sigma$-algebra $\mathcal{B}(\mathcal{X})$, the space of probability measures on $\mathcal{X}$ is denoted by $\mathcal{P}(\mathcal{X})$. We equip $\mathcal{P} (\mathcal{X})$ with the topology of weak convergence. Given $\mu\in\mathcal{P}(\mathcal{X})$, let 
$
    \mathcal{A}(\mu)\doteq \left\{ \gamma\in \mathcal{P}(\mathcal{X} \times \mathcal{X}):[\gamma]_1=[\gamma]_2=\mu\right\},
$
where $[\gamma]_1$ and $[\gamma]_2$ denote the first and second marginal of $\gamma$, respectively.

Let $\mathcal{X}$ and $\mathcal{Y}$ be Polish spaces and let $p(x,dy)$ be a collection of probability measures on $\mathcal{Y}$ parametrized by $x \in \mathcal{X}$. Then $p$ is a stochastic kernel on $\mathcal{Y}$ given $\mathcal{X}$ if, for every $A \in \mathcal{B} (\mathcal{Y})$, the map $x \mapsto p(x,A)) \in [0,1]$ is measurable. For a stochastic kernel $p$ on $\mathcal{Y}$ given $\mathcal{X}$ and $\theta\in\mathcal{P}(\mathcal{X})$, $\theta\otimes p$ is defined to be the unique probability measure on $(\mathcal{X}\times\mathcal{Y},\mathcal{B}(\mathcal{X}\times\mathcal{Y}))$ with the property that, for $A\in\mathcal{B}(\mathcal{X})$ and $B\in\mathcal{B}(\mathcal{Y})$,
\[
    \theta\otimes p(A\times B)\doteq \int_{A\times B}\theta(dx)p(x,dy) =\int_{A} p(x,B) \theta(dx).
\]
The formula is summarized by the notation $\theta\otimes p(dx \times dy)=\theta(dx)\otimes p(x,dy)$.

For a Markov chain $\{ X _i \} _{i \in \mathbb{N}}$ taking values in $\mathcal{X}$, the transition kernel of the chain is a stochastic kernel $p$ on $\mathcal{X}$ given $\mathcal{X}$, such that the distribution of $X_i$ given $X_{i-1}$ is $p(X_{i-1}, \cdot)$.  We say that a transition kernel $p(x, dy)$ satisfies the Feller property if, for any sequence $\{ x_n \} _{n \in \mathbb{N}}$ such that $x_n \to x \in \mathcal{X}$ as $n \to \infty$, $p(x_n, \cdot) \to p(x, \cdot)$ in $\mathcal{P}(\mathcal{X})$. Given a transition kernel $p(x,dy)$ on $\mathcal{X}$ and $k\in\mathbb{N}$, let $p^{(1)}(x,dy)=p(x,dy)$ and, for $k \geq 1$, $p^{(k)}(x,dy)$ denotes the $k$-step transition probability function defined recursively by 
\[
    p^{(k+1)}(x,A)=\int_{\mathcal{X}}p(y,A)p^{(k)}(x,dy), \quad A \in \mathcal{B}(\mathcal{X}).
\]

For $\theta\in\mathcal{P}(\mathcal{X})$, 
the relative entropy $R(\cdot\|\theta)$ is a map from $\mathcal{P}(\mathcal{X})$ into the extended real numbers, defined by
\begin{align*}
    R(\gamma\|\theta)\doteq \begin{cases} \int_{\mathcal{X}}\left(\log\frac{d\gamma}{d\theta}\right)d\gamma, & \gamma \ll \theta, \\
    +\infty, & \textrm{otherwise.}
    \end{cases}
\end{align*}
We refer to $R(\gamma\|\theta)$ as the relative entropy of $\gamma$ with respect to $\theta$. We recall the following properties of relative entropy (see Lemmas 1.4.1 and 1.4.3 in \cite{dupell4}): $R(\cdot | \cdot)$ is jointly convex and jointly lower semi-continuous with respect to the weak topology on $\mathcal{P}(\mathcal{X}) ^2$, and $R(\mu | \nu) = 0$ if and only if $\mu = \nu$. The following useful property follows from the chain rule for relative entropy (see Theorem 2.6 and Corollary 2.7 in \cite{buddup4}): given two transition kernels $p,q$, for any $\mu \in \mathcal{P} (\mathcal{X})$,
\[
    R(\mu \otimes p \| \mu \otimes q) = \int _{\mathcal{X}} R\left(p(x, \cdot) \| q(x, \cdot)\right) \mu (dx).
\]

\subsection{Large deviations}
\label{sec:LD}
At the heart of the theory of large deviations is the \textit{large deviation principle} (LDP): a sequence $\{ Z^n \}$ of random elements on some space $\mathcal{X}$ is said to satisfy an LDP with scaling sequence, or speed, $\{\beta_n\}$ and \textit{rate function} $I : \mathcal{X} \to [0, \infty]$, if $\beta_n\to \infty$ as $n\to\infty$, $I$ is lower semi-continuous, has compact sub-level sets and, for any measurable $A \subset \mathcal{X}$,
\begin{align*}
    - \inf _{z \in A ^\circ} I(z) &\leq \liminf _{n \to \infty} \frac{1}{\beta_n} \log P (Z ^n \in A ^\circ) \\
    &\leq \limsup _{n \to \infty} \frac{1}{\beta_n} \log P (Z ^n \in \bar A ) \leq - \inf _{z \in \bar A} I (z).    
\end{align*}
The gist of these inequalities is that, if $\{ Z^n \}$ satisfies an LDP with speed $\beta_n$ and rate function $I$, then for any $z \in \mathcal{X}$ and $n$ large, \[
    P (Z ^n \approx z) \simeq \exp\{-\beta_n I(z)\}.
\]
The definition of an LDP makes this statement rigorous in the limit $n \to \infty$. 

For any Polish space, an equivalent formulation of the LDP is the \textit{Laplace principle} (see e.g., Theorems 1.5 and 1.8 in \cite{buddup4}). Due to this equivalence, we will use the terminology of LDP and Laplace principle interchangeably throughout the paper. 


\begin{definition}[Laplace principle] Let $I$ be a rate function on $\mathcal{X}$. The sequence $\{X^n\}$ is said to satisfy the Laplace principle on $\mathcal{X}$ with rate function $I$ and scaling sequence $\{\beta_n\}$ if $\beta_n\to \infty$ as $n\to\infty$, and for all bounded continuous functions $F: \mathcal{X} \to \mathbb{R}$, 
\[
    \lim_{n\to\infty}\frac{1}{\beta_n}\log E e^{-\beta_n F(X^n)} = -\inf_{x\in\mathcal{X}}[F(x)+I(x)].
\]
The term Laplace principle upper bound refers to the validity of 
\[
    \limsup_{n\to\infty}\frac{1}{\beta_n}\log E e^{-\beta_n F(X^n)} \leq  -\inf_{x\in\mathcal{X}}[F(x)+I(x)],
\]
for all bounded continuous functions $F$, while the term Laplace principle lower bound refers to the validity of 
\[
    \liminf_{n\to\infty}\frac{1}{\beta_n}\log E e^{-\beta_n F(X^n)} \geq -\inf_{x\in\mathcal{X}}[F(x)+I(x)],
\]
for all bounded continuous functions $F$.
\end{definition}
Henceforth, when there is no ambiguity, we will refer to these only as upper and lower bounds, dropping the term ``Laplace principle''.

\subsection{Stochastic approximation}
\label{sec:model}
We here repeat the definition of the SA algorithm under consideration, first given in Section \ref{sec:intro}, and define an associated continuous-time process that will be used in the analysis. 

Let $(\Omega,\mathcal{F},P)$ be a probability space. Consider an SA algorithm $\{X_k\}_{n\in\mathbb{N}_0}$ of the Robbins-Monro type, with state-dependent noise sequence $\{Y_k\}_{k \in \mathbb{N}}$, starting from $X_0$ and satisfying the recursion, 
\[
    X_{k+1} = X_k + \varepsilon_{k+1} g(X_k,Y_{k+1}), \quad k \geq 0,
\]
where $g:\mathbb{R}^{d_1} \times\mathbb{R}^{d_2}\to \mathbb{R}^{d_1}$, and $\{Y_n\}_{n\in\mathbb{N}_0}$ starting from $Y_0$, and, for every $k\in\mathbb{N}_0$ and $A\in\mathcal{B}(\mathbb{R}^{d_2})$,
\[
    P(Y_{k+1}\in A| X_k,Y_k) = \rho_{X_k}(Y_k,A),
\]
with $\rho_x(y,\cdot)\in\mathcal{P}(\mathbb{R}^{d_2})$ for any $x\in\mathbb{R}^{d_1}$ and $y\in\mathbb{R}^{d_2}$.

The focus of this paper is the asymptotic behavior of $\{ X_n\}_{n\in\mathbb{N}}$ for large values of $n$. Therefore, for each $n\in\mathbb{N}$ and $x_0\in\mathbb{R}^{d_1}$, define a process $\{X^n_k\}_{k \geq n}$ that follows the same recursive iterations but starts from the $n$-th step recursion. That is, let $X^n_n=x_0$ and, for $k\geq n$, 
\begin{align}\label{eqn_recursion}
    X^n_{k+1} = X^n_k + \varepsilon_{n+k+1} g(X^n_k, Y_{n+k+1}).
\end{align}  
We consider a family of continuous interpolations of $\{X^n_k\}_{k\geq n}$: for each $n$, $X^n = \{X^n(t):t\in [0,T]\}$ is given by $X^n(t_{n+k}-t_n)=X^n_{n+k}$ for $k=0,1,\dots$, and for intermediate time points $t$, $X^n(t)$ is defined by a piecewise linear interpolation. Note that, for each $n$, $X^n \in C_{x_0}([0,T]:\mathbb{R}^{d_1})$.

\subsection{Assumptions}
\label{sec:ass}
In this section we state the assumptions we make on $\{X^n\}$, the family of continuous interpolations of the stochastic approximation $\{X^n_k\}_{k\geq n}$, both introduced in Section \ref{sec:model}. We have aimed for assumptions that are general enough to cover a large class of SA algorithms, while also being tangible from a modeling perspective (i.e., assumptions that are phrased in terms of objects that define the dynamics of $\{ X^n \}$). 

\begin{assumption}\label{sufficient_conditions}  The stochastic approximation \eqref{eqn_recursion} is assumed to satisfy the following conditions. 
\begin{enumerate}[label = (A.\arabic*)]
\item \label{ass:Lipschitz} The function $g$ is a measurable function, and for any $z\in\mathbb{R}^{d_2}$, $x\mapsto g(x,z)$ is Lipschitz continuous.
\item \label{ass:kernel} The transition kernel $\rho_x(y,dz)$ is of the form $\rho_x(y,dz) = \eta_x(y,z)\lambda(dz)$,
for some reference measure $\lambda\in\mathcal{P}(\mathbb{R}^{d_2})$. Moreover, $x\mapsto\eta_x(y,z)$ is uniformly continuous, in $(z,y)$, and for any $x$, $(y,z)\mapsto\eta_x(y,z)$ is continuous. 
\item \label{ass:cont_logMGF}The function 
\[
\Lambda(x,\alpha, y) = \log \int \exp\{\langle \alpha, g(x,z)\rangle\} \rho_x(y, dz),
\]
is continuous in $(x,\alpha)$, uniformly in $y$.
\item \label{ass:bounded_kernel}  For every compact set $K$, there is a constant $C(K)$, such that for all $y,z\in\mathbb{R}^{d_2}$, 
\[
\sup_{x,w\in K} \frac{\eta_x(y,z)}{\eta_w(y,z)}<C(K).
\]

\item \label{ass:transitivity} For any $x\in\mathbb{R}^{d_1}$, there exist positive integers $l_0$ and $n_0$ such that for all $y$ and $w$, 
\[
    \sum_{i=l_0}^{\infty} \frac{1}{2^i} \rho_x^{(i)}(y,dz) \ll \sum_{j=n_0}^{\infty} \frac{1}{2^j} \rho_x^{(j)}(w,dz),
\]
where $\rho_x^{(i)}$ denotes the $i$-step transition probability.
\item \label{ass:logMGF} For every $\alpha\in\mathbb{R}^{d_2}$,
\[
    \sup_{x\in\mathbb{R}^{d_1}}\sup_{y\in \mathbb{R}^{d_2}}\left( \log\int_{\mathbb{R}^{d_2}}e^{\langle \alpha,g(x,z)\rangle}\rho_x(y,dz) \right)<\infty,
\]
\[
    \sup_{x\in\mathbb{R}^{d_1}}\sup_{y\in\mathbb{R}^{d_2}}\left( \log\int_{\mathbb{R}^{d_2}}e^{\langle \alpha,z\rangle}\rho_x(y,dz) \right)<\infty.
\]
\item  \label{ass:limith} The sequence $\{\varepsilon_k\}_{k \in \mathbb{N}}$ satisfies $\varepsilon_k > 0$ for each $k \geq 1$,  $\lim_{k \to \infty} \varepsilon_k = 0$ and $\sum_{k} \varepsilon_k = \infty$. Let $\{\beta_n\}\doteq \{\mathbf{m}(t_n+T)-n\}$ and suppose that the function $h^n: [0,T] \to (0,\infty)$, given by,
\[
    h^n(t) =\beta_n\varepsilon_{n+i-1}, \quad\text{ for } t \in [t_{n+i-1}-t_n, t_{n+i}-t_n),  \quad i\in\{1,\dots,\beta_n\}, 
\]
converges uniformly on $[0,T]$ to some limit $h$.
\end{enumerate}
\end{assumption}

\begin{remark}
\label{rem:Feller}
 Throughout the proof of the Laplace principle, for $x \in \mathbb{R} ^{d_1}$, we require the Feller property of $\rho _x (y, dz)$. Therefore, we note here that this property follows from the continuity of $(y,z) \mapsto \eta_x(y,z)$ in \ref{ass:kernel} and Pratt's lemma. 
%
\end{remark}

Assumption \ref{ass:Lipschitz} is a standard assumption for the existence and uniqueness of a classical solution to an ordinary differential equation; Assumption \ref{ass:kernel} guarantees the Feller property (see Remark \ref{rem:Feller}) and the existence of an invariant probability measure for $\rho_x(y,dz)$; Assumption \ref{ass:transitivity} is a transitivity condition ensuring that the invariant probability measure is unique and the Markov chain with transition probability $\rho_x(y,dz)$ is ergodic. For each $x \in \mathbb{R}^{d_1}$, we let $\pi _x$ denote this unique invariant measure for $\rho _x$. Assumption \ref{ass:logMGF} is used to guarantee that the updates have finite exponential moments. Lastly, \ref{ass:limith} is needed to prove convergence of the stochastic approximation algorithm and the limit function $h$ may be interpreted as an asymptotic time-scale of the process $\{X^n\}$. For example, with $\varepsilon_k = 1/k$, a straightforward calculation shows that the limit function is given by $h(t) = e^{-t}(e^T-1)$. Note that, from the definition and the properties of $h^n$, it follows that the limiting function $h$ must be non-increasing.

\subsection{A heuristic derivation and form of the rate function for SA}
\label{sec:heuristic}
Before stating the Laplace principle for the sequence $\{ X^n\}$, we here present a heuristic derivation that suggests the correct form of the rate function. The derivation  contains several non-rigorous approximations and is only intended to provide the reader an intuitive understanding of the rigorous results that are stated and proved in Section \ref{sec:main} and onward.


As a starting point, we recall that the empirical measure of an ergodic Markov chain with transition probability $\rho(y, dz)$ satisfies an LDP with scaling sequence $\{n\}$ and rate function given by, 
\begin{align} \label{eq:emprate}
    J_0(\mu) = \inf_{\gamma \in \mathcal{A}(\mu)} R(\gamma||\mu \otimes \rho),
\end{align}
where $\mathcal{A}(\mu)$ is defined in Section \ref{sec:notation}; see, e.g., \cite[Ch.\ 8]{dupell4}. Take a bounded continuous function $g$ on $\mathbb{R}^{d_2}$ and consider the map $\mu \mapsto \int g(y) \mu(dy)$. The contraction principle then implies that 
%
the sample average $\frac{1}{n}\sum_{i=1}^{n} g(Y_i)$ satisfies an LDP with rate function,
\[
    L_0(\beta) = \inf _{\mu} \left\{ J_0(\mu) : \int g(y) \mu(dy) = \beta \right \} = 
    \inf_\mu\left\{\inf_{\gamma \in \mathcal{A}(\mu)} R(\gamma||\mu \otimes \rho) : \int g(y) \mu(dy) = \beta\right \}. 
\]
By incorporating a time variable, the continuous linear interpolation of $\frac{1}{n}\sum_{i=1}^{[nt]} g(Y_i)$ satisfies an LDP on $C_0([0,T]:\bR ^d)$ with rate function $J_1 : C_0 ([0,T]: \mathbb{R}^{d_1}) \to [0, \infty]$ given by,
\[
    J_1(\varphi) = \int_0^T L_0(\dot \varphi(t)) dt. 
\]

We first consider SA with fixed step-sizes: for each $n \in \mathbb{N}$, $\varepsilon_k \equiv 1/n$ for all $k$. Take $\varphi$ in $C([0,T]:\bR ^{d_1})$ and consider the probability that the trajectory of $X^n$ resides in a ball of radius $\sigma > 0$ around $\varphi$. In this case, $X^n$ can be approximated over a small interval $[s,s+\delta]$ of length $\delta > 0$ by,
\[
    X^n(s+\delta)-X^n(s)  \approx \frac{1}{n}\sum_{i=\lfloor ns\rfloor}^{\lfloor n(s+\delta)\rfloor} g(X^n(s),Y_i) \approx \frac{1}{n}\sum_{i=\lfloor ns\rfloor}^{\lfloor n(s+\delta)\rfloor} g(\varphi(s),Y_i) 
    ,
\]
where $Y_i$ is a Markov chain with transition probability $\rho_{\varphi(s)}(y,dz)$. Using the LDP for the sample average, the increment $X^n(s+\delta)-X^n(s)$ satisfies an LDP with rate function, 
\[
    \delta L(\varphi(s), \beta), 
\]
where, 
\begin{align}\label{eqn_local_rate_function}
    L(x,\beta) \doteq \inf_{\mu}\left\{\inf_{\gamma\in \mathcal{A}(\mu)} R(\gamma||\mu\otimes \rho_x): \beta=\int g(x,z)\mu(dz)\right\}.
\end{align}
%
Pasting together the local approximations over small intervals, we have that $X^n$ satisfies an LDP on $C([0,T]; \bR^d)$ with rate function,
\begin{align*}
    J_2(\varphi) = \int_0^T L\left(\varphi(s),\dot\varphi(s)\right)ds.
\end{align*}
To see this more explicitly, we have,
\begin{align*}
    &-\frac{1}{n}\log P\left(X^n(\cdot) \approx  \varphi(\cdot)\right)
    \approx -\frac{1}{n}\log P\left( X^n(j\delta)\approx \varphi(j\delta) \text{ for all } 1 \leq j\leq \lfloor \frac{T}{\delta}\rfloor \right)\\
    &\quad \approx -\frac{1}{n}\log P\left(X^n((j+1)\delta)-X^n(j\delta)\approx \delta \dot \varphi(j\delta)  \text{ for all }0\leq j\leq \frac{T}{\delta}-1\right)\\
    &\quad \approx -\frac{1}{n}\log \prod_{j=0}^{\lfloor \frac{T}{\delta}\rfloor -1} P\left(X^n((j+1)\delta)-X^n(j\delta)\in \delta \dot \varphi(j\delta) \mid X^n(j\delta) \right)\\
    &\quad \approx -\frac{1}{n}\log \prod_{j=0}^{\lfloor \frac{T}{\delta}\rfloor -1} \exp\left\{{-n\delta L\left(\varphi(j\delta), \dot \varphi(j\delta)\right)}\right\}\\
    &\quad \approx \sum_{j=0}^{\lfloor \frac{T}{\delta}\rfloor -1} \delta L(\varphi(j\delta), \dot \varphi(j\delta)) \\ 
    &\quad \approx \int_0^T L\left(\varphi(s),\dot\varphi(s)\right)ds.
\end{align*}

Consider now $\{X^n\}$ with decreasing step-size $\{\varepsilon_n\}$ as defined in Section \ref{sec:model}. For the scaling in the LDP we take the sequence $\{\beta_n\} =  \{\mathbf{m}(t_n+T)-n\}$ and define $h^n(t)$ as in Assumption \ref{ass:limith}. The decreasing step-sizes correspond to a change of time scale and the rate of change of $X^n$ over a small interval $[s,s+\delta]$ of length $\delta > 0$ may be approximated by,
\begin{align*}
    \frac{X^n(s+\delta)-X^n(s)}{\delta}& \approx \frac{1}{\delta}\sum_{i=\mathbf{m}(t_n + s)-n+1}^{\mathbf{m}(t_n +s+\delta)-n} \varepsilon_{n+i-1} g(\varphi(s),Y_i) \\ &\approx \frac{1}{\delta}\left(\frac{\sum_{i=\mathbf{m}(t_n + s)-n+1}^{\mathbf{m}(t_n +s+\delta)-n} \varepsilon_{n+i-1}}{\mathbf{m}(t_n + s +\delta)-\mathbf{m}(t_n + s)}\right)\left(\sum_{i=\mathbf{m}(t_n + s)-n+1}^{\mathbf{m}(t_n +s+\delta)-n}  g(\varphi(s),Y_i)\right) \\ &\approx \frac{1}{\mathbf{m}(t_n + s +\delta)-\mathbf{m}(t_n + s)}\left(\sum_{i=\mathbf{m}(t_n + s)-n+1}^{\mathbf{m}(t_n +s+\delta)-n}  g(\varphi(s),Y_i)\right),
\end{align*}
for which a LDP holds, similar to the case of constant step-sizes. 
Using a similar argument again, we have the following approximation, where $\tau_i  ^n = t_{n+i} - t_i$,
\begin{align*}
    &-\frac{1}{\beta_n}\log P\left(X^n(\cdot) \approx  \varphi(\cdot)\right)
    \approx -\frac{1}{\beta_n}\log P\left( X^n(j\delta)\approx \varphi(j\delta) \text{ for all } 1 \leq j\leq \lfloor \frac{T}{\delta}\rfloor \right)\\
    &\quad \approx -\frac{1}{\beta_n}\log P\left(X^n((j+1)\delta)-X^n(j\delta)\approx \delta \dot \varphi(j\delta)  \text{ for all }0\leq j\leq \frac{T}{\delta}-1\right)\\
    &\quad \approx -\frac{1}{\beta_n}\log \prod_{j=0}^{\lfloor \frac{T}{\delta}\rfloor -1} P\left(\frac{X^n((j+1)\delta)-X^n(j\delta)}{\delta}\approx \dot \varphi(j\delta) \mid X^n(j\delta) \right)\\
    &\quad \approx \frac{1}{\beta_n} \sum_{j=0}^{\lfloor \frac{T}{\delta}\rfloor -1}  (\mathbf{m}(t_n + (j+1)\delta)-\mathbf{m}(t_n + j\delta))L\left(\varphi(j\delta), \dot \varphi(j\delta)\right)\\
    &\quad \approx \frac{1}{\beta_n} \sum_{j=0}^{\lfloor \frac{T}{\delta}\rfloor -1} \left(\sum_{i=\mathbf{m}(t_n + j\delta)-n+1}^{\mathbf{m}(t_n +(j+1)\delta)-n} L(\varphi(\tau^n_i), \dot \varphi(\tau^n_i))  \right) \\
     &\quad \approx \frac{1}{\beta_n} \sum_{i=1}^{\beta_n}   \frac{1}{\varepsilon_{n+i-1}}L(\varphi(\tau^n_i), \dot \varphi(\tau^n_i)) \varepsilon_{n+i-1}\\
     &\quad \approx \sum_{i=1}^{\beta_n}   \frac{1}{h^n(\tau^n_i)}L(\varphi(\tau^n_i), \dot \varphi(\tau^n_i)) \varepsilon_{n+i-1}\\
    &\quad \approx \int_0^T \frac{1}{h(s)}L\left(\varphi(s),\dot\varphi(s)\right)ds.
\end{align*}
This calculation suggests that the appropriate rate function in the LDP for $X^n$, the piecewise linear interpolation of the SA, is given by,
\begin{align*}
    I(\varphi) = \int_0^T \frac{1}{h(t)}L(\varphi(t),\dot\varphi(t))dt,
\end{align*}
where $h(t)$ is the limit of $h^n(t)$. 

\section{Large deviations for stochastic approximations with state-dependent noise}
\label{sec:main}
The goal of this paper is to establish the LDP for the sequence of $X^n = \{X^n(t):t\in[0,T]\}$, the linear interpolations of $\{ X^n _k \} _{k \geq n}$ starting from $X^n_n = x_0 \in \bR^{d_1}$. To this end, 
we define the function $I: C_{x_0} ([0,T] : \bR ^{d_{1}}) \to [0, \infty]$ as,
 \begin{align}\label{eqn_rate_function}
    I(\varphi) = \begin{cases} \int_0^T \frac{1}{h(t)}L(\varphi(t),\dot\varphi(t))dt, & \varphi\in AC_{x_0}([0,T]:\mathbb{R}^{d_1}), \\
    +\infty, & \textrm{otherwise},
    \end{cases}
\end{align}
with the local rate function $L$ defined in \eqref{eqn_local_rate_function}. 
Note that we suppress the dependence on the choice of starting point $x_0$ in the notation. Recall from Section \ref{sec:LD} that, in the setting considered here, an LDP is equivalent to a Laplace principle. The following Laplace principle is the main result of the paper, where $I$ plays the role of the large deviation rate function. 
\begin{theorem}[Laplace principle]
\label{thm:main}
 Let $X^n = \{X^n(t):t\in[0,T]\}$ be the continuous interpolations of $\{X^n_k\}_{k\geq n}$ given by \eqref{eqn_recursion} and take $L$ as in \eqref{eqn_local_rate_function}. Under Assumptions \ref{ass:Lipschitz}-\ref{ass:limith}, $I$ is a rate function, and $\{X^n\} _{n \in  \mathbb{N}}$ satisfies a Laplace principle with scaling sequence $\beta_n=\mathbf{m}(t_n+T)-n$ and rate function $I$. 
\end{theorem}
\begin{proof}
    The proof relies on the weak convergence approach to large deviations; see \cite{dupell4, buddup4} for detailed accounts of the approach. In particular, the proof is divided into proving the (Laplace principle) upper bound,
\begin{align*}
    \liminf _{n \to \infty} - \frac{1}{\beta _n} \log E \left[ e^{-\beta _n F(X^n)} \right] \geq \inf _\varphi \left\{ F(\varphi) + I(\varphi) \right\},
\end{align*}
and the (Laplace principle) lower bound, 
\begin{align*}
    \limsup _{n \to \infty} - \frac{1}{\beta _n} \log E \left[ e^{-\beta _n F(X^n)} \right] \leq \inf _\varphi \left\{ F(\varphi) + I(\varphi) \right\},
\end{align*}
where the infima are over $\varphi \in \ac _{x_0} ([0,T]: \bR ^{d_1})$ and $F: C([0,T]:\mathbb{R}^{d_1}) \to \mathbb{R} $ is an arbitrary bounded continuous function. The upper bound is proved in Theorem \ref{thm:upper} and the lower bound in Theorem \ref{thm:lower}; the respective proofs are given in Sections \ref{sec:upper} and \ref{sec:lower}. Combining these results thus proves the stated Laplace principle for $X^n$.
\end{proof}
The starting point for proving the upper and lower bounds is the following representation formula, which is a straightforward modification of Theorem 4.5 in \cite{buddup4}.

\begin{proposition}\label{prop_representation}
Fix $n \in \bN$ and let $\{X^n(t):t\in[0,T]\}$ be the continuous interpolations of $\{X^n_k\}_{k\geq n}$ given by \eqref{eqn_recursion}, and $X^n _n = x$. For any bounded continuous function $F:C([0,T]:\bR ^{d_1}) \to \bR$, 
\begin{align}\label{eqn_representation}
    -\frac{1}{\beta_n}\log E e^{-\beta_n F(X^n)} = \inf_{\{\bar \mu^n_i\}} E\left[F(\bar X^n) + \frac{1}{\beta_n} \sum_{i=n+1}^{\beta_n+n} R(\bar \mu^n_i(\cdot) \| \rho_{\bar X^n_{i-1}}(\bar Y^n_{i-1},\cdot))\right],
\end{align}
where $\{\bar{\mu}^n_i\}_{i\in\{n+1,\dots,\beta_n+n\}}$ is a collection of random probability measures satisfying the following two conditions:
\begin{enumerate}
    \item $\bar{\mu}^n_i$ is measurable with respect to the $\sigma$-algebra $\mathcal{F}^n_{i-1}$, where $\mathcal{F}^n_{n}=\{\emptyset, \Omega\}$ and for $i\in\{n+1,\dots,\beta_n+n\}$, $\mathcal{F}^n_{i}=\sigma\{\bar{Y}^n_n,\dots, \bar{Y}^n_i\}$;
    \item the conditional distribution of $\bar{Y}^n_i$, given $\mathcal{F}^n_{i-1}$, is $\bar{\mu}^n_i$. 
\end{enumerate}
Moreover, $\{\bar{X}^n_k\}_{k\geq n}$ are defined by \eqref{eqn_recursion} with $\{Y_k\}$ replaced by $\{{\bar{Y}^n_k}\}$, and $\{\bar{X}^n(t):t\in[0,T]\}$ is the continuous interpolations of $\{\bar{X}^n_k\}_{k\geq n}$.
\end{proposition}

\begin{proof}
Observe that $\{X^n(t):t\in[0,T]\}$ are determined by $\{x, X^n_{n+1},\dots, X^n_{\mathbf{m}(t_n+T)}\}$, which depends only on the state-dependent noise $\{Y_n,\dots, Y_{\mathbf{m}(t_n+T)-1}\}$ via the recursive formula. Therefore, the variational formula in  \cite[Proposition 2.3]{buddup4} combined with the chain rule for relative entropy (see Section \ref{sec:notation}),  with $\beta_n = \mathbf{m}(t_n+T)-n$ and base measure, 
\[
    \rho_{x^{\beta_n}_0}(y_0,dy_1)\rho_{x^{\beta_n}_1}(y_1, dy_2)\times\cdots\times \rho_{x^{\beta_n}_{\beta_n-1}}(y_{\beta_n-1},dy_{\beta_n}),
\]
gives the claimed result.
\end{proof}

Let us briefly outline the main ideas of the proof of 
the upper and lower bounds used to prove Theorem \ref{thm:main}. For the upper bound, for any $\varepsilon > 0$, from the representation formula we can choose a sequence of $\varepsilon$-optimal controls $\bar \mu ^n = \{\bar{\mu}^n_i\}_{n+1} ^{\beta_n+n} $; we suppress the dependence on $\varepsilon$ in the notation. This sequence in turn defines a controlled process $\bar X ^n = \{\bar{X}^n(t):t\in[0,T]\}$. To prove the upper bound, in Lemma \ref{lemma:tightUB} we show tightness of both the controls and the controlled process, and identify the limit $(\bar X, \bar \mu)$ along a convergent subsequence of $\{(\bar X ^n, \bar \mu ^n  )\}$. In particular we identify the limit ODE for $\bar X$, the limit of the controlled processes. With these results, the proof of the upper bound follows from fairly standard arguments involving Fatou's lemma, lower semi-continuity of relative entropy and the chain rule; see Section \ref{sec:upper} for the details.

The difficult part of proving Theorem \ref{thm:main} is in proving the lower bound. Whereas for the upper bound we can use the definition of the infimum in \eqref{eqn_representation} to obtain a suitable sequence of controls, for the lower bound we have to explicitly construct a sequence of nearly-optimal controls $ \bar \nu ^n = \{ \bar \nu ^n _i \} _{i=n+1} ^{\beta_n+n}$. This is carried out in Section \ref{sec:optimalControl}. The first step is to show that for any trajectory $\xi$ such that $I(\xi) < \infty$, for any $\varepsilon >0$, there is a piecewise linear $\xi ^*$ such that $|| \xi ^* - \xi || _\infty < \varepsilon$ and $I(\xi ^*) \leq I (\xi) + \varepsilon$ (see Lemma \ref{lem:piecewise_conti}). Such trajectories, along with transition kernels that are nearly-optimal for the local rate function $L$---see Lemma \ref{lem:optimal_control}---are used to construct the sequence of controls $\bar \nu ^n$ for each $n$. Moreover, in Lemma \ref{lem:tightness} we show tightness of the sequence $\{\bar \nu ^n \} _n$.

With suitable controls $\bar \nu ^n$ identified, we obtain an upper bound of the right-hand side of the representation formula \eqref{eqn_representation}. It remains to show that, asymptotically as $n \to \infty$, this upper bound is, in turn, bounded from above by $\inf _{\rho} \{ F(\rho) + I (\rho) \}$. This is achieved in Section \ref{sec:lower} through a series of approximations. An essential ingredient in the proof of the lower bound is to use the two time-scales that the controlled state process $\bar X^n$ and the controlled noise process $ \{ \bar Y ^n \}$ operate on: the controlled noise process will ultimately move at a (much) faster time scale than the state process, which helps us deal with the state-dependence in the noise. To make this rigorous, we divide $[0,T]$ into subintervals, each containing a given number of time points $t^n _j$ associated with the controlled process arising from the $\bar \nu ^n _j$s. We then use (local) ergodicity to show that, as the number of such time points in each subinterval grows, the controlled process converges, and identify the corresponding limit process (\ref{lem:limit_m}). Next, we show that as the number of intervals grows, this limit process converges to the trajectory $\xi$ of interest. In Section \ref{sec:lower} these approximations are combined to obtain the lower bound.    

\subsection{Alternative representations of the local rate function} 
\label{sec:altRep}
For each $x \in \bR^{d_1}$,  $J_x$ defined as in \eqref{eq:emprate} with $\rho(y, dz)$ replaced by $\rho_x(y, dz)$ is the rate function associated with the empirical measure of a Markov chain with transition probability $\rho_x(y, dz)$. An alternative representation of $J_x(\mu)$, due to Donsker and Varadhan \cite{donvar1}, is given by,
\begin{align}\label{eq:DVrate}
     \sup_{u > 0} \int \log\left(\frac{u(y)}{\rho_x u(y)} \right)\mu(dy), 
\end{align}
where the supremum is taken over strictly positive continuous functions $u$ and $\rho_x u(y) = \int u(z) \rho_x(y, dz)$. 

Another representation of $J_x$ is provided by Dinwood and Ney, see Lemma 3.1 in \cite{dinwoodie_ney}. For bounded Lipschitz functions $f$, let $T_{f}^x$ be the operator on the space of bounded measurable functions, equipped with the uniform metric, given by,
\[
    T_{f}^x(u)(y) = e^{f(y)} \rho_x u(y).
\]
With $r_f(x)$ the spectral radius of $T_f^x$, Dinwood and Ney show that $J_x(\mu)$ can be represented as
\begin{align} \label{eq:DNrate}
    \sup_{f} \left\{ \int f(y) \mu(dy) - r_f(x) \right\},
\end{align}
where the supremum is taken over bounded Lipschitz functions. As a result, the local rate function in \eqref{eqn_local_rate_function} can be written as, 
\begin{align*}
    L(x,\beta) = \inf_\mu \left\{J_x(\mu) : \beta = \int g(x,z)\mu(dz)\right\},
\end{align*}
where $J_x(\mu)$ is given by any of the expressions \eqref{eq:emprate}, \eqref{eq:DVrate} or \eqref{eq:DNrate}. 




\subsection{The limiting Hamiltonian}

Consider the Hamiltonian $H$ that is given as the Fenchel-Legendre transform of the local rate function $L$ in \eqref{eqn_local_rate_function}: 
\[
    H(x,\alpha) = \sup_\beta \{\alpha, \beta\rangle - L(x,\beta)\}. 
\]
We are now interested in making connections between $H$ and the type of limiting Hamiltonian assumed in previous works on large deviations for SA (see Sections \ref{sec:intro} and \ref{sec:litRev}). To this end, we have the following result, where we use the Laplace principle for the empirical measure of a Markov chain (see, e.g, Chapter 6 in \cite{buddup4}) and standard results from convex analysis to show that $H(x,\alpha)$ can be interpreted as a limiting log-moment generating function associated with the transition probability $\rho_x(y, dz)$. 
\begin{proposition}\label{prop:limitHamiltonian}
    Suppose \ref{ass:kernel}, \ref{ass:transitivity} and \ref{ass:logMGF} of Assumption \ref{sufficient_conditions} hold. 
   Take $x \in \bR^{d_1}$, let $\{Y_i\}$ be a Markov chain with transition kernel $\rho_x$ and $J_x$ be defined as in \eqref{eq:emprate}, with $\rho$ replaced by $\rho_x$. Then, 
   \begin{align}\label{eq:limitHamiltonian}
       H(x,\alpha) = \lim_n \frac{1}{n}\log E\left[\exp\left\{ \sum_{i=1}^n \langle \alpha, g(x,Y_i)\rangle \right\}\right], \qquad \alpha \in \mathbb{R}
^{d_1}.   \end{align}
\end{proposition}
\begin{proof}
     By the Feller property for $\rho$ (see Remark \ref{rem:Feller}) and \ref{ass:transitivity}, it follows that the empirical measure of $\{Y_i\}$ satisfies a Laplace principle on $\mathcal{P}(\mathbb{R}^{d_2})$ with rate function $J_x$, see \cite[Theorem 6.6]{buddup}.  For every bounded and measurable function $f$, the linear functional $\mu \mapsto \int f d\mu$, defined on $\mathcal{P}(\mathbb{R}^{d_2})$, is bounded and continuous. For each $x \in \mathbb{R}^{d_1}$, by the Laplace principle for the empirical measure of $\{Y_i\}$, the map $f \mapsto \hat H(x,f)$ given by the limit,
\begin{align*}
    \hat H(x, f) &= \lim_n \frac{1}{n}\log E\left[\exp\left\{ \sum_{i=1}^n f(Y_i)\right\}\right] \\
    &= \lim_n  \frac{1}{n}\log E\left[\exp\left\{-n \left \langle -f, \frac{1}{n}\sum_{i=1}^n \delta_{Y_i}\right\rangle\right\}\right],
\end{align*}
is well-defined on the set of bounded measurable functions.  Moreover,  $\hat H(x, f)$ may be identified as the Fenchel-Legendre transform of $J_x$,
\[
    \hat H(x, f) = \sup_{\mu} \{\langle f,\mu\rangle - J_x(\mu)\}. 
\]
By Assumption \ref{ass:logMGF}, 
the function $\hat H(x, \cdot)$ may be extended to the, possibly unbounded, function $\langle \alpha , g(x,\cdot) \rangle$. Indeed, 
the function $\langle \alpha , g(x,\cdot) \rangle$ can be approximated from below by bounded measurable functions and the dominated convergence theorem can be applied because of the upper bound, 
\begin{align*}
    &\sup_n \frac{1}{n}\log E\left[\exp\left\{ \sum_{i=1}^n \langle \alpha, g(x,Y_i) \rangle \right\}\right]\\ &\quad  = \sup_n \frac{1}{n}\log E\left[\exp\left\{ \sum_{i=1}^{n-1} \langle \alpha, g(x,Y_i) \rangle \right\} E\left[\exp\left\{  \langle \alpha, g(x,Y_n) \rangle \right\} \mid Y_{n-1}, \ldots, Y_1\right]\right] \\ & \quad 
    = \sup_n \frac{1}{n}\log E\left[\exp\left\{ \sum_{i=1}^{n-1} \langle \alpha, g(x,Y_i) \rangle \right\} \int \exp\left\{  \langle \alpha, g(x,y_n) \rangle \right\} \rho_x(Y_{n-1}, dy_n) \right] \\ & \quad 
    \leq \sup_n \frac{1}{n}\log \left(K \cdot E\left[\exp\left\{ \sum_{i=1}^{n-1} \langle \alpha, g(x,Y_i) \rangle \right\}\right] \right)  
    \\ & \quad 
    \leq K < \infty,
\end{align*}
where $K = \sup_{x \in \mathbb{R}^{d_1}} \sup_{y \in K} \log \int_K e^{\langle \alpha, g(x,z) \rangle} \rho_x(y, dz)$ is finite by \ref{ass:logMGF}.
It remains to  show that 
\[
H(x,\alpha) = \hat H(x, \langle \alpha, g(x,\cdot) \rangle),
\]
is the Fenchel-Legendre transform of $L$. We prove this by showing that $L(x,\beta) = \sup_{\alpha} \{\langle \alpha, \beta\rangle - \hat H(x,\langle \alpha, g(x,\cdot)\rangle )\}$, which is proved using a standard argument from convex analysis. Consider the set 
\[
    \Gamma_x = \{(r,s) \subset \mathbb{R} \times \mathbb{R}^{d_1}: r \geq J_x(\mu), \int g(x,y) \mu(dy) = s, \text{ some } \mu \in \mathcal{P}(\mathbb{R}^{d_2})\}. 
\]
Note that $\Gamma_x$ is convex for each $x$. By taking a normal of the form $(1, \lambda_\beta)$ to the tangent plane of $\Gamma_x$ at $(L(x,\beta), \beta)$
it follows that 
\[
    \langle (1, \lambda_\beta), (r,s) - (L(x,\beta), \beta) \rangle \geq 0, \quad (r,s) \in \Gamma_x.
\]
Moreover, the sup in 
\[
    \sup_{\alpha} \inf_{(r,s) \in \Gamma_x}\{\langle (1, -\alpha),(r,s) - (L(x,\beta), \beta))\rangle\},  
\]
is attained at $-\alpha = \lambda_\beta$. 
Therefore, we have that
\begin{align*}
    \sup_{\alpha} \inf_{(r,s) \in \Gamma_x}\{\langle (1, -\alpha),(r,s) - (L(x,\beta), \beta))\rangle\} = \inf_{(r,s) \in \Gamma}\{\langle (1, \lambda_\beta),(r,s) - (L(x, \beta), \beta))\rangle\} \geq 0. 
\end{align*}
Moreover, for all $\alpha$, 
\begin{align*}
    \inf_{(r,s) \in \Gamma_x}\{\langle (1, -\alpha),(r,s) - (L(x,\beta), \beta))\rangle\} \leq 0. 
\end{align*}
Combining the two, it holds that
\begin{align*}
    \sup_{\alpha} \inf_{(r,s) \in \Gamma_x}\{\langle (1, -\alpha),(r,s) - (L(x,\beta), \beta))\rangle\} = \inf_{(r,s) \in \Gamma_x}\{\langle (1, \lambda_\beta),(r,s) - (L(x,\beta), \beta))\rangle\} = 0, 
\end{align*}
which is equivalent to
\begin{align*}
    L(x,\beta) &= \sup_{\alpha} \inf_{(r,s) \in \Gamma_x}\{ r-\langle \alpha, s -\beta\rangle\} \\ &= \sup_{\alpha} \{\langle \alpha, \beta\rangle + \inf_{(r,s) \in \Gamma_x}\{ r - \langle \alpha, s \rangle\}\} \\
    &= \sup_{\alpha} \{\langle \alpha, \beta\rangle + \inf_{\mu}\{ J_x(\mu)  - \langle \alpha, \int g(x,y) \mu(dy) \rangle\}\} \\
    &= \sup_{\alpha} \{\langle \alpha, \beta\rangle - \sup_{\mu}\{ \int \langle \alpha, g(x,y)\rangle \mu(dy)  - J_x(\mu)\}\} \\
     &= \sup_{\alpha} \{\langle \alpha, \beta\rangle - \hat H(x,\langle \alpha, g(x,\cdot) \rangle )\}. 
\end{align*}
This completes the proof. 
\end{proof}

\begin{remark} \label{rem:limitingHamiltonian} 
    Using the representation \eqref{eq:limitHamiltonian} of the limiting Hamiltonian, it follows that the time-dependent limiting Hamiltonian, the Fenchel-Legendre transform of the time-dependent local rate function $L(t,x,\beta) = \frac{1}{h(t)}L(x,\beta)$, is given by, 
    \begin{align*}
         H(t,x,\alpha) &= 
         \sup_\beta \left\{\{\alpha, \beta\rangle - \frac{1}{h(t)}L(x,\beta)\right\} \\ &=\frac{1}{h(t)}
         \sup_\beta \left\{\left \langle \alpha h(t), \beta \right\rangle - L(x,\beta)\right\} \\
         &= \frac{1}{h(t)} H\left(x, \alpha h(t)\right).
    \end{align*}
\end{remark}

\subsection{Continuity of the local rate function}

In this section we prove that, under Assumption \ref{sufficient_conditions}, the local rate function $L$ in \eqref{eqn_local_rate_function} is continuous at every point where it is finite.   

\begin{lemma}\label{lem:conti_L}
Suppose \ref{ass:Lipschitz},\ref{ass:kernel} and \ref{ass:logMGF} hold. For any $(x_1,\beta_1) \in \bR ^{d_1} \times \bR ^{d_1}$ such that $L(x_1,\beta_1)<\infty$, $L$ is continuous at $(x_1,\beta_1)$. 
\end{lemma}

\begin{proof}
Let $H$ be the limiting Hamiltonian given by \eqref{eq:limitHamiltonian}. By Proposition \ref{prop:limitHamiltonian}, $L(x, \cdot)$ is equal to the Legendre-Fenchel transform of $H(x, \cdot)$. In \cite{iscneynum} the authors show that $\alpha \mapsto H(x,\alpha)$ is convex and smooth; see also Section 4.3 in \cite{dup2}. 
To prove the continuity of $L$ at $(x,\beta)$, by the arguments used in \cite[Lemma 4.16 (f)]{buddup4}, it suffices to show the continuity of $H(x,\alpha)$ in $(x,\alpha)$.
In turn, to prove that $(x,\alpha) \mapsto H(x,\alpha)$ is continuous, it suffices to show that the family $\{H_n\}_{n \in \mathbb{N}}$, defined by
\[
    H_n(x, \alpha) = \frac{1}{n}\log E\left[\exp\left\{ \sum_{i=1}^n \langle \alpha, g(x,Y_i)\rangle \right\}\right], \quad (x, \alpha) \in \bR^{d_1} \times \bR^{d_1}, 
\]
is equicontinuous. 

By Assumption \ref{ass:kernel}, for each $\varepsilon_1 > 0$ there exists $\delta_1 > 0$ such that $|x_1-x_2|< \delta_1$ and $|\alpha_1 - \alpha_2|< \delta_1$ implies that 
\begin{align*}
    -\varepsilon_1 \leq \Lambda(x_1, \alpha_1, y) - \Lambda(x_2, \alpha_2, y) \leq \varepsilon_1, \qquad y \in K. 
\end{align*}
By exponentiation, for each expression in the last display, and selecting $\varepsilon_1$ sufficiently small, there is, for each $\varepsilon > 0$, a $\delta > 0$ such that $|x_1-x_2|< \delta$ and $|\alpha_1 - \alpha_2|< \delta$ implies that, 
\begin{align*}
    1-\varepsilon \leq \frac{\int \exp\{\langle \alpha_1, g(x_1, z)\rangle\} \rho_{x_1}(y, dz)}{\int \exp\{\langle \alpha_2, g(x_2, z)\rangle\} \rho_{x_2}(y, dz)}\leq 1+\varepsilon, \qquad y \in K. 
\end{align*}

Repeatedly applying the inequalities in the previous display yields
\begin{align*}
    &(1-\varepsilon)^{n}\int\cdots\int e^{\langle \alpha_2, g(x_2,y_1)\rangle+\cdots+\langle \alpha_2, g(x_2,y_n)\rangle}\rho_{x_2}(y_0,dy_1)\cdots\rho_{x_2}(y_{n-1},dy_n) \\
    &\qquad\leq \int\cdots\int e^{\langle \alpha_1, g(x_1,y_1)\rangle+\cdots+\langle \alpha_1, g(x_1,y_n)\rangle}\rho_{x_1}(y_0,dy_1)\cdots\rho_{x_1}(y_{n-1},dy_n)\\
   &\qquad\leq (1+\varepsilon)^{n}\int\cdots\int e^{\langle \alpha_2, g(x_2,y_1)\rangle+\cdots+\langle \alpha_2, g(x_2,y_n)\rangle}\rho_{x_2}(y_0,dy_1)\cdots\rho_{x_2}(y_{n-1},dy_n) .
\end{align*}
Taking logarithm and scaling with $1/n$, and rearranging the inequalities, we obtain,
\[
   \log(1-\varepsilon) \leq H_n(x_1,\alpha_1)-H_n(x_2,\alpha_2)\leq \log(1+\varepsilon). 
\]
This proves that $\{H_n\}_{n \in \mathbb{N}}$ is equicontinuous and completes the proof.
\end{proof}

\section{Related work for constant and decreasing step size}
\label{sec:litRev}
The literature on large deviations for recursive algorithms of the form \eqref{eqn_recursion} is concerned with the two cases of (i) constant step size, where $\varepsilon_n = \varepsilon > 0$, does not depend on $n$, and (ii) decreasing step size, where $\varepsilon_n \to 0$ as $n \to \infty$. For constant step size, $t_n = \varepsilon n$ and $\mathbf{m}(t) = \lfloor n \varepsilon \rfloor$, and LDPs for the piecewise linearly interpolated process $X^\varepsilon(t)$ of $\{X^\varepsilon_n\}$ with interpolation time $\varepsilon$, are obtained as $\varepsilon \to 0$, akin to the small-noise results of Freidlin and Wentzell \cite{frewen2}. The rate function $I$ associated with such an LDP takes the form of an action functional: 
\begin{align*}
    I(\varphi) = \int_0^T L(\varphi(t), \dot \varphi(t)) ds, 
\end{align*}
if $\varphi$ is an absolutely continuous function, and $I(\varphi) = \infty$, otherwise, where $L$ is a local rate function.  

In the case of decreasing step size, LDPs are obtained for the process $\{X^n\}$ defined in \eqref{eqn_recursion}. In this case, with $\beta_n \doteq m(t_n+T)-n$, the limiting time scale is $h(t) = \lim_n h^n(t)$, where $h^n(t) = \beta_n \varepsilon_{n+k-1}$ for $t \in [t_{n+k-1} - t_n, t_{n+k}- t_n)$, $k \in \{1, \dots, \beta_n\}$. The associated rate function takes the form
\begin{align*}
    I(\varphi) = \int_0^T \frac{1}{h(t)} L(\varphi(t), \dot \varphi(t)) ds.  
\end{align*}

The main difference between the constant and decreasing step size settings is thus the inclusion of the limiting time scale $h(t)$ in the rate function. Note, however, that for some choices of $\varepsilon_n$, the limiting time scale $h(t)$ may be constant and equal to $1$. One such example is $\varepsilon_n = (n+1)^{-\alpha}$, for $\alpha \in (0,1)$.  

In the existing literature, the LDP is obtained by identifying a Hamiltonian $H(x,\alpha)$, that sometimes can be interpreted as a limiting log-moment generating function, and defining the local rate function $L(x,\beta)$ as the convex conjugate of $H(x,\alpha)$. An issue with this approach, and in applying it to different SA algorithms, is that the Hamiltonian is defined as a limit and its relation to the underlying dynamics, such as the transition kernel $\rho_x(y, dz)$, can only be established in some special cases.

\subsection{Large deviations for constant step size}
\label{sec:const_step}
The large deviations theory for stochastic approximation with constant step size originates from the work of Freidlin \cite{fre76, fre78}. Therein, focus is on dynamical systems in continuous time of the form,
\begin{align}\label{eq:dyn}
    \dot x^\varepsilon(t) = b(x^\varepsilon(t), \xi(t/\varepsilon)), \qquad x^\varepsilon(0) = x, 
\end{align}
over a finite time interval $[0,T]$.  The function $b$ is assumed bounded, with bounded derivatives, $\{\xi(t), t \geq 0\}$ is bounded and $\varepsilon \to 0$. It is also assumed that there is a limiting Hamiltonian $H(x,\alpha)$ such that for arbitrary step functions $\varphi$ and $\alpha$ from  $[0,T]$ to $\mathbb{R}^{d_1}$, the following limit exists:
\begin{align}\label{eq:FreH}
    \lim_{\varepsilon \to 0} \varepsilon \log E\left[\exp\left\{ \frac{1}{\varepsilon} \int_0^T \langle \alpha(t), b(\varphi(t), \xi(t/\varepsilon)) \rangle dt \right\} \right] = \int_0^T H(\varphi(t), \alpha(t)) dt. 
\end{align}
With $L$ as the convex conjugate of $H$, an LDP is proved for $\{x^\varepsilon\}$, as $\varepsilon \to 0$, on,
\[
C_{0,T}^x = \{\varphi \in C([0,T];\mathbb{R}^{d_1}), \varphi(0) = x\},
\]
with rate function given by,
\begin{align*}
    I(\varphi) = \int_0^T L(\varphi(t), \dot \varphi(t)) ds, 
\end{align*}
if $\varphi$ is absolutely continuous and $I(\varphi) = \infty$, otherwise. Moreover, when $\{\xi(t)\, t \geq 0\}$ is a finite state Markov chain, Freidlin identifies the limiting Hamiltonian as the largest eigenvalue of a tilted intensity matrix. 


In \cite{iscneynum}, Iscoe, Ney and Nummelin generalize the results of Freidlin by considering LDPs for Markov-additive processes in both continuous time and discrete time. In the discrete time setting, which relates more closely to the results of this paper, they consider a process of the form $(Y_n, X_n)$, where, 
\begin{align*}
    &P\left((Y_n, X_n-X_{n-1}) \in A \times \Gamma \,|\, (Y_{n-1}, X_{n-1}) = (y,x)\right) \\ & \quad = P((Y_n, X_n-X_{n-1}) \in A \times \Gamma \,|\, Y_{n-1} = y). 
\end{align*}
This corresponds to recursions of the form \eqref{eqn_recursion} where $g(x,y) = g(y)$ does not depend on $x$. They assume that there exists a probability measure $\nu$ on $E \times \mathbb{R}^{d_1}$, an integer $m_0$, and real numbers $0 < a \leq b < \infty$, such that, 
\begin{align*}
    a \nu(A \times \Gamma) \leq P^{m_0}( A \times \Gamma) \leq b \nu(A \times \Gamma),
\end{align*}
for all $x \in E$, sets $A$ in the relevant $\sigma$-algebra on the state space, $\Gamma \in\mathbb{R}^{d_1}$. With $\hat P(\alpha) = \hat P(y, A;\alpha) =  \int \exp\{\langle \alpha, x \rangle\} P(y, A \times dx)$ they derive an LDP, and more detailed asymtotics, for $P^n(y, A \times nF) = P((Y_n, X_n-X_0) \in A \times nF \,|\, Y_0 = y)$. In particular, it follows from Lemma 3.1 (ii) in \cite{iscneynum} that, 
\begin{align*}
     \lim_{n \to \infty} \frac{1}{n} \log \hat P^n(y, A; \alpha)  = \log \lambda(\alpha), \quad \alpha \in \mathcal{D}, 
\end{align*}
where $\lambda(\alpha)$ is the principal eigenvalue of $\hat P(\alpha)$. 

Dupuis then further develops the large deviations results for discrete systems of the form \eqref{eqn_recursion} with constant step size in \cite{dup2}, using milder conditions on the limiting Hamiltonian. 
More specifically, one of the results of the paper (see Section 4.3 of \cite{dup2}), is an LDP for the model \eqref{eqn_recursion} with $\varepsilon_n = \varepsilon$ and $g(x,y)$ bounded and uniformly (in $y$) Lipschitz-continuous in x, and measurable in $y$. Additional assumptions used therein, which can be viewed as stronger versions of \ref{ass:kernel} and \ref{ass:bounded_kernel}, are: the process $\{Y_n \}$ is sampled from a transition kernel $\rho_{X_n}(y, \cdot)$ with density $\eta_{X_n}(y,\cdot)$ with respect to a common reference measure $\lambda(dz)$ such that, for a given compact set $F_1$, 
\begin{itemize}
    \item[i)] there are $0< a \leq A < \infty$ such that for all $x \in F_1$, and all $y,z$, $a \leq \eta_x(y,z) \leq A$, and
    \item[ii)] $\eta_x(y,z)$ is Lipschitz continuous in $x$, uniformly in $y,z$, for $x \in F_1$.
\end{itemize}

\subsection{Large deviations for decreasing step size}
\label{sec:decr_step}
As described in Section \ref{sec:intro}, in the case of SA with decreasing step size, the theory of large deviations is not as well-developed. The first results are obtained by Kushner \cite{kus16}, who considers the recursion \eqref{eqn_recursion}, with $\varepsilon_n = (n+1)^{-\rho}$, $\rho \in (0,1]$, and $g(x,y)) = b(x) + y$, with $b(\cdot)$ Lipschitz continuous and the noise sequence $\{Y_n\}$ a sequence of iid centered Gaussian variables. The discrete time and time-changed analogue of \eqref{eq:FreH}, the assumption of a limiting Hamiltonian, is given by,
\begin{align*}
    & \lim_{n \to \infty} \lambda_n \log E\left[\exp\left\{ \sum_{i=0}^{N-1} \left\langle \alpha(i\Delta), \sum_{j=\mathbf{m}(t_n+i\Delta)}^{\mathbf{m}(t_n + (i+1)\Delta)-1} \varepsilon_j (b(x) + Y_j)/\lambda_n  \right\rangle\right\} \right] \\ & \quad = \int_0^T H(t,x,\alpha(t)) dt, 
\end{align*}
where $T = N\Delta$, $\Delta > 0$, and $\alpha$ is constant on intervals $[i\Delta, (i+1)\Delta)$. In \cite{kus16}, Kushner identifies the appropriate normalising sequence,
\[
\lambda_n = \sum_{j=n}^{m(t_n+T)} \varepsilon_j^2, 
\]
which can be shown to be asymptotically proportional to $1/\beta_n$ with $\beta_n$ as in Assumption \ref{ass:limith}.  He goes on to prove an LDP with rate function,
\begin{align*}
    I(\varphi) = \int_0^T L(t, \varphi(t), \dot \varphi(t)) ds, 
\end{align*}
where the local rate function $L(t, \varphi(t), \dot \varphi(t))$ is the convex conjugate of $H$.

In the follow-up work \cite{dupkus2}, Dupuis and Kushner develop the theory further by considering recursions of the form \eqref{eqn_recursion} with step-sizes satisfying $\varepsilon_n \geq 0,$ $\sum_n \varepsilon_n = \infty$, $\varepsilon_n \to 0$, and $g(x,y) = \bar{b}(x) + b(x,y)$, with $E[b(x,Y_n )] = 0$. It is further assumed that $Y_n = (\tilde Y_n, \hat Y_n)$, where $\{\tilde Y_n\}$ and $\{\hat Y_n\}$ are mutually independent, $\{\tilde Y_n\}$ is stationary and bounded, and $\{\hat Y_n\}$ is a stationary centered Gaussian process with summable correlation function. Moreover, $b(x,Y_n) = b_1(x,\tilde Y_n ) + b_0(x)\hat Y_n$, where $b_1(\cdot, \tilde y)$, $b_0$ and $\bar b$ are uniformly (in $\tilde y, x$) Lipschitz-continuous and bounded. Note that, in contrast to our setting, in \cite{dupkus2} it is not assumed that the distribution of the noise $Y_{n+1}$ may depend on the state $X_n$. 
It is assumed that there exists a continuous function $h_1$ such that, 
\[
\lim_{\delta \to 0} \lim_{n\to \infty} \frac{\varepsilon_{m_n(t+\delta)}}{\varepsilon_n} = h_1(t).
\]
The results in \cite{dupkus2} also rely on the assumptions of existence of a limiting Hamiltonian: i.e., that there is a continuous function $H(t, x, \alpha)$ with $\alpha \mapsto H(t,x,\alpha)$ continuously differentiable for each $t,x$ such that the following limit exists,
\begin{align*}
    &\lim_{\delta \to 0}\lim_{n \to \infty} \lambda_n \log E\left[\exp\left\{ \sum_{i=0}^{T/\delta-1} \left\langle \alpha(i\delta) \varepsilon_{m_n(i\delta)}, \sum_{j=m(i\delta)}^{m((i+1)\delta)-1} b(\psi(i\delta),Y_j)  \right\rangle\right\} \right] \\ & \quad = \int_0^T H(t,x,\alpha(t)) dt.
\end{align*}
A particular example studied in \cite{dupkus2} is when $\{Y_n, -\infty < n < \infty\}$ is bounded and stationary and there is a continuous $\hat H_0(\cdot, \cdot)$ with $\alpha \mapsto \hat H_0(\alpha, x)$ continuously differentiable for each $x$ such that 
\begin{align*}
    \lim_{N \to \infty} \frac{1}{N} \log E\left[e^{ \left\langle \alpha, \sum_{j=0}^{N-1} b(\psi, Y_j)   \right\rangle} \right]
     &= \lim_{N \to \infty} \frac{1}{N} \log E_0 \left[e^{\left\langle \alpha, \sum_{j=0}^{N-1} b(\psi, Y_j)   \right\rangle} \right] \\
     &= \hat H_0(\alpha, \psi),
\end{align*}
where $E_0$ denotes expectation conditional on $\{Y_j, j < 0\}$ and the convergence is uniform in the conditioning data. 
With the limiting Hamiltonian established in Proposition \ref{prop:limitHamiltonian}, we provide an analogous representation in the setting where the distribution of the noise may be state-dependent.

\section{Applications}
\label{sec:appl}
In this section we present applications to learning algorithms in statistics, machine learning and statistical physics that can be stated as stochastic approximations satisfying Assumption \ref{sufficient_conditions}. 

\subsection{Stochastic gradients}
Consider minimizing a function $G(x) = \sum_{m =1}^M G_m(x)$, by stochastic gradient descent (SGD). Let us assume that $x \mapsto \nabla G_m(x)$ is bounded and Lipschitz continuous for all $m \in \{1, \dots, M\}$. Consider a standard SGD algorithm; in the $k$th iteration an index $Y_{k+1}$ is selected uniformly at random on $\{1, \dots, M\}$ and updated according to,
\begin{align*}
    X_{n+1} = X_n - \varepsilon_{n+1} \nabla G_{Y_{n+1}}(X_n). 
\end{align*}
Consequently, $\{X_n\}$ satisfies the stochastic approximation  \eqref{eqn_recursion} where $\{Y_k\}$ is an iid sequence,  $\rho_x(y, \cdot) = \rho(\cdot)$ is the uniform distribution on the integers $\{1,\dots, M\}$, and  $g(x,m) = - \nabla G_m (x)$. Assumption \ref{sufficient_conditions} is automatically satisfied by the assumptions on $\nabla G_m$, $\lambda$ as the counting measure and since $\rho_x(y, m)$ does not depend on $x,y$. 

By Theorem \ref{thm:main} the continuous interpolations of $\{X^n_k\}$ given by \eqref{eqn_recursion} satisfies a Laplace principle with rate function $I$ given by \eqref{eqn_rate_function} where the local rate function $L$ is given by \eqref{eqn_local_rate_function}. Since, $\rho_x(y, m) = \rho(m) = 1/M$ does not depend on $x,y$ the local rate function simplifies to,
\begin{align*}
L(x,\beta) &= \inf_\mu \left \{ R(\mu \| \rho): \beta = - \sum_{m=1}^M \nabla G_m(x) \mu(m) \right\} \\
&= \sup_\alpha \{ \langle \alpha, \beta\rangle - \bar{H}(x,\alpha)\},  
\end{align*}
where,
\begin{align*}
    \bar{H}(x,\alpha) &= 
    \log \left( \frac{1}{M} \sum_{m=1}^M \exp\left\{
    - \langle \alpha, \nabla G_m(x) \rangle \right\}\right). 
\end{align*}
A concrete example arises in maximum likelihood estimation of a logistic regression with data $\{(\xi_m, \upsilon_m)\}_{m=1}^M$ where $\xi_m$ are explanatory variables and $\upsilon_m$ labels in $\{-1,1\}$ and $\phi$ represents a feature function. Then the negative log-likelihood to be minimized is given by,
\[
G(x) = \sum_{m=1}^M - \log \mathrm{sigm}\left(\upsilon_m x^T\phi(\xi_m) \right),
\] 
where $\mathrm{sigm}(t) = (1+e^{-t})^{-1}$ is the sigmoid function and 
\[
    \nabla G_m(x) = \upsilon_m \phi(\xi_m)\left(1-\mathrm{sigm}(\upsilon_m x^T\phi(\xi_m)\right), 
\]
which is bounded and Lipschitz continuous in $x$, for all $m \in \{1, \dots, M\}$. 

More general stochastic gradients appear in the minimization of functions of the form $\bar G(x) = \int G(x,y) \gamma(dy)$ for some distribution $\gamma$. With $\{Y_n\}$ iid with distribution $\gamma$ the algorithm,
\begin{align*}
    X_{n+1} = X_n - \varepsilon_{n+1} \nabla_x G(X_n, Y_{n+1}), 
\end{align*}
can be used to minimize $\bar G$. If $\nabla_x G(x,y)$ is bounded and Lipschitz continuous in $x$, then Assumption \ref{sufficient_conditions} is satisfied and the Laplace principle holds with local rate function,
\[
L(x,\beta) = \inf_\mu \left \{ R(\mu \| \gamma): \beta = - \int \nabla_x G(x,y) \mu(dy) \right\} = \sup_\alpha \{ \langle \alpha, \beta\rangle - \bar{H}(x,\alpha)\},  
\]
where,
\begin{align*}
    \bar{H}(x,\alpha) &= 
    \log \left( \int \exp\left\{
    - \langle \alpha, \nabla_x G(x,y) \rangle \right\}\gamma(dy) \right). 
\end{align*}

\subsection{Persistent contrastive divergence} 

Consider a parametrized probability density of the form,
\[
    p(v, h | x) = \exp\left\{ -E(v, h; x) + F(x)\right\}, 
\]
where $x$ denotes the parameters; $v$ represents observed (visible) variables, $h$ represents unobserved (hidden) variables, $E$ is referred to as the energy and $F$ as the free energy, 
\[
    F(x) = - \log \int \exp\{-E(v,h; x) \} \lambda(dv,dh). 
\]

In order to estimate the parameters, given independent observations $v^{(1)}, \dots, v^{(M)}$ from $p(v | x) = \int p(v,h|x) \lambda(dh)$, one can use a maximum likelihood approach. This amounts to minimizing the 
negative log-likelihood, which is proportional to,
\[
    - \log L(x) = - \frac{1}{m}\sum_{m=1}^M \log p(v^{(m)} | x). 
\]
To employ a gradient descent algorithm for this minimization task would require knowledge of the (negative) gradient $- \nabla _x \log L(x)$, which here takes the form,
\begin{align*}
    -\nabla_x \log L(x) &= - \frac{1}{m}\sum_{m=1}^M \frac{\nabla_x \int \exp\left\{-E(v^{(m)},h; x) + F(x) \right\} \lambda(dh)}{p(v^{(m)} | x)} \nonumber \\ 
    &= \frac{1}{m}\sum_{m=1}^M \frac{\int \left(\nabla_x E(v^{(m)},h; x) - \nabla_x F(x) \right)p(v^{(m)},h | x) \lambda(dh)}{p(v^{(m)} | x)} \nonumber \\ 
    &= \frac{1}{m}\sum_{m=1}^M \left[\int \nabla_x E(v^{(m)},h; x) p(h | v^{(m)},x) \lambda(dh) - \nabla_x F(x)  \right] \nonumber \\ 
     &= \frac{1}{m}\sum_{m=1}^M \left[\int \nabla_x E(v^{(m)},h; x) p(h | v^{(m)},x) \lambda(dh) \right. \\
     & \qquad \qquad \qquad \left. - \int \nabla_x E(v,h; x) p(v,h |x) \lambda(dv,dh) \right]. 
\end{align*}
Thus is, it would require being able to compute terms of the form, for $m=1, \dots, M$,
\begin{align}
\label{eq:RBM_pos_neg}
\nabla_x E(v^{(m)},h; x) p(h | v^{(m)},x) \lambda(dh) - \int \nabla_x E(v,h; x) p(v,h |x) \lambda(dv,dh),     
\end{align}
which may be intractable. Examples \ref{ex:RBM}-\ref{ex:expFamily} illustrate how in some cases simplifying model assumptions can assist in computing the first term in \eqref{eq:RBM_pos_neg} explicitly. However, for the general case we can write,
\[
    p(h | v, x) = \exp\left\{-E(v,h; x) + F_H(v,x)\right\}, 
\]
where,
\[ F_H(v,x) = - \log \int  \exp\left\{-E(v,h; x) \right\} \lambda(dh). 
\]
This can help us approximate the gradient in \eqref{eq:RBM_pos_neg} in the following way: First, we can construct Markov kernels, $\rho^{(m,1)}_x(y^{(1)}, dz^{(1)})$ and  $\rho^{(2)}_x(y^{(2)}, dz^{(2)})$, where $y^{(1)} = h$, $y^{(2)} = (v,h)$ and $\rho^{(m,1)}_x(y^{(1)}, dz^{(1)})$ has invariant distribution $p(h | v^{(m)},x)$ and $\rho^{(2)}_x(y^{(2)}, dz^{(2)})$ has invariant distribution $p(v,h|x)$. Second, we sample $Y_{n+1} = (Y_{n+1}^{(1)}, Y_{n+1}^{(2)})$ by drawing an index $m$ at random and drawing $Y_{n+1}^{(1)}$ from $\rho^{(m,1)}_{X_n}(Y_n^{(1)}, dz^{(1)})$ and $Y_{n+1}^{(2)}$ from $\rho^{(2)}_{X_n}(Y_n^{(2)}, dz^{(2)})$ independently of each other and updating,
\[
X_{n+1} = X_n - \varepsilon_{n+1} \left(\nabla_x E(v^{(m)}, Y_{n+1}^{(1)}; X_n) - \nabla_x E( Y_{n+1}^{(2)}; X_n) \right). 
\]
This can be identified as the stochastic recursion \eqref{eqn_recursion} with,
\[
    \rho_x(y, dz) = \frac{1}{m}\sum_{m=1}^M \rho^{(m,1)}_x(y^{(1)}, dz^{(1)}) \rho_x^{(2)}(y^{(2)}, dz^{(2)}),
\]
and, 
\[
g(x,y) = \nabla_x E(v^{(m)},y^{(1)}; x)  - \nabla_x E(y^{(2)}; x). 
\]

\begin{example}
\label{ex:RBM}
In Restricted Boltzmann Machines (RBMs) $v$ and $h$ are binary with $x = (W,b_V, b_H)$, where $W$ is a matrix, $b_V, b_H$ are vectors, and 
\[
E(v,h; W,b_V,b_H) = -v^T Wh - v^Tb_V - h^T b_H. 
\]
This form for $E(v,h;W,b_V, b_H)$ implies that the components of $h$ are conditionally independent given $v$ with success probability  $p(h_j = 1 | v,x) = \mathrm{sigm}(v^TW e_j + e_j^Tb_H)$. The first term in \eqref{eq:RBM_pos_neg} therefore reduces to,
\begin{align*}
   &\frac{1}{m}\sum_{m=1}^M \sum_h \left[\nabla_{W_{ij}} E(v^{(m)},h; W,b_V, b_H) \right] p(h | v^{(m)},W,b_V,b_H) \\ & \quad = \frac{1}{m}
   \sum_{m=1}^M \sum_h -v^{(m)}_i h_j \mathrm{sigm}(v^TWe_j + e_j^Tb_H)^{h_j}\mathrm{sigm}(-(v^TWe_j + e_j^Tb_H))^{1-h_j}  \\ & \quad = 
   - \frac{1}{m}\sum_{m=1}^M v^{(m)}_i \mathrm{sigm}(v^TWe_j + e_j^Tb_H). 
\end{align*}
The second term is given by the expectation $E[V_i H_j]$ under the joint distribution $p(v,h | x)$. Let $\rho_x((v_0,h_0), (v_1,h_1))$ denote a stochastic kernel with $p(v,h | x)$ as its invariant distribution and approximate $E[V_i H_j]$ by its expectation under $\rho_x((v_0,h_0), \cdot)$, where $\rho_x((v_0,h_0), \cdot)$ may be taken as the block-Gibbs sampler,
\begin{align*}
    &\rho_x((v_0,h_0), (v_1,h_1)) = p(h_1|v_0,x) p(v_1|h_1,x) \\ &\quad = \prod_{j=1}^{d_H} \mathrm{sigm}(v_0^TWe_j + e_j^Tb_H)^{h_{1j}}\mathrm{sigm}(-(v_0^TWe_j + e_j^Tb_H))^{1-h_{1j}} \\ & \qquad \times \prod_{i=1}^{d_V} \mathrm{sigm}(e_i^TWh_1 + e_i^Tb_V)^{v_{1i}}\mathrm{sigm}(-(e_i^TWh_1 + e_1^Tb_V))^{1-v_{1i}}.  
\end{align*}

The persistent contrastive divergence algorithm for estimating the parameters is then given by \eqref{eqn_recursion}, where $Y_{n+1} = (v_{n+1},h_{n+1})$ is now sampled from $\rho_{X_n}(Y_n, \cdot)$ and 
\[
g(x,y) = \frac{1}{m}\sum_{m=1}^M \left[\int \nabla_x E(v^{(m)},h; x) p(h | v^{(m)},x) \lambda(dh) - \nabla_x E(y; x)  \right].
\]
Because $\mathrm{sigm}$ is bounded and continuous and the state space $\{0,1\}^{d_V} \times \{0,1\}^{d_H}$ is finite, 
it is straightforward to verify Assumption \ref{sufficient_conditions}. 
\end{example}

\begin{example} 
\label{ex:expFamily}
Consider an exponential family with $E(v,h;x) = E(v;x) = x^T\phi(v) - \log c(v) $, that does not depend on hidden variables and is linear in the parameters $x$. Then, $\nabla_x E(v^{(m)}; x) = \phi(v^{(m)})$, whereas for the free energy we have $\nabla_x F(x) = E[\phi(V)]$, where the expectation is taken under $p(v | x)$ and may be intractable. Thus, $g(x,y)$ becomes,
\begin{align*}
   g(x,y) =  \frac{1}{m} \sum_{m=1}^M \phi(v^{(m)}) - \int \phi(v) p(v | x) \lambda(dv). 
\end{align*}
\end{example}

\subsection{The Wang-Landau Algorithm}

The Wang-Landau algorithm for general state spaces includes many popular multicanonical Monte Carlo methods, such as simulated tempering. Let $\{(\mathcal{Y}_i, \mathcal{B}_i, \lambda_i)\}_{i=1}^d$ be measure spaces with $\lambda_i$ being $\sigma$-finite for each $i$. Let $\mathcal{Y} = \cup_{i=1}^d \mathcal{Y}_i \times \{i\}$, be the union space equipped with the $\sigma$-field $\mathcal{B}$ generated by the sets $\{(A_i, i): i \in \{1,\dots, d\}, A_i \in \mathcal{B}_i\}$ and define the measure $\lambda$ on $\mathcal{B}$ by  $\lambda(A,i) = \lambda_i(A) I\{A \in \mathcal{B}_i\}$. Given non-negative integrable functions $f_i$, $i=1\dots,d$, let $x(i) = \int_{\mathcal{Y}_i} f_i(y) \lambda_i(dy)/Z$, where $Z = \sum_{i=1}^d \int_{\mathcal{Y}_i} f_i(y) \lambda_i(dy)$. Assuming that $x(i) > 0$ for each $i =1, \dots, d,$ the aim is to sample from $\pi$ on $\mathcal{B}$ given by
\[
    \pi(dy,i) \propto \frac{f_i(y)}{x(i)} I\{y \in \mathcal{Y}_i\} \lambda_i(dy),
\]
and to estimate the normalizing constants  $x(i)$. Let $\rho_x((y,i), (dz,j))$ be a Markov kernel with invariant density $\pi$. The original algorithm considers the case where $\pi$ is uniform in $i$, whereas the general case considered here is due to \cite{yvesLiu}. The basic for of the Wang-Landau algorithm initiates $(Y_0, I_0) \in \mathcal{Y}$, $\phi_0 \in (0,\infty)^d$ and $x_0 = \phi_0/\sum_{i=1}^d \phi_0(i)$. At each $k \geq 0$, given $(Y_k,I_k)$, $\phi_k$ and $x_k$, sample $(Y_{k+1}, I_{k+1})$ from $\rho_{x_k}((Y_k, I_k), \cdot)$ and update,
\begin{align*}
    \phi_{k+1}(i) &= \phi_k(i)(1+\varepsilon_k I\{I_{k+1} = i\}),\quad i=1, \dots, d,\\
    x_{k+1}(i) &= \frac{\phi_{k+1}(i)}{\sum_{j=1}^d \phi_{k+1}(j)}.  
\end{align*}

The Wang-Landau algorithm is a stochastic approximation with update function $g$ given by,
\[
g(\phi,(z,j)) = \phi + \phi(j) e_j,
\] 
where $e_j$ is the unit-vector in the $j$th coordinate.

\begin{example}[Multicanonical Monte Carlo]
\label{ex:MulticanMC}
Let $\Sigma$ be a finite state space, e.g.\ $\{-1,1\}^N$, and $E: \Sigma \to \mathbb{R}$ an energy function and consider the Gibbs distribution with density $\bar \pi$ proportional to $\exp\{-E(\sigma)\}$. A collection of energy levels $-\infty \leq E_0 < \dots  < E_d \leq \infty$ induces a partition $\mathcal{Y}_i = \{\sigma \in \Sigma: E_{i-1} < E(\sigma) \leq E_i\}$. With $x(i) = \bar \pi(\mathcal{Y}_i)$, $f_i(y) = E(y)$ and $\lambda_i = \lambda$, $x(i)$ can be estimated from samples from the measure $\pi(dy,i)$, obtained using the stochastic approximation scheme described above.    
\end{example}

\begin{example}[Estimation of free energy differences]
\label{ex:FreeEnergy}
Let $\Sigma$ be a finite state space, e.g.\ $\{-1,1\}^N$, some $N < \infty$, $E: \Sigma \times \Omega \to \mathbb{R}$ an energy function parametrized by a finite set $\Omega$ (for example temperatures), and consider the Gibbs distribution with density $p_{\Sigma, \Omega}$ proportional to $\exp\{-E(\sigma, \omega)\}$. The conditional density of the state given the parameter $\omega$ is given by 
\[
p_{\Sigma \mid \Omega}(\sigma | \omega) = \exp\{-E(\sigma, \omega) + F(\omega)\},
\] 
where $F(\omega) = -\log \sum_{\sigma} \exp\{-E(\sigma, \omega)\}$ is the free energy associated with $E$. Consider the problem of estimating free energy differences. That is, fix $\omega_1 \in \Omega$ and, for $\omega \in \Omega$, consider estimating $F(\omega) - F(\omega_1)$.  To put this in the Wang-Ladau framework, we can enumerate $\Omega = \{\omega_i\}_{i=1}^{d}$, take $\mathcal{Y}_i = \Sigma \times \Omega$, and $\lambda_i$ as the counting measure on $\mathcal{Y}$. With $f_i(\sigma, \omega) = \exp\{-E(\sigma, \omega_i)\} I\{\omega = \omega_i\}$, it follows that 
\begin{align*}
    -\log(x(i)/x(1)) = \log \sum_{\sigma} \exp\{-E(\sigma, \omega_1)\} - \log \sum_{\sigma} \exp\{-E(\sigma, \omega_i)\} = F(\omega_i) - F(\omega_1). 
\end{align*}  
Since, $-\log(x(i)/x(1))$ may be estimated by $-\log(\phi_{k}(i)/\phi_k(1))$, where $\phi_k$ is generated by the Wang-Landau algorithm, the free energy differences may be estimated accordingly. 

\end{example}

\section{Laplace upper bound}
\label{sec:upper}
In this section we take the first step 
towards proving Theorem \ref{thm:main}, by proving the Laplace principle upper bound, stated in Theorem \ref{thm:upper}.
\begin{theorem}
\label{thm:upper}
Assume \ref{ass:Lipschitz}-\ref{ass:limith}. With $I$ defined as in \eqref{eqn_rate_function}, 
for any bounded, continuous function $F:C([0,T]:\bR ^{d_1}) \to \bR$,
\begin{align}
\label{eq:upper_bound}
    \liminf _{n \to\infty} - \frac{1}{\beta_n} \log E \left[ e^{- \beta_n F(X ^n)} \right] \geq \inf _{\varphi} \left( F(\varphi) + I (\varphi) \right),
\end{align}
where the infimum is over $\varphi \in AC_{x_0}([0,T]:\bR ^d)$.
\end{theorem}

Recall the representation formula \eqref{eqn_representation},
\begin{align*}
    -\frac{1}{\beta_n}\log E e^{-\beta_n F(X^n)} = \inf_{\{\bar \mu^n_i\}} E\left[F(\bar X^n) + \frac{1}{\beta_n} \sum_{i=n+1}^{\beta_n+n} R(\bar \mu^n_i(\cdot) \| \rho_{\bar X^n_{i-1}}(\bar Y^n_{i-1},\cdot))\right],
\end{align*}
where the full details of the random probability measures $\{\bar{\mu}^n_i\}_{i\in\{n+1,\dots,\beta_n+n\}}$ being minimized over are given in Proposition \ref{prop_representation}. From this formula, for fixed $n$ and arbitrary (fixed) $\varepsilon > 0$, it is possible to choose a sequence of controls $\{ \bar \mu ^n \}$ such that
\begin{align}
\label{eq:upper_RelEnt}
    -\frac{1}{\beta _n} \log E\left[ e^{- \beta_n F(X^n)} \right] + \varepsilon \geq E \left[ F(\bar X^n) +\frac{1}{\beta_n} \sum _{i=n+1} ^{\beta _n + n} R(\bar{\mu} _i ^n (\cdot) || \rho_{\bar{X} ^n _ {i-1}}(\bar Y _{i-1} ^n, \cdot  )) \right].
\end{align}
We augment the controls to also keep track of the time dependence of the $\bar \mu ^n _i$s: for a Borel set $A$ and $t \in [t_n, t_n +T]$, define $\bar \mu ^n (A|t)$ by
\[
    \bar \mu ^n (A |t) = \bar \mu ^n _i (A), \ \ \textrm{for } i \textrm{ such that } t \in [\tau _{i} ^n, \tau _{i+1} ^n), 
\]
where $\tau _i ^n = t_{n+i} -t_n$. The controlled measures $\bar \mu ^n$ can now be defined as
\[
    \bar \mu ^n (A \times C) = \int _C \frac{1}{h^n (t)} \bar \mu ^n (A|t) dt,
\]
where
\[
    h^n (t) = \beta_n \varepsilon _{n+i-1},
\]
with $i \in \{n+1, \dots, \beta_n + n \}$ such that $t \in [\tau_i ^n, \tau _{i+1} ^n)$. We also define a collection of sequences of measures, involving the controlled process $\bar X ^n$, the controlled noise $\bar Y ^n$ and the noise distribution $\rho$, that will play a role in the convergence analysis of the controlled process $\bar X ^n$ and the corresponding controls $\bar \mu ^n$: for $A, B \subset \bR ^{d_2} ,C \subset [0,T]$ Borel sets,
\begin{align*}
    \lambda ^n (A \times B \times C) &= \int _C \frac{1}{h^n(t)} \lambda ^n (A \times B|t)dt, \ \ \lambda ^n (A \times B |t) = \delta _{\bar Y ^n _{i-1}} (A)   \bar \mu ^ n _i (B), \\
    \gamma ^n (A \times B \times C) &= \int _C \frac{1}{h^n(t)} \gamma ^n (A \times B |t) dt,  \ \ \gamma ^n (A \times B |t) = \delta _{\bar Y ^n _{i-1}} (A) \rho_{\bar X ^n _{i-1}} (\bar Y ^n _{i-1}, B ). 
\end{align*}
In each definition, $i$ is such that $t \in [\tau ^n _i, \tau ^n _{i+1})$. From the definitions of $\bar \mu ^n$ and $\lambda ^n$, we have that $\bar \mu ^n (A \times C) = \lambda ^n (\bR ^{d_2} \times A \times C)$. The following lemma establishes the necessary tightness and characterises the limits of subsequences of the sequences of measures defined above.
\begin{lemma}
\label{lemma:tightUB}
    Assume \ref{ass:Lipschitz}-\ref{ass:limith} hold. Then $\{ \bar X ^n\}$, $\{\bar \mu ^n \}$, $\{ \bar \lambda ^n \}$ and $\{ \gamma ^n \}$ are tight sequences, and for every subsequence of $\{ \bar X^n, \bar \mu ^n \}$ there exists a further subsequence that converges to $(\bar X, \bar \mu)$, with the limit satisfying the following relations:
    \begin{align}
        \bar \mu (A \times C) = \int _C \frac{1}{h(t)} \bar \mu (A |t) dt, \label{eq:UBbarmu}\\
        \bar X (t) = x + \int _0 ^t \int _{\bR ^d} g(\bar X (s), y) \bar \mu (dy|s)ds. \label{eq:UBbarX}
    \end{align}
Furthermore, any limit point $\lambda$ and $\gamma$ of a convergent subsequence of $\{ \lambda ^n\}$ and $\{ \gamma ^n \}$, respectively, will have the following properties,
\begin{align*}
    \lambda (A \times B \times C) &= \int _C \frac{1}{h(t)} \lambda (A \times B|t)dt, \\
    \gamma (A \times B \times C) &= \int _C \frac{1}{h(t)} \left( \int _A \rho _{\bar X (t)}  (x,B) \bar \mu (dx |t) \right)dt,
\end{align*}
for some stochastic kernel $\lambda(dy \times dz |t)$, and 
\begin{align*}
    \lambda (A \times \bR ^{d_2} \times C) = \lambda (\bR ^{d_2} \times A \times C) = \bar \mu (A \times C) = \int _C \frac{1}{h(t)} \bar \mu (A |t) dt.
\end{align*}
\end{lemma}
Before giving the proof of Lemma \ref{lemma:tightUB}, we show how the result allows us to prove the upper bound \eqref{eq:upper_bound}.

\begin{proof}[Proof of Theorem \ref{thm:upper}]
As a first step, we use the chain rule to decompose the relative-entropy term on the right-hand side of \eqref{eq:upper_RelEnt},
\begin{align}
    R (\bar \mu ^n _i (\cdot) || \rho _{\bar X ^n _{i-1}} (\bar Y ^n _{i-1} , \cdot  )) &= R(\delta_{\bar Y ^n _{i-1}} (\cdot) || \delta _{\bar Y ^n _{i-1}} (\cdot)) + R (\bar \mu ^n _i (\cdot) || \rho_{\bar X ^n _{i-1}} (\bar Y  ^n _{i-1} , \cdot)) \notag \\
    & = R(\delta_{\bar Y ^n _{i-1}} (dy) \bar \mu ^n _i (dz) || \delta _{\bar Y ^n _{i-1}} (dy) \rho_{\bar X ^n _{i-1}} (y, dz )) \notag \\
    & = R (\lambda ^n (dy \times dz |t) || \gamma ^n (dy \times dz |t)). \label{eq:relEntDecomp}
\end{align}
By tightness, we can pick a subsequence, also labelled by $n$ for notational convenience, along which all the measures involved converge. Along this subsequence, because of how we chose the sequence $\{ \bar \mu ^n \}$, we also have the following lower bound:
\begin{align}
    \liminf _{n \to \infty} -\frac{1}{\beta _n} &E \left[ e^{- \beta _n F(X ^n)} \right] + \varepsilon \geq \liminf _{n \to \infty} E\left[ F(\bar X^n) +  \frac{1}{\beta _n} \sum _{i=n} ^{\beta _n + n -1 } R(\bar{\mu} _i ^n (\cdot) || \rho_{\bar{X} ^n _ i}(\bar Y _i ^n, \cdot  )) \right] \notag \\
    &= \liminf _{n \to \infty} E \left[ F(\bar X ^n) + R (\lambda ^n (dy \times dz \times dt) || \gamma ^n (dy \times dz \times dt)) \right] \label{eq:upperDecompEqual} \\
    &\geq E \left[ F(\bar X) + R(\lambda (dx \times dy \times dt) || \gamma (dx \times dy \times dt)) \right]. \label{eq:upperDecompIneq}
\end{align}
In the first step in the last display, the equality \eqref{eq:upperDecompEqual}, we use the decomposition \eqref{eq:relEntDecomp} combined with the definition of $h^n$ and the fact that the measures $\lambda ^n(\cdot |t), \gamma ^n (\cdot | t)$ are constant over the intervals $[\tau ^n _i, \tau ^n _{i+1})$. In the second step, the inequality \eqref{eq:upperDecompIneq}, we combine Lemma \ref{lemma:tightUB} with Fatou's lemma and the lower semi-continuity of relative entropy (see, e.g., \cite{dupell4, buddup4}). Next, we use the chain rule once more combined with the structure of the measures $\lambda$ and $\gamma$, 
\begin{align*}
    & E \left[ F(\bar X) + R(\lambda (dy \times dz \times dt) || \gamma (dy \times dz \times dt)) \right] \\
    & \quad = E \left[ F(\bar X) + \int _0 ^T \frac{1}{h(t)} R(\lambda (dy \times dz |t) || \bar \mu(dy |t) \rho_{\bar X (t)} (y, dz |t)) dt \right].
\end{align*}
The relative-entropy term on the right-hand side can be bounded from below by the local rate function $L$ in \eqref{eqn_local_rate_function}:
\begin{align*}
   & E \left[ F(\bar X) + \int _0 ^T \frac{1}{h(t)} R(\lambda (dy \times dz |t) || \bar \mu(dy |t) \rho_{\bar X (t)} (y, dz |t)) dt \right] \\
   &\quad \geq E \left[ F(\bar X) + \int _0 ^T \frac{1}{h(t)} L(\bar X (t), \dot{\bar X} (t))dt \right] \\
    & \quad \geq \inf _{\varphi} \{ F(\varphi) + \int _0 ^T \frac{1}{h(t)} L(\varphi (t), \dot \varphi (t))dt \},
\end{align*}
where the infimum is over $\varphi \in AC_{x_0} ([0,T]:\bR^{d_1})$. The integral on the right-hand side is precisely how the rate function $I$ was defined in Theorem \ref{thm:main}. Combining the sequence of inequalities therefore leads to the desired bound,
\begin{align*}
    \liminf _{n \to \infty} -\frac{1}{\beta _n} E \left[ e^{-\beta _n F(\bar X ^n)}  \right] + \varepsilon \geq \inf _{\varphi } \left\{ F(\varphi) + I(\varphi) \right\}.
\end{align*}
Since $\varepsilon$ was chosen arbitrarily, this shows the upper bound \eqref{eq:upper_bound} for the specific subsequence used. A standard argument by contradiction extends the upper bound to hold for the full sequence, thereby proving that the Laplace principle upper bound follows from Lemma \ref{lemma:tightUB}. 
\end{proof}

\smallskip
\begin{proof}[Proof of Lemma \ref{lemma:tightUB}]
Because we can always choose the controls $\{ \bar \mu ^n _i \}$ such that the expectation of the sum of the relative entropy terms, appearing in \eqref{eqn_representation}, is bounded, tightness of $\{ \bar X ^n \}$ and $\{ \bar \mu ^n \}$ follows from Theorem \ref{thm:limit}, which also gives the characterisation of the limit points as in \eqref{eq:UBbarX}-\eqref{eq:UBbarmu}. From the definition of the controlled process, tightness of $\{ \bar \mu ^n \}_n$ implies tightness of $\{ \delta _{\bar Y _i ^n} \} _{i=n} ^{\beta _n + n}$, as a sequence in $n$. This in turn gives tightness of $\{\lambda ^n\}$. The tightness of $\{ \gamma ^n \}$ is obtained by the tightness of $\{ \bar X ^n \}$ and $\{ \delta _{\bar Y _i ^n} \} _{i=n} ^{\beta _n + n}$ together with the uniform continuity of $\rho_x(y,dz)$.

To characterise limit points $\lambda$ of subsequences of $\{ \lambda ^n\}$, by  Lemma 3.3.1 in \cite{dupell4} and the uniform convergence of $h^n$ we have the decomposition $\lambda (dy \times dz \times dt ) = (h(t))^{-1} \lambda (dy \times dz |t)dt $, for some stochastic kernel $\lambda (dy \times dz |t)$. Moreover, note that $\lambda ^n (\bR ^{d_2} \times A \times C) = \bar \mu ^n (A \times C)$ implies that $\lambda (\bR ^{d_2} \times A \times C) = \bar \mu (A\times C)$. For the marginal obtained when integrating out the second variable, we use arguments similar to those used in proving Lemma 6.12 in \cite{buddup4}. Take $\{ f_m \}$ as a countable collection of bounded continuous functions that is also a separating class on $\bR ^{d_2}$. We will prove that, for any $\varepsilon>0$ and all $t\in [0,T]$, as $n \to \infty$,
\begin{align}
\label{eq:convLambda}
    P \left( \left|\left| \int _0 ^t \int \frac{1}{h^n(s)} f_m(y) \bar \mu ^n (dy|s)ds - \int _0 ^t \int \frac{1}{h^n(s)} f_m(y) \lambda ^n (dy \times \bR ^{d_2} |s)ds \right| \right| > \varepsilon \right) \to 0.
\end{align}
Suppose this limit holds. Because the collection of sets of the form $[0,t]$, for $t \in [0,T]$, is a separating class of $[0,T]$, \eqref{eq:convLambda} combined with Fatou's lemma ensures that w.p.\ 1 the limit of $\lambda ^n$ will satisfy $\lambda (A \times \bR ^{d_2} \times C) = \bar \mu (A \times C)$. 

To prove \eqref{eq:convLambda}, define $K_m = \norm{f_m} _\infty$. Suppose $n$ is such that $\beta _n > 4K_m/\varepsilon$--since $\beta _n \to \infty$ as $n \to \infty$, this is possible. Using the definitions                 of $\bar \mu ^n$ and $\lambda ^n$, and an application of Markov's inequality we have, 
\begin{align*}      
    & P \left( \left|\left| \int _0 ^t \int \frac{1}{h^n(s)} f_m(y) \bar \mu ^n (dy|s)ds - \int _0 ^t \int \frac{1}{h^n(s)} f_m(y) \lambda ^n (dy \times \bR ^{d_2} |s)ds \right| \right| > \varepsilon \right) \\
    & \quad = P \left( \left|\left| \frac{1}{\beta _n} \sum _{i=n +1} ^{\mathbf{m}(t_n +t)} \int f_m (y) \bar \mu ^n _i (dy) - \frac{1}{\beta _n} \sum _{i=n} ^{\mathbf{m} (t_n + t) -1} f_m (\bar Y ^n _i) \right| \right| > \varepsilon \right)  \\
    &\quad \leq P \left( \left|\left| \frac{1}{\beta _n} \sum _{i=n +1} ^{\mathbf{m}(t_n +t)} \int f_m (y) \bar \mu ^n _i (dy) - \frac{1}{\beta _n} \sum _{i=n+1} ^{\mathbf{m} (t_n + t)} f_m (\bar Y ^n _i) \right| \right| > \frac{\varepsilon}{2} \right) \\
    &\quad \quad \quad + P \left(\left| \left| \frac{1}{\beta _n} \left( f_m (\bar Y ^n _{\mathbf{m} (t_n +t)}) - f_m (\bar Y ^n _{n}) \right) \right| \right| > \frac{\varepsilon}{2} \right) \\
    &\quad \leq P \left( \left|\left| \frac{1}{\beta _n} \sum _{i=n+1} ^{\mathbf{m}(t_n +t)} \left( \int f_m (y) \bar \mu ^n _i (dy) -  f_m (\bar Y ^n _i) \right) \right| \right| > \frac{\varepsilon}{2} \right) \\
    & \quad \leq \frac{4}{\varepsilon ^2} E \left[ \frac{1}{\beta _n ^2} \sum_{i,j = n+1} ^{\mathbf{m}(t_n +t)} \Delta ^n _{m,i} \Delta ^n _{m, j}\right],
\end{align*}
where we have defined,
    \[
        \Delta ^n _{m,i} = \int f_m (y) \bar \mu ^n _i (dy) - f_m (\bar Y ^n _i).
    \]
The term  $P \left(\left| \left| \frac{1}{\beta _n} \left( f_m (\bar Y ^n _{\mathbf{m} (t_n +t)}) - f_m (\bar Y ^n _{n}) \right) \right| \right| > \frac{\varepsilon}{2} \right) = 0$, since, 
\[ 
    \frac{1}{\beta _n} \left( f_m (\bar Y ^n _{\mathbf{m} (t_n +t)}) - f_m (\bar Y ^n _{n}) \right)< \frac{\varepsilon}{4K_m}2\norm{f_m} = \frac{\varepsilon}{2}.
\]
The sequence $\{ \Delta ^n _{m,i} \} $ is a martingale difference sequence with respect to the filtration $\calF ^n _i = \sigma \left( (\bar X ^n _j, \bar Y ^n _j): \ j<i \right)$. Therefore, the off-diagonal terms in the sum have expected value 0: for $i > j$,
\begin{align*}
    E \left[ \Delta ^n _{m,i} \Delta ^n _{m, j} \right] = E\left[ E\left[ \Delta ^n _{m,i} \Delta ^{n} _{m,j} | \calF ^n _{i-1}\right] \right] = E\left[ E\left[ \Delta ^n _{m,i}  | \calF ^n _{i-1} \right] \Delta ^n _{m, j}\right] = 0.
\end{align*}
Combined with the previous inequalities this leads to the upper bound
\begin{align*}
    & P \left( \left|\left| \int _0 ^t \int \frac{1}{h^n(s)} f_m(y) \bar \mu ^n (dy|s)ds - \int _0 ^t \int \frac{1}{h^n(s)} f_m(y) \lambda ^n (dy \times \bR ^{d_2} |s)ds \right| \right| > \varepsilon \right) \\
    & \quad \leq \frac{4}{\varepsilon ^2} E \left[ \frac{1}{\beta _n ^2} \sum_{i = n+1} ^{\mathbf{m}(t_n +t) } \left( \Delta ^n _{m,i} \right) ^2 \right] \\
    & \quad \leq \frac{4}{\varepsilon ^2} E \left[ \frac{1}{\beta _n ^2} \sum_{i = n+1} ^{\beta _n +n } \left( 2 K_m \right) ^2 \right] \\
    & \quad \leq \frac{16 K_m ^2}{\varepsilon ^2 \beta _n}.
\end{align*}
We can make this arbitrarily small by choosing $n$ large enough, which proves \eqref{eq:convLambda}.

In order to show the claimed form for $\gamma$ we use a strategy similar to the one used for $\lambda$. Take $\{ f_m \}$ to now be a countable separating class on $\bR ^{d_2} \times \bR ^{d_2}$ of bounded continuous functions. We define a sequence of measures $\{ \eta ^n \}$ by
\begin{align*}
    \eta ^n (A \times B \times C) &= \int _C \frac{1}{h(t)} \eta ^n (A \times B |t) dt, \ \ \eta ^n (A \times B |t) = \int _A \rho_{\bar X ^n _{i-1}} (y, B ) \bar \mu ^n _{i-1} (dy).
\end{align*}
From the convergence of $\bar \mu ^n$ and the continuity of $\rho$, $\eta ^n$ converges to $\gamma$. To finish the proof we therefore show that $\gamma ^n$ must have the same limit as $\eta ^n$, by proving that, for arbitrary $\varepsilon >0$,
\begin{align*}
    &P \left( \left|\left| \int _0 ^t \!\!\! \int\!\!\! \int\!\!\!  \frac{1}{h^n(s)} f_m(y,z) \eta ^n (dy \times dz|s)ds - \int _0 ^t \!\!\!\int\!\!\! \int\!\!\! \frac{1}{h^n(s)} f_m(y,z) \gamma ^n (dy \times dz |s)ds \right| \right| > \varepsilon \right) \\ & \quad  \to 0.
\end{align*}
Similar to the above, take $K_m = \norm{f_m}_\infty$. Then,
\begin{align*}
   & P \left( \left|\left| \int _0^t\!\!\! \int\!\!\! \int\!\!\!  \frac{1}{h^n(s)} f_m(y,z) \eta ^n (dy \times dz|s)ds - \int _0^t\!\!\! \int\!\!\! \int\!\!\! \frac{1}{h^n(s)} f_m(y,z) \gamma ^n (dy \times dz |s)ds \right| \right| > \varepsilon \right) \\
   & \quad = \!P\! \Bigg( \Bigg\| \frac{1}{\beta _n}\!\!\! \sum _{i=n+1}^{\mathbf{m}(t_n +t)} \!\!\!\int \!\!\!\int f_m(y,z) \rho _{\bar X ^n _{i-1}} (y, dz) \bar \mu ^n _{i-1} (dy)\! \\ & \qquad \qquad  - \!\frac{1}{\beta_n} \!\!\!\sum _{i=n+1} ^{\mathbf{m} (t_n + t) }\!\! \int f_m (\bar Y ^n _{i-1}, z) \rho _{\bar X ^n _{i-1} } (\bar Y ^n _{i-1} , dz) \Bigg\| > \varepsilon \Bigg) 
   \\
   & \quad = P \Bigg( \Bigg\| \frac{1}{\beta _n} \sum _{i=n+1} ^{\mathbf{m}(t_n +t)} \left( \int\!\!\! \int f_m(y,z) \rho _{\bar X ^n _{i-1}} (y, dz) \bar \mu ^n _{i-1} (dy) \right. \\ & \qquad \qquad - \left. \int f_m (\bar Y ^n _{i-1}, z) \rho _{\bar X ^n _{i-1} } (\bar Y ^n _{i-1} , dz) \right) \Bigg\| > \varepsilon \Bigg) \\
   & \quad \leq \frac{1}{\varepsilon ^2} E \left[ \frac{1}{\beta _n ^2} \sum _{i,j=n} ^{\mathbf{m} (t_n +t)-1} \tilde \Delta ^n _{m,i} \tilde \Delta ^n _{m,j} \right],
\end{align*}
where we have defined
\[
    \tilde \Delta ^n _{m,i} = \int \int f_m(y,z) \rho _{\bar X ^n _{i}} (y, dz) \bar \mu ^n _i (dy)- \int f_m (\bar Y ^n _{i}, z) \rho _{\bar X ^n _{i} } (\bar Y ^n _{i} , dz).
\]
Similar to the convergence analysis for $\lambda ^n$, $\{ \tilde \Delta ^{n} _ {m,i}\}$ forms a martingale difference sequence with respect to the filtration $\calF ^n _j$. The off-diagonal terms thus disappear from the sum,
\[
    E \left[ \frac{1}{\beta _n ^2} \sum _{i,j=n} ^{\mathbf{m} (t_n +t) - 1} \tilde \Delta ^n _{m,i} \tilde \Delta ^n _{m,j} \right] = E \left[ \frac{1}{\beta _n ^2} \sum _{i=n} ^{\mathbf{m} (t_n +t) - 1} \left( \tilde \Delta ^n _{m,i} \right)^2  \right], 
\]
and we obtain the upper bound,
\begin{align*}
    & P \left( \left|\left| \int _0^t \!\!\!\int\!\!\! \int\!\!\!  \frac{1}{h^n(s)} f_m(y,z) \bar \eta ^n (dy \times dz|s)ds - \int _0^t\!\!\! \int\!\!\! \int\!\!\! \frac{1}{h^n(s)} f_m(y,z) \gamma ^n (dy \times dz |s)ds \right| \right| > \varepsilon \right) \\
    & \quad \leq \frac{1}{\varepsilon ^2} E \left[ \frac{1}{\beta _n ^2} \sum _{i=n} ^{\mathbf{m} (t_n +t) - 1} \left( \tilde \Delta ^n _{m,i} \right)^2  \right] \\
    & \quad \leq \frac{1}{\varepsilon ^2} E \left[ \frac{1}{\beta _n ^2} \sum _{i=n} ^{\beta _n +n -1} \left( 2K_m \right)^2  \right] \\
    & \quad = \frac{4 K_m ^2}{\varepsilon ^2 \beta _n}.
\end{align*}
We can choose $n$ large enough to make the expression in the last display arbitrarily small. Since $\varepsilon$ was taken arbitrarily, this proves the claimed convergence. Having already established that $\eta^n \to \gamma$, we conclude that $\gamma ^n \to \gamma$.
\end{proof}

\section{Laplace lower bound}
\label{sec:lower}
In this section we prove the Laplace principle lower bound.
\begin{theorem}
\label{thm:lower}
Assume \ref{ass:Lipschitz}-\ref{ass:limith}. With $I$ defined as in \eqref{eqn_rate_function}, 
for any bounded, continuous function $F:C([0,T]:\bR^{d_1}) \to \bR$,
\begin{align}
\label{eq:lower_bound}
     \limsup _{n \to\infty} - \frac{1}{\beta_n} \log E \left[ e^{- \beta_n F(X ^n)} \right] \leq \inf _{\varphi} \left( F(\varphi) + I (\varphi) \right),
\end{align}
where the infimum is over $\varphi \in AC_{x_0} ([0,T]: \bR ^{d_1})$.
\end{theorem}
Together with the upper bound of Theorem \ref{thm:upper}, this proves the Laplace principle stated in Theorem \ref{thm:main}. The proof of the upper bound, given in Section \ref{sec:upper}, is aided by the fact that, by definition of the infimum, we can choose a sequence of nearly-optimal controls (see \eqref{eq:upper_RelEnt}). Proving the lower bound \eqref{eq:lower_bound} is considerably more involved because such a ($\varepsilon$-optimal) sequence is no longer readily available and we must instead explicitly construct a nearly-optimal controls. 

The proof of Theorem \ref{thm:lower} is split into two main parts: the first part is the construction of nearly-optimal controls and proving that the constructed sequence is tight, in Section \ref{sec:optimalControl}, whereas the second part consists of proving convergence of the controls and the associated controlled processes, in Section \ref{sec:conv}. In Section \ref{sec:proofLower} we combine these results to complete the proof of Theorem \ref{thm:lower}
.
\subsection{Construction and tightness of nearly-optimal controls}
\label{sec:optimalControl}
In this section we construct, for each $n$, a sequence of nearly-optimal controls to be used in proving the lower bound. Recall from Section \ref{sec:main} that the local rate function $L$, defined in \eqref{eqn_local_rate_function} is continuous at every point where it is finite (Lemma \ref{lem:conti_L}). 

As a first step, we show that for any $(x, \beta)$ such that $L(x, \beta) < \infty$, there exists nearly-optimal transition kernels with respect to the infimum in the definition of $L(x,\beta)$, see Lemma \ref{lem:optimal_control}. Next, in Lemma \ref{lem:piecewise_conti} we show that for any function $\zeta \in C([0,1]:\bR ^d)$ such that $I(\zeta)<\infty$, for any $\varepsilon >0 $ we can find a piecewise linear function, with a finite number of pieces, that is $\varepsilon$-close to $\zeta$ both in sup-norm and in evaluating $I$. From these two results we can construct a sequence of nearly-optimal controls $\bar \nu ^n$. In Lemma \ref{lem:tightness} we show that the sequence $\{ \bar \nu ^n \}_n$ is tight. 


Recall that, for each $x \in \mathbb{R}^{d_1}$, $\pi_x$ is the unique invariant measure of $\rho_x$. Our first result, concerning nearly-optimal transition kernels, is a direct consequence of the definition of $L$ and results in \cite{buddup4}.

\begin{lemma}\label{lem:optimal_control} Suppose \ref{ass:kernel}, \ref{ass:bounded_kernel} and \ref{ass:transitivity} hold. 
For any $(x,\beta) \in \mathbb{R} ^{d_1} \times \mathbb{R} ^{d_1}$ such that $L(x,\beta)<\infty$ and $\varepsilon>0$, there exists a probability measure $\nu^{x,\beta}(dy)$ such that,
\[
    \inf_{\gamma\in\mathcal{A}(\nu^{x,\beta})} R(\gamma\|\nu^{x,\beta}\otimes\rho_x(\cdot,\cdot) )\leq L(x,\beta)+\varepsilon\mbox{ and }\beta = \int g(x,y)\nu^{x,\beta}(dy).
\]
For any $\delta>0$, define a probability measure $\mu^{x,\beta,\delta}\doteq (1-\delta/2)\nu^{x,\beta} + (\delta/2)\pi_x$. There exists a transition kernel $q^{x,\beta,\delta}(y,dz)$ such that $\mu^{x,\beta,\delta}$ is the unique invariant measure of $q^{x,\beta, \delta}(y,dz)$ and the associated Markov chain is ergodic. In addition, 
\[
     R(\mu^{x,\beta,\delta} \otimes q^{x,\beta,\delta}(\cdot,\cdot)\|\mu^{x,\beta,\delta} \otimes \rho_x(\cdot,\cdot))\leq \inf_{\gamma\in\mathcal{A}[\nu^{x,\beta}]} R(\gamma\|\nu^{x,\beta}\otimes \rho_x)\leq L(x,\beta)+\varepsilon.
\]
Moreover, the selection of $\nu^{x,\beta}$ can be made measurable. 
\end{lemma}
\begin{proof}
For the existence part, we note that the existence of $\nu^{x,\beta}$ follows from the definition of $L(x,\beta)$ in terms of an infimum. The existence of $\mu^{x,\beta,\delta}$ and $q^{x,\beta,\delta}$ then follows from Lemma 6.17 in \cite{buddup4}, which relies on assumptions \ref{ass:kernel} and \ref{ass:transitivity}.

We now prove that the selection of $\nu^{x,\beta}$ can be made measurable. To this end we appeal to a measurable selection theorem, such as Theorem 1 in \cite{brown}, which says that there exists a Borel measurable selection of $E \subset U \times V$, where $U$ and $V$ are complete separable metric spaces, if $E$ is a Borel set and for each $u \in U$ the section $E_u = \{v: (u,v) \in E\}$ is $\sigma$-compact. With $U = \mathbb{R}^{d_1} \times \mathbb{R}^{d_1}$, $V = \mathcal{P}(\mathbb{R}^{d_1})$ and the identification $u = (x,\beta)$ and $v = \mu$ the measurable selection of $\nu^{x,\beta}$ follows if
$E = \{(x,\beta, \mu) \in \mathbb{R}^{d_1} \times \mathbb{R}^{d_1} \times \mathcal{P}(\mathbb{R}^{d_1}): J(x,\mu) \leq L(x, \beta) + \varepsilon, \int g(x,y) \mu(dy) = \beta\}$ is Borel measurable and for each $(x,\beta) \in \mathbb{R}^{d_1} \times \mathbb{R}^{d_1}$ the section $E_{(x,\beta)} = \{\mu \in \mathcal{P}(\mathbb{R}^{d_1}): (x,\beta, \mu) \in E\}$ is $\sigma$-compact. Here we denote, 
\[
    J(x,\mu) = \inf_{q \in \mathcal{A}(\mu)} \int R(q(y,\cdot) \\ \rho_x(y,\cdot)) \mu(dy). 
\]
It holds that, for each $x \in \mathbb{R}^{d_1}$,  $\mu \mapsto J(x,\mu)$ is lower semi-continuous and has compact sub-level sets; see, e.g., \cite{buddup4} or \cite{dupell2}. In fact, under Assumption \ref{ass:kernel} and \ref{ass:bounded_kernel} it is lower semi-continuous as a function of $(x,\mu)$. To prove this, we show that for each $a >0$, the set $\{(x,\mu): J(x,\mu) > a\}$ is open. Take $(x,\mu)$ such that $J(x,\mu) > a$, $\varepsilon \in (0, (J(x,\mu) - a)/2)$, and $(x',\mu')$ such that $|x-x'|<\delta$ and $d_w(\mu, \mu') < \delta$, where $d_w$ metrizes weak convergence. By continuity of $x \mapsto \eta_x(y,z)$ and Assumption \ref{ass:bounded_kernel}, we can choose $\delta$ sufficiently small that, 
\[  
    \log \frac{\eta_x(y,z)}{\eta_{x'}(y,z)} < \varepsilon, \quad \text{for all } (y,z).
\]
Consequently,  
\begin{align*}
   J(x, \mu) &= \inf_{q \in \mathcal{A}(\mu)} \int R(q(y,\cdot) \| \rho_{x}(y,\cdot)) \mu(dy) \\ 
    & = \inf_{q \in \mathcal{A}(\mu)} \left[\int R(q(y,\cdot) \| \rho_{x'}(y,\cdot)) \mu(dy)
    + \int \int q(y,z) \log \frac{\eta_x(y,z)}{\eta_{x'}(y,z)} \lambda(dz)\mu(dy)\right] \\
    &\leq J(x',\mu) + \varepsilon.
\end{align*}
By lower semi-continuity of $\mu \mapsto J(x',\mu)$, it follows that for $\delta$ sufficiently small $J(x',\mu') < J(x',\mu) + \varepsilon$ and we conclude that 
$J(x,\mu) < J(x',\mu') + 2\varepsilon$. By the choice of $\varepsilon$, $J(x',\mu') > a$. This completes the proof of lower semi-continuity of $J(x,\mu)$. Since a lower semi-continuous function is measurable, $L$ is continuous and $x \mapsto g(x,y)$ is continuous, it follows that the set $E$ is measurable. For each $(x,\beta)$ the section $E_{x,\beta}$ is a subset of the sub-level set $\{\mu: J(x,\mu) \leq L(x,\beta) + \varepsilon\}$, and hence $\sigma$-compact. 

\end{proof}

In proving the lower bound, Theorem \ref{thm:lower}, we may assume $\inf_{\varphi}\{F(\varphi)+I(\varphi)\}<\infty$, as otherwise the bound is trivially true. By the definition of the infimum, for any $\varepsilon>0$, there is $\zeta\in C([0,1]:\mathbb{R}^{d_1})$ such that,
\[
    F(\zeta)+I(\zeta)\leq \inf_{\varphi}\{F(\varphi)+I(\varphi)\} + \varepsilon.
\]
Recall that $F$ is bounded and $I$ is of the form,
\[
    I(\zeta) = \int_0^T \frac{1}{h(t)}L(\zeta(t),\dot{\zeta}(t))dt.
\]
We can therefore assume that $L(\zeta(t),\dot{\zeta}(t))<\infty$ for all $t\in[0,T]$. Moreover, the following lemma states that we can focus on $\zeta$ that are piecewise linear with finitely many pieces. 

\begin{lemma}\label{lem:piecewise_conti} Assume \ref{ass:kernel}, \ref{ass:transitivity} and \ref{ass:logMGF}. 
For every $\zeta\in C([0,1];\mathbb{R}^{d_1})$ satisfying $I(\zeta)<\infty$, and every $\varepsilon>0$, there exists a $\zeta^*\in C([0,1]:\mathbb{R}^{d_1})$ that is piecewise linear with finitely many pieces, such that $\|\zeta^*-\zeta\|_{\infty}<\varepsilon$ and, 
\[
    I(\zeta^*)=\int_0^T \frac{1}{h(t)}L(\zeta^*(t),\dot{\zeta}^*(t))dt \leq \int_0^T \frac{1}{h(t)}L(\zeta(t),\dot{\zeta}(t))dt+\varepsilon =I(\zeta)+\varepsilon. 
\]
\end{lemma}
\begin{proof}
The proof relies on parts of several different results from \cite{buddup4}. First, since $(x,\beta) \mapsto L(x,\beta)$ is continuous by Lemma \ref{lem:conti_L}, it suffices---see the argument used for Part (e) of Lemma 4.21 in \cite{buddup4}---to show that, for the given $\varepsilon > 0$, there is a $\zeta ^* _1 \in C([0,T]:\bR ^{d_1})$ such that $\{ \dot \zeta ^* _1 (t) : t \in [0,T] \}$ is bounded, $\norm{\zeta - \zeta ^*_1} _{\infty} < \varepsilon$, and
\[
    I(\zeta^*)\leq I(\zeta)+\varepsilon.
\]
The existence of such an $\zeta_1^*$ is the topic of Lemma 4.17 in \cite{buddup4}. The same arguments as used in the proof of that result applies also in the setting considered here, if we can show that $L$ is uniformly superlinear in $\beta$, see Section \ref{sec:notation}. Recall that $H$ is the Lengendre-Fenchel transform of $L$. The uniform superlinearity of $L$ then holds if, 
\begin{align}
\label{eq:boundH}
    \sup_{x\in\mathbb{R}^{d_1}}\sup_{\alpha\in\mathbb{R}^{d_1}:\|\alpha\|=M} H(x,\alpha)<\infty,
\end{align}
for every $M < \infty$; see \cite[Lemma 4.14(c)]{buddup4} for why this bound ensures the superlinearity of $L$. Combining these arguments, to prove the existence of $\zeta ^* _1$ with the properties described above, it is enough to prove \eqref{eq:boundH}. 

To show \eqref{eq:boundH}, we recall the alternative representation from Proposition \ref{prop:limitHamiltonian},
\[
    H(x,\alpha) \doteq \lim_{n\to\infty} \frac{1}{n}\log\left(\int\cdots\int e^{\langle \alpha, g(x,y_1)\rangle+\cdots+\langle \alpha, g(x,y_n)\rangle}\rho_x(y_0,dy_1)\cdots\rho_x(y_{n-1},dy_n)\right).
\]
Moreover, Assumption \ref{ass:logMGF} ensures that, for every $\alpha \in \bR ^{d_1}$,
\[
    \hat{C}_\alpha = \sup_x\sup_y\left( \log\int_{\mathbb{R}^{d_2}}e^{\langle \alpha,g(x,z)\rangle}\rho_x(y,dz) \right)<\infty.
\]
Combining the two, we have that, for every $\alpha \in \bR ^{d_1}$,
\[
H(x,\alpha)\leq \hat{C}_\alpha<\infty.
\] 
In addition, for every $(x,y) \in \mathbb{R}^{d_1} \times \mathbb{R}^{d_2}$, the function, 
\[
    \alpha \mapsto \log\int_{\mathbb{R}^{d_2}}e^{\langle \alpha,g(x,z)\rangle}\rho_x(y,dz), 
\]
is convex. Because the supremum of a collection of convex functions is also convex, it holds that $\alpha \mapsto \hat{C}_\alpha$ is a convex function, with finite values for all $\alpha\in\mathbb{R}^{d_1}$. Therefore, $\hat{C}_\alpha$ is continuous in $\alpha$, due to it being convex and finite-valued for any $\alpha$, and we have, 
\[
    \sup_{x\in\mathbb{R}^{d_1}}\sup_{\alpha\in\mathbb{R}^{d_1}:\|\alpha\|=M} H(x,\alpha)\leq \sup_{\alpha\in\mathbb{R}^{d_1}:\|\alpha\|=M}\hat{C}_\alpha<\infty,
\]
for every $M<\infty$. 

This shows \eqref{eq:boundH}, which ensures the uniform superlinearity of $L$, and in turn the existence of an $\zeta ^* _1 \in C([0,T]:\bR ^{d_1})$ such that $\{ \dot \zeta ^* _1 (t) : t\in [0,T] \}$ is bounded, $\| \zeta - \zeta ^*_1 \| < \varepsilon$, and $(\zeta _1 ^*) \leq I(\zeta) + \varepsilon$. Using the continuity of $L$, a function $\zeta ^*$ with the claimed properties can then be obtained as a piecewise linear approximation of $\zeta ^* _1$.
\end{proof}

With Lemmas \ref{lem:optimal_control} and \ref{lem:piecewise_conti} established, in addition to the continuity of $L$, see Lemma \ref{lem:conti_L}, we are now ready to construct the (nearly-optimal) controls that will play a central role in the proof of the lower bound Theorem \ref{thm:lower}. A crucial part of the construction of the controls is to divide the interval $\{n,n+1,\dots, n+\beta_n\}$ into $\ell+1$ segments. Let $\ell,m\in\mathbb{N}$, where for any $\ell \leq \beta _n$, $m$ is the largest integer such that $\ell m\leq \beta_n$. Note that $m = m_n$ depends on $n$ and increases proportionally to $\beta_n$. The idea is that, for fixed $\ell$, we can freeze the state dependence of the noise sequence and use a (local) ergodicity argument to establish convergence. To carry out this strategy, we define $\tau^\ell_k, k = 0,1,\dots,\ell+1$ as,
\[
\tau^\ell_k = \lim_{n\to\infty}\sum_{i=n}^{ n+km\wedge \beta_n }\varepsilon_i,
\]
the limiting times associated to the $\ell$ intervals. From the definition it follows that $\tau^\ell_0 = 0$ and $\tau^\ell_{\ell+1} = T$.

The controls will be defined in terms of the transition probabilities obtained in Lemma \ref{lem:optimal_control}. Set $\bar{X}^n_n=x_0$, $\bar{Y}^n_n= y_0$, and recall that $\zeta(0)=x_0$. Given $\delta>0$, in the first interval, that is, for $j=n,\dots,n+m$, we define $\hat \nu ^n _j$ and $\bar{Y}^n_j$ as follows. The controls are 
\[
    \hat{\nu}^n_j(dz)=\begin{cases}
    &\rho_{\zeta^*(0)}(\bar{Y}^n_{j-1},dz)\quad n \leq j< n+l_0,\\
    &q^{\zeta^*(0),\dot{\zeta}^*(0),\delta}(\bar{Y}^n_{j-1},dz)\quad n+l_0 \leq j \leq n+m-1,
    \end{cases}
\]
where $l_0$ is the constant appearing in the transitivity condition \ref{ass:transitivity}, whereas  $\bar Y^n_j$ is sampled from $\hat \nu ^n _j$. More precisely, $\hat \nu ^n _j$ is the conditional distribution of $\bar Y ^n _j$ given $\calF ^n _{j-1}$. These controlled measures are such that for the first $l_0$ variables $\bar Y ^n _{n}, \dots, \bar Y ^n _{n+l_0 -1}$, the conditional distribution is the same as the noise distribution with fixed $x$-argument, and for the remaining variables, $\bar Y ^n _{n+l_0}, \dots, \bar Y ^n _{n +m}$, the conditional distribution is the transition kernel of Lemma \ref{lem:optimal_control} associated with the triplet $(\zeta (0), \dot \zeta (0), \delta)$. 

For the following intervals, that is, for each $1 \leq k \leq \ell$, for $j=n+km+1,\dots,n+km+m$, we define, 
\[
    \hat{\nu}^n_j(dz)= \begin{cases}
    &\rho_{\zeta^*(\tau^\ell_k)}(\bar{Y}^n_{j-1},dz)\quad n+km+1 \leq j< n+km+l_0+1,\\
    &q^{\zeta^*(\tau^\ell_k),\dot{\zeta}^*(\tau^\ell_k),\delta}(\bar{Y}^n_{j-1},dz),\quad  n+km+l_0+1 \leq j \leq n+(k+1)m,
    \end{cases} 
\]
where $\bar Y ^n _{j}$ is sampled from $\hat \nu ^n _j$. For the final interval, for $j = n+\ell m +1, \dots , n+\beta_n$, we set
\[
\hat{\nu}^n_j(dz) = \rho_{\bar{X}^n_{j-1}}(\bar{Y}^n_{j-1},dz),
\]
where $\bar Y ^n _{j}$ is, again, sampled from $\hat \nu ^n _j$.

Having defined the controlled noise variables, $\bar Y^n_j$, $j=n, \dots, n+\beta_n$, the  controlled process $\bar X ^n$ is defined as, 
\[
    \bar{X}^n_j = \bar{X}^n_{j-1} + \varepsilon_j g(\bar{X}^n_{j-1},\bar{Y}^n_j), \ \ j=n, \dots, n +\beta_n.
\]
To make sure that the controlled process $\bar {X}^n_i$ is not too far away from the path $\zeta^*(t)$, we define the stopping index $\hat{i}^n$ as, 
\[
\hat{i}^n = \inf\left\{ i\leq \beta_n + n: \|\bar X^n_i - \zeta^*(t^n_i)\|>1\right\}\wedge (\beta_n+n),
\]
and the stopping time $\hat{S}^n $ as,
\[
\hat S^n = \sum_{i = n}^{\hat{i}^n}\varepsilon_i.
\]
Observe that since $\bar{X}^n_n = \zeta(0) = x$, it follows by construction that $\hat{i}^n>n$ and $\hat{S}^n>0$. Now we define the controls $\bar \nu^n$ as, 
\begin{align*}
\bar \nu^n_j = \left\{\begin{array}{ll}
                \hat \nu^n_j(dz), & \text{ if } j<\hat{i}^n,\\
                \rho_{\bar X^n_{j-1}}(\bar Y^n_{j-1},dz), & \text{ otherwise}.\end{array} \right.
\end{align*}
This defines the controls, that is the conditional distributions $\{ \bar \nu ^n _j \}_{j=n}^{n+\beta_n}$ for the noise, and the corresponding controlled process $\bar X ^n = \{ \bar X ^n _j\}_{j=n}^{n+\beta_n}$. To have a control in continuous time, we define $\bar{\nu}^n(A|t) = \bar{\nu}^n_i(A)$ for $t\in [t_{i-1}-t_{n},t_{i}-t_n)$ and the measure $\bar\nu ^n \in \mathcal{P}(\mathbb{R}^{d_2}\times [0,T])$ by, 
\[
\bar\nu^n(A\times B) = \int_B\frac{1}{h^n(t)}\bar\nu^n(A|t)dt.
\]
Throughout the paper, unless there is ambiguity, we suppress the dependence on $\delta$ in the control sequence $\{\Bar{\nu}^n_j\}$ in the notation.
The following lemma gives the tightness of the control sequence $\{ \bar \nu ^n \} _n$, which will be used in the convergence analysis needed for proving Theorem \ref{thm:lower}.
%
%
\begin{lemma}
\label{lem:tightness}
Under \ref{ass:Lipschitz}-\ref{ass:logMGF}, the control sequence $\{\bar{\nu}^n\}_n$ is tight.
\end{lemma}

\begin{proof}
    The proof relies on arguments similar to those used for Lemma 4.11 in \cite{buddup4} and Proposition 5.3.2 in \cite{dupell}. We will need that,  under the given conditions,
    \[
    \sup_nE\left[\frac{1}{\beta_n}\sum_{i=n}^{\beta_n+n-1}R(\bar \nu_{i+1} ^n(\cdot)||\rho_{\bar{X}_{i+1} ^n}(\bar{Y}^n_i,\cdot))\right]<\infty,
    \]
    for the constructed sequence of controls. The statement in the last display is proved in Lemma \ref{lem:bounded_rel_ent}. To prove the claimed tightness, it is sufficient to prove that $\bar\nu^n$ satisfies the uniform integrability property,
    \begin{equation*}
    \lim_{C\to \infty}\limsup_n E\left[\int_0^T\int_{\|z\|>C} \|z\|\bar\nu^n(dz\times dt) \right]= 0.
    \end{equation*}
    To prove this uniform integrability, we use the inequality $ab\leq e^{\sigma a} + \frac{1}{\sigma}(b\log(b) - b+1)$ with $a = \|z\|$ and $b = \frac{d \bar \nu^n_i(\cdot)}{d\rho_{\tilde{X}^n_i}(\tilde{Y}^n_i,\cdot)}$ evaluated at points $z$ with $\|z\| > C$. For $t \in [0,T]$, and fixed $C$ and $n$, we have,
\begin{align*}
    &\int_{\| z\|>C} \| z\|d \bar\nu^n_i(dz) \\
    & \quad = \int_{\|z\|>C} \| z\|\frac{d\bar\nu^n_i(z)}{d\rho_{\bar{X}^n_i}(\bar{Y}^n_i,z)}\rho_{\bar{X}^n_i}(\bar{Y}^n_i,dz)\\
    & \quad \leq \int_{\|z\|>C} e^{\sigma\| z\|} \rho_{\bar{X}^n_i}(\bar{Y}^n_i,dz) \\
    &\qquad  + \frac{1}{\sigma}\int_{\|z\|>C}\left(\frac{d\nu^n_i(z)}{d\rho_{\bar{X}^n_i}(\bar{Y}^n_i,z)}\log \left(\frac{d\bar\nu^n_i(z)}{d\rho_{\bar{X}^n_i}(\bar{Y}^n_i,z)}\right) -\frac{d\bar\nu^n_i(z)}{d\rho_{\bar{X}^n_i}(\bar{Y}^n_i,z)} +1  \right)\rho_{\bar{X}^n_i}(\bar{Y}^n_i,dz)\\
    &\quad \leq  \int_{\| z\|>C} e^{\sigma\| z\|} \rho_{\bar{X}^n_i}(\bar{Y}^n_i,dz) + \frac{1}{\sigma}R(\bar\nu^n_i(\cdot)||\rho_{\bar{X}^n_i}(\bar{Y}^n_i,\cdot))\\
    &\quad \leq e^{-\sigma C}\sup_x\sup_y \int e^{2\sigma \|z\|}\rho_x(y,dz) + \frac{1}{\sigma}R(\bar\nu^n_i(\cdot)||\rho_{\bar{X}^n_i}(\bar{Y}^n_i,\cdot)),
\end{align*}
where in the last step we have used Assumption \ref{ass:logMGF} to guarantee that the first term is finite. Moreover, the bounded expected running cost, see \eqref{eq:finite_control} and Lemma \ref{lem:bounded_rel_ent}, ensures that the second term in the last display is bounded in $n$. Using the bound in the previous display yields,
\begin{align*}
    &E\left[\int_0^T\int_{\|z\|>C} \|z\|\bar \nu^n(dz\times dt) \right] \\
    &\leq\! E\left[\sum_{i=n}^{\beta_n+n-1} \!\!\!\int_{t_i,t_{i+1}}\!\! \frac{1}{h_n(t)} e^{-\sigma C}\sup_x\sup_y \!\!\int \!\!e^{2\sigma \|z\|}\rho_x(y,dz) \!+\! \frac{1}{h_n(t)}\frac{1}{\sigma}R(\bar\nu^n_i(\cdot)||\rho_{\bar{X}^n_i}(\bar{Y}^n_i,\cdot))dt\right] \\
    &=\leq  e^{-\sigma C}\sup_x\sup_y \int e^{2\sigma \|z\|}\rho_x(y,dz) + \frac{1}{\sigma}E\left[\frac{1}{\beta_n}\sum_{i=n}^{n+\beta_n-1}R(\bar\nu^n_i(\cdot)||\rho_{\bar{X}^n_i}(\bar {Y}^n_i,\cdot))\right].
\end{align*}
The first term does not depend on $n$ and by Lemma \ref{lem:bounded_rel_ent} the second term is bounded. Sending first $C\to \infty$ and then $\sigma\to\infty$ yields the uniform integrability and the tightness, which completes the proof. 
\end{proof}
\subsection{Convergence of controls and controlled processes}
\label{sec:conv}
A key step in the weak convergence approach is to show convergence of the controls and associated controlled processes, and to identify the limit objects and their properties. In this section we carry out such an analysis for the pairs $(\bar \nu ^n, \bar X ^n)$. 

Take $\varepsilon >0$ and let $\zeta \in C([0,T]:\mathbb{R} ^{d_1}) $ be $\varepsilon$-optimal with regards to $\inf _\varphi \{ F(\varphi) + I(\varphi)\}$. From Lemma \ref{lem:piecewise_conti}, using the continuity of $F$, we know that there is a piecewise linear $\zeta ^* \in C([0,T]:\mathbb{R} ^{d_1})$, with finitely many pieces, such that $|| \zeta - \zeta ^*||_\infty < \varepsilon$, and, 
\begin{align}
\label{eq:boundFI}
F(\zeta^*) + I(\zeta ^*) \leq F(\zeta) + I(\zeta) + \varepsilon.
\end{align}
For such a $\zeta ^*$, consider the associated measures $\{\nu ^{\zeta ^* (t), \dot \zeta ^* (t)}: t \in [0,T]\}$ from Lemma \ref{lem:optimal_control}; throughout the section we suppress the dependence on $\varepsilon$ in the notation. The following theorem is the main result of this section. 

\begin{theorem}
\label{thm:convX}
Fix $\varepsilon >0$ and $\zeta ^*$ according to \eqref{eq:boundFI}. Under \ref{ass:Lipschitz}-\ref{ass:limith}, for every subsequence of $\{(\bar \nu^n,\bar{X}^n)\}$, there exists a further subsequence that converges weakly to $(\bar \nu, \zeta^*)$, where $\bar \nu$ satisfies,
\[
    \bar \nu (A \times B) = \int _B \frac{1}{h(t)}\bar \nu (A|t)dt,
\]
and $\bar \nu (\cdot|t) = \nu ^{\zeta^* (t), \dot \zeta ^* (t)} (\cdot)$.
\end{theorem}
The proof relies on showing that the limit $\bar X$ of $\bar X^n$ satisfies, 
\[
     \bar{X}(t) = x + \int_0^t\int_{\mathbb{R}^{d_2}}g(\bar{X}(s),y)\bar \nu(dy|s)ds,
\]
and that, by construction of the $\nu ^{\zeta ^*(t), \dot \zeta^* (t)}$-measures, $\zeta ^*$ satisfies the equation,
\begin{align}
\label{eq:ODEzeta}
    \zeta ^* (t) = \int _0 ^t \int _{\mathbb{R}^{d_2}} g(\zeta ^*(s),y) \nu ^{\zeta^* (s), \dot \zeta ^* (s) } (dy)ds.
\end{align}
To see the latter, and that the solution of \eqref{eq:ODEzeta} is piecewise linear, first note that for any $s \in [0,T]$, by definition we have,
\[
    \int _{\mathbb{R}^{d_2}} g(\zeta ^*(s),y) \nu ^{\zeta^* (s), \dot \zeta ^* (s) } (dy) = \dot \zeta ^* (s).
\]
Moreover, because $\zeta ^*$ is piecewise linear, $\dot \zeta ^*$ is constant on intervals. Time-integrals over such intervals are therefore just the length of the interval times the corresponding value of $\dot \zeta ^*$. Integrating over multiple intervals will then result in a sum of the corresponding values of $\dot \zeta ^*$ times the length of the different intervals: this is precisely $\zeta ^*$. Thus, the piecewise linear function $\zeta ^*$ satisfies \eqref{eq:ODEzeta}. Moreover, in Lemma \ref{lem_unique} in the Appendix we show that \eqref{eq:ODEzeta} has a unique solution, which must then be $\zeta ^*$.

Theorem \ref{thm:convX} is proved by a series of lemmas and theorems. We start with an ancillary result, Lemma \ref{lem:mark_l0}, which will be used to prove tightness of the controlled processes $\bar X ^n$ for generic controlled measures with bounded relative entropy with respect to $\rho$ along the controlled process, see Theorem \ref{thm:limit}. The proof is identical to the proof of Lemma 6.16(b) in \cite{buddup4}; we omit the details.
\begin{lemma}
\label{lem:mark_l0}
Let $l_0$ be the constant in the transitivity condition \ref{ass:transitivity}. If a Borel set $A$ has the property that $\rho_x^{l_0}(y,A)>0$ for some $x,y$, then $\pi_x(A)>0$.
\end{lemma}

Using Lemma \ref{lem:mark_l0}, we now prove that the expected running cost associated with the controlled measures $\{\bar \nu ^n \}_n$ is bounded.

\begin{lemma}
\label{lem:bounded_rel_ent}
Under \ref{ass:kernel}, \ref{ass:bounded_kernel}, \ref{ass:transitivity} and \ref{ass:logMGF}, with $\bar \nu ^n = \{\bar \nu ^n _j \} _{j=n+1} ^{\beta_n + n}$ defined as in Section \ref{sec:optimalControl}, it holds that,
\begin{align*}
     \sup_nE\left[\frac{1}{\beta_n}\sum_{i=n}^{\beta_n+n-1}R(\bar \nu_{i+1} ^n(\cdot)||\rho_{\bar{X}_{i+1} ^n}(\bar{Y}^n_i,\cdot))\right]<\infty.
\end{align*}
\end{lemma}
\begin{proof}
First observe that for $i\geq \hat{i}^n$ we have that $\bar \nu ^n _{i+1} = \rho _{\bar{X}_{i} ^n}(\bar{Y}^n_i,\cdot)$, and consequently,
\[
R(\bar \nu_{i+1} ^n(\cdot)||\rho_{\bar{X}_{i} ^n}(\bar{Y}^n_i,\cdot)) = 0.
\]
From here on we only need to consider indices $j<\hat{i}^n$.
    For each $n$, with $\ell$ and $m$ as in Section \ref{sec:optimalControl}, from the definition of the $\bar \nu ^n _j$s we have,
    \begin{align*}
\frac{1}{\beta_n} \sum_{i=n}^{\beta_n +n -1} R\left( \bar{\nu}^n_{i+1}(\cdot) \| \rho_{\bar X^n_{i+1}}(\bar Y^n_{i},\cdot)\right) &= 
\frac{1}{\beta_n} \sum_{k=0}^{{\ell}}\sum_{j=1}^{m} R\left( \bar{\nu}^n_{n+km+j}(\cdot)\| \rho_{\bar X^n_{n+km+j-1}}(\bar Y^n_{n+km+j-1},\cdot)\right).
    \end{align*}
Using the definition of relative entropy, for each $k$ and $j$ in the relevant ranges, we can re-write the relative entropy-term on the right-hand side of the last display as,
\begin{align}
\label{eqn_relative_entropy}
    &R\left( \bar{\nu}^n_{n+km+j}(\cdot)\| \rho_{\bar X^n_{n+km+j}}(\bar Y^n_{n+km+j-1},\cdot)\right) \nonumber\\
    & \quad = R\left( \bar{\nu}^n_{n+km+j}(\cdot)\| \rho_{\zeta^*(\tau^\ell_k)}(\bar Y^n_{km+j-1},\cdot)\right) \\ & \qquad 
    \!+\! \int_{\mathbb{R}^{d_2}}\!\!\! \left(\log\frac{d\rho_{\zeta^*(\tau^\ell_k)}(\bar Y^n_{n+km+j-1},y)}{d\rho_{\bar X^n_{n+km+j}}(\bar Y^n_{n+km+j-1},y)}\right)\bar{\nu}^n_{n+km+j}(dy). \nonumber 
\end{align}
Note also that $R\left( \bar{\nu}^n_{n+km+j-1}(\cdot)\| \rho_{\zeta^*(\tau^\ell_k)}(\bar Y^n_{km+j-1},\cdot)\right) = 0$ for $j \leq l_0$.

Take $\delta >0$. For any $k \in \{1, \dots, \ell\}$, consider the integral
\begin{align*}
    \int _{\bR ^{d_2}} R\left( q^{\zeta^*(\tau^\ell_k),\dot{\zeta} ^* (\tau^\ell_k),\delta}(y,\cdot)\| \rho_{\zeta^*(\tau^\ell_k}(y,\cdot)\right)\mu^{\zeta^*(\tau^\ell_k),\dot{\zeta}^*(\tau^\ell_k), \delta}(dy).
\end{align*}
We will show that, as $m \to \infty$, which corresponds to the limit $n \to \infty$, this integral approximates 
\begin{align*}
 \frac{1}{m}\sum_{j=l_0+1}^{m} R\left( q^{\zeta^*(\tau^\ell_k),\dot{\zeta}^*(\tau^\ell_k),\delta}(\bar{Y}^n_{km+j-1},\cdot)\| \rho_{\zeta^*(\tau^\ell_k)}(\bar Y^n_{km+j-1},\cdot)\right),
\end{align*}
that is the (normalised) sum over first term appearing in the alternative representation \eqref{eqn_relative_entropy} of the running cost. To show this, we use arguments similar to those used in the proof of Proposition 6.15 in \cite{buddup4}. First, from Lemma \ref{lem:optimal_control}, 
\begin{align*}
    & \bE \left[ \frac{1}{m}\sum_{j=l_0 + 1}^{m} R\left( q^{\zeta^*(\tau^\ell_k),\dot{\zeta}^*(\tau^\ell_k),\delta}(\bar{Y}^n_{km+j-1},\cdot)\| \rho_{\zeta^*(\tau^\ell_k)}(\bar Y^n_{km+j-1},\cdot)\right) \right] \\
    &\quad =  \int _{\bR ^{d_2}} R\left( q^{\zeta^*(\tau^\ell_k),\dot{\zeta}^*(\tau^\ell_k),\delta}(y,\cdot)\| \rho_{\zeta^*(\tau^\ell_k}(y,\cdot)\right)\mu^{\zeta^*(\tau^\ell_k),\dot{\zeta}^*(\tau^\ell_k), \delta}(dy) \\
    &\quad \leq L\left(\zeta ^* (t ^n _{n +km}), \dot \zeta ^* (t^n _{n+km})\right) + \varepsilon,
\end{align*}
and from the properties of $\zeta ^*$ this upper bound is finite. The non-negativity of the relative entropy, and the properties of the $\mu^{\zeta^*(\tau^\ell_k),\dot{\zeta^*}(\tau^\ell_k), \delta}$-measures, and associated Markov chains, the $L^1$-ergodic theorem implies the convergence,
\begin{align*}
     &\lim_{m\to\infty}E \Bigg[ \Bigg\|\frac{1}{m}\sum_{j=l_0 +1 }^{m} R\left( q^{\zeta^*(\tau^\ell_k),\dot{\zeta^*}(\tau^\ell_k),\delta}(\bar{Y}^n_{km+j-1},\cdot)\| \rho_{\zeta^*(\tau^\ell_k)}(\bar Y^n_{km+j-1},\cdot)\right)\\
    &\qquad\qquad\qquad\qquad - \int R\left( q^{\zeta^*(\tau^\ell_k),\dot{\zeta}(\tau^\ell_k),\delta}(y,\cdot)\| \rho_{\zeta^*(\tau^\ell_k)}(y,\cdot)\right)\mu^{\zeta^*(\tau^\ell_k),\dot{\zeta^*}(\tau^\ell_k),\delta}(dy)\Bigg\| \Bigg]=0.
\end{align*}
As a consequence of the convergence in the last display, it follows that, for any $y_k$,
\begin{align*}
& E _{y_k} \Bigg[ \Bigg\|\frac{1}{m}\sum_{j=l_0 + 1}^{m} R\left( q^{\zeta^*(\tau^\ell_k),\dot{\zeta}(\tau^\ell_k),\delta}(\bar{Y}^n_{km+j-1},\cdot)\| \rho_{\zeta^*(\tau^\ell_k)}(\bar Y^n_{km+j-1},\cdot)\right)\\
    &\qquad\qquad\qquad\qquad- \int R\left( q^{\zeta^*(\tau^\ell_k),\dot{\zeta}(\tau^\ell_k),\delta}(y,\cdot)\| \rho_{\zeta^*(\tau^\ell_k)}(y,\cdot)\right)\mu^{\zeta^*(\tau^\ell_k),\dot{\zeta^*}(\tau^\ell_k),\delta}(dy)\Bigg\| \Bigg],
\end{align*}
converges in probability to 0, as $m \to \infty$. This is turn ensures that, for any $k \in \{1, \dots ,\ell \}$, there is a further subsequence of $\{m \}$---we abuse notation and denote this subsequence by $\{ m \}$ as well---and a Borel set $\Phi _k$ such that $\mu^{\zeta^*(\tau^\ell_k),\dot{\zeta^*}(\tau^\ell_k),\delta} (\Phi _k) =1$, and for any $\bar Y ^n _{n +km +l_0} = y_k \in \Phi _k$,
\begin{align*}
    &\lim_{m\to\infty}E_{y_k}\left[ \left|\frac{1}{m}\sum_{j=l_0+1}^{m} R\left( q^{\zeta^*(\tau^\ell_k),\dot{\zeta^*}(\tau^\ell_k),\delta}(\bar{Y}^n_{km+j-1},\cdot)\| \rho_{\zeta^*(\tau^\ell_k)}(\bar Y^n_{km+j-1},\cdot)\right)\right. \right.\\
    &\qquad\qquad\qquad\qquad\left.\left.- \int R\left( q^{\zeta^*(\tau^\ell_k),\dot{\zeta^*}(\tau^\ell_k),\delta}(y,\cdot)\| \rho_{\zeta^*(\tau^\ell_k)}(y,\cdot)\right)\mu^{\zeta^*(\tau^\ell_k),\dot{\zeta^*}(\tau^\ell_k),\delta}(dy)\right| \right] =0.
\end{align*}
We now show that $\bar Y ^n _{n +km + l_0} \in \Phi _k$ w.p.\ 1. Because $\mu^{\zeta^*(\tau^\ell_k),\dot{\zeta^*}(\tau^\ell_k),\delta} (\Phi _k) =1$ and $ \pi_{\zeta^*(\tau^\ell_k)}\ll \mu^{\zeta^*(t^n_{km}),\dot\zeta^*(t^n_{km}),\delta}$, it holds that $\pi_{\zeta^*(\tau^\ell_k)}(\Phi_k^c)=0$. Lemma \ref{lem:mark_l0} then implies that $\rho_{\zeta^*(\tau^\ell_k)}^{l_0}(y,\Phi_k^c)=0$. 
This, combined with the fact that we only consider a finite number $\ell$ terms, gives the convergence,
\begin{align*}
    &\lim_{m\to\infty}\max_{k\in\{1,\dots,\ell\}}E \left[\left|\frac{1}{m}\sum_{j=l_0}^{m-1} R\left( q^{\zeta^*(\tau^\ell_k),\dot{\zeta}(\tau^\ell_k),\delta}(\bar{Y}^n_{km+j},\cdot)\| \rho_{\zeta^*(\tau^\ell_k}(\bar Y^n_{km+j},\cdot)\right)\right.\right.\\
    &\qquad\qquad\qquad\qquad\left.\left. - \int R\left( q^{\zeta^*(\tau^\ell_k),\dot{\zeta}(\tau^\ell_k),\delta}(y,\cdot)\| \rho_{\zeta^*(\tau^\ell_k}(y,\cdot)\right)\mu^{\zeta^*(\tau^\ell_k),\dot{\zeta^*}(\tau^\ell_k)}(dy)\right| \right] =0.
\end{align*}
It follows that,
\[
     \sup_m E\left[ \frac{1}{m}\sum_{j=0}^{m-1} R\left( q^{\zeta^*(\tau^\ell_k),\dot{\zeta}(\tau^\ell_k),\delta}(\bar{Y}^n_{km+j},\cdot)\| \rho_{\zeta^*(\tau^\ell_k)}(\bar Y^n_{km+j},\cdot)\right)\right] < \infty.
\]
Next, we consider the second term in \eqref{eqn_relative_entropy},
\begin{align}
\label{eq:running2}
\int_{\mathbb{R}^{d_2}} \left(\log\frac{d\rho_{\zeta^*(\tau^\ell_k)}(\bar Y^n_{n+km+j-1},y)}{d\rho_{\bar X^n_{n+km+j}}(\bar Y^n_{n+km+j-1},y)}\right)\bar{\nu}^n_{n+km+j}(dy).
\end{align}
By the continuity of $\zeta^*$, and because $n+km+j<\hat{i}^n$, there exists a compact set $K$ such that $\zeta^*(\tau^l_k),\Bar{X}^n_{n+km+j}\in K$. By Assumption \ref{ass:bounded_kernel}, there exists a $C$ such that $\log\frac{d\rho_{\zeta^*(\tau^\ell_k)}(y,z)}{d\rho_{\bar{X}^n_{n+km+j}}(y,z)}\leq C$, for all $n$. This ensures that \eqref{eq:running2} is bounded. We conclude that the sums over the terms appearing in the representation \eqref{eqn_relative_entropy} are both bounded in $n$. Consequently, the expected running cost associated with $\{ \bar \nu ^n \}_n$ is bounded.



\end{proof}
In proving Theorem \ref{thm:convX}, the main step is to prove a version of the theorem, Theorem \ref{thm:limit}, for a general class of control measures, satisfying bounded expected running cost; by Lemma \ref{lem:bounded_rel_ent}, we know that the sequence of control measures constructed in Section \ref{sec:optimalControl} belongs to this class. To this end, 
%
we consider a (generic) sequence of measures $\tilde \nu ^n _i \in \calP (\bR ^{d_2})$ such that,
\begin{equation}
\label{eq:finite_control}
    \sup_nE\left[\frac{1}{\beta_n}\sum_{i=n}^{\beta_n+n-1}R( \tilde \nu_i^n(\cdot)||\rho_{\tilde{X}_i^n}(\tilde{Y}^n_i,\cdot))\right]<\infty.
\end{equation}
Define the corresponding controlled process $\{\tilde X ^n _k\} _{k \geq n}$ as before: $\tilde X ^n _n = x$ and
\[
\tilde X ^n _{k+1} = \tilde X ^n _k + \varepsilon _k g(\tilde X ^n_k, \tilde Y ^n _k),
\]
where $\tilde \nu ^n _k$ is the conditional distribution for $\tilde Y ^n _k$ given $\sigma \left( \tilde Y ^n _n, \dots , \tilde Y ^n _{i-1} \right)$. Similar to before, we take $\tilde X ^n \in C([0,T]:\bR ^{d_{1}})$ as the linear interpolation with breakpoints $\tilde X ^n (t_{n+k} -t_n) = \tilde X ^n _k$. We also abuse notation a bit and define $\tilde \nu \in \calP (\bR ^{d_2} \times [0,T])$ as,
\[
    \tilde \nu ^n (A \times B) = \int _B \frac{1}{h^n(t)} \tilde \nu ^n (A|t) dt,
\]
where $\tilde \nu ^n (A|t) = \tilde \nu ^n _i (A)$ when $t \in [t_{n+i-1} - t_n, t_{n+i}-t_n)$. 



\begin{theorem}
\label{thm:limit}
Under \ref{ass:Lipschitz}-\ref{ass:limith}, for every subsequence of $\{(\tilde \nu^n,\tilde{X}^n)\}$ where $\{ \tilde \nu^n\}$ satisfies \eqref{eq:finite_control}, there exists a further subsequence that converges weakly to $( \tilde \nu,\tilde {X})$. Furthermore, there exists a stochastic kernel $\tilde \nu(dy|t)$ such that, 
\begin{equation*}
    \tilde \nu(A\times B) = \int_B \tilde \nu(A|t)\frac{1}{h(t)}dt,
\end{equation*}
and $\tilde {X}$ satisfies,
\begin{equation}\label{eqn:x_bar}
     \tilde{X}(t) = x + \int_0^t\int_{\mathbb{R}^{d_2}}g(\tilde{X}(s),y)\tilde \nu(dy|s)ds.
\end{equation}
\end{theorem}
The proof of Theorem \ref{thm:limit} is presented in Section \ref{sec:limit}. Note that the form of the limit measure $\tilde \nu$ is a direct consequence of Lemma 3.3.1 in \cite{dupell4} and the uniform convergence $h^n \to h$, ensured by \ref{ass:limith}. We also have that since $\hat{S}^n$ takes values in the compact set $[0,T]$, there is a subsequence that converges to $\hat{S}\in [0,T]$.  Theorem \ref{thm:convX} follows from this result if we can show that the limit point for the appropriate subsequences of the specific choice of control measures in Section \ref{sec:optimalControl} have the claimed form. We prove this in two steps, carried out in Lemmas \ref{lem:limit_m} and \ref{lem:limit_delta}. The strategy is to first send $m$ to infinity, and find the corresponding limit point $\bar x ^\ell$ of $\bar X ^n$, see Lemma \ref{lem:limit_m}. Recall that from how we chose $m$ and $\ell$, for fix $\ell$, taking $n$ to infinity also means taking $m$ to infinity, and vice versa; as before, at times we suppress the dependence on $\delta$ and $\varepsilon$ in the notation. Next, we send $\ell$ to infinity and $\delta$ to 0, and show that the corresponding limit for the $\bar x ^\ell$ is $\zeta ^*$, see Lemma \ref{lem:limit_delta}.  That is, we show the following convergence results: 
\[
    \bar{X}^n \xrightarrow[Lemma\hspace{0.3em} \ref{lem:limit_m}]{m\to\infty}\bar{x}^{\ell}\xrightarrow[Lemma\hspace{0.3em} \ref{lem:limit_delta}]{\delta\to 0, \hspace{0.3em}\ell\to\infty}  \zeta ^*.
\]
We start with the first part: using Theorem \ref{thm:limit}, applied to the sequence $\{ (\bar \nu ^n, \bar X ^n) \}$, we characterise the limit point $\bar x ^\ell$ and prove that $\bar{X}^n \to \bar{x}^{\ell}$ in probability as $m \to \infty$. 

\begin{lemma}\label{lem:limit_m}Under \ref{ass:Lipschitz}-\ref{ass:limith}, for any $\delta>0$, $\ell\in\mathbb{N}$, $\{\bar X^n\} _n$ is tight. Moreover, the convergent subsequences of $\{\bar X^n\}$ converge to $\bar x^{\ell}$ in probability, as $m\to \infty$, where $\bar x^{\ell}$ satisfies,
\begin{align*}
    \bar x^{\ell}(t)  = x_0 &+ \sum_{i=0}^{k-1}\int_{\tau^\ell_i}^{\tau^\ell_{i+1}}\int_{\mathbb{R}^{d_2}} g(\bar x^{\ell}(s),y)\mu^{\zeta^*(\tau ^\ell _i),\dot{\zeta^*}(\tau^\ell_i),\delta}(dy)ds \\
    &+ \int_{\tau^\ell_k}^t\int_{\mathbb{R}^{d_2}} g(\bar x^{\ell}(s),y)\mu^{\zeta^*(\tau ^\ell _k),\dot{\zeta^*}(\tau ^\ell _k),\delta}(dy)ds, 
\end{align*}
for $t\in[\tau^\ell_k,\tau^\ell_{k+1})$ and $t\leq\hat{S^n}$.
\end{lemma}

\begin{proof}

For $j \in \{n+km +1, \dots, n + (k+1)m \}$ and $k \in \{0, \dots, \ell\}$, consider $t \in [t^n _{j-1}, t^n _j )$.  
Because we will consider the limit as $m \to \infty$, to emphasise the dependence on $m$ in the $\bar \nu ^n _j$s, we define,
\begin{align*}
    \gamma ^m (dz | t) = \bar \nu ^n _j (dz) = \begin{cases} 
        \rho _{\zeta ^* (\tau^\ell _k)} (\bar Y ^n _{j-1}, dz), & n +km +1 \leq j \leq n +km + l_0, \\
        q ^{\zeta ^*(\tau^\ell _k), \dot{\zeta} ^* (\tau^\ell _k), \delta} (\bar Y ^n _{j-1}, dz), & n+km+l_0 +1 \leq j \leq n + (k+1)m -1.
    \end{cases}
\end{align*}
Moreover, for $j \in \{ n +\ell m +1, \dots, n + \beta_n \}$, i.e., when $t \in [t^n _{\ell m}, T)$ or $t\geq \hat{S}^n$, we set,
\begin{align*}
    \gamma ^m (dz|t) = \rho _{\bar X ^n _{j-1}} (\bar Y ^n _{j-1}, dz).
\end{align*}
For notational brevity and clarity, we also define,
\[
    \gamma^m(A\times B) \doteq \int_B \gamma^m(A|t)dt,
\]
and
\[
     \gamma(A\times B) \doteq \int_B \gamma(A|t)dt, \quad \gamma(A|t) = \mu^{\zeta^*(\tau ^\ell _{k}),\dot{\zeta^*}(\tau ^\ell _{k}),\delta}(A), \quad t\leq \hat{S},
\]
where $t$ and $k$ are as above. For $t>\hat{S}$, we set $\gamma(A|t) = \lim_{m\to\infty} \gamma^m(A|t)$. Note that $\gamma ^m$ and $\gamma ^m(\cdot |t)$ are playing the roles of $\bar \nu ^n$ and $\bar \nu ^n (\cdot |t)$. Combining Lemma \ref{lem:bounded_rel_ent} and Theorem \ref{thm:limit}, with these definitions of $\gamma ^m$ and $\gamma$, it is enough to show that $\gamma^m$ converges weakly to $\gamma$ w.\ p.\ 1.

To prove the convergence of $\gamma^m$, 
consider any bounded and uniformly continuous function $f:\mathbb{R}^{d_2} \times [0,T]\to \mathbb{R}$. By the Portmanteau theorem, it is enough to prove that, 
\[
\int _{\bR ^{d_2} \times [0, T]} f(y,t)\gamma ^m(dydt), 
\]
converges, as $m \to \infty$, to,
\begin{align*}
\int_{\bR ^{d_2} \times [0, T]} f(y,t) \gamma (dy dt).
\end{align*}
Since $\gamma^m(dy|t) \to \gamma(dy|t)$ by definition for $t>\hat{S}$ the only interesting case is for the interval $[0,\hat{S}]$. Below the proof is constructed with $\hat{S} = T$. The case with $\hat{S}<T$ is completely analogous (carried out over a shorter time interval). 
From the definition of $\gamma ^m$ we have,
\begin{align*}
     \int _{\bR ^{d_2} \times [0, T]} f(y,t)\gamma ^m(dydt) 
    & = \sum _{k=0} ^{\ell} \sum_{j=n+km + 1}^{n+(k+1)m} \int_{t^n_{j-1}}^{t^n_{j}}\int_{\mathbb{R}^{d_2}}f(y,t)\gamma^m(dy|t)dt \\
    &\quad + \int _{t^n _{n + \ell m}} ^T \int _{\bR ^{d_2}} f(y,t ) \gamma ^m (dy |t)dt.
\end{align*}
As a first step, for each  $k \in \{0, 1, \dots, \ell\}$, we consider the difference 
\begin{align}
\label{eq:fixk}
\sum_{j=n+km + 1}^{n+(k+1)m} \int_{t^n_{j-1}}^{t^n_{j}}\int_{\mathbb{R}^{d_2}}f(y,t)\gamma^m(dy|t)dt - \int _{\bR ^{d_2} \times [\tau ^\ell _k, \tau ^\ell _{k+1})} f(y,t) \gamma (dy dt).
\end{align}

In preparation for studying \eqref{eq:fixk} in the limit $m \to \infty$, we make the following definitions. Let, 
\begin{align*}
    C_1 ^k (m) &= \sum_{j=n+km + 1}^{n+km + l_0} \int_{t^n_{j-1}}^{t^n_{j}}\int_{\mathbb{R}^{d_2}}f(x,t)\rho_{\bar X^n_{j-1}}(\bar{Y}^n_{j-1},dx)dt,\\
    C_2 ^k (m) &= \sum_{j=n+km + l_0+1} ^{n+(k+1) m} \int_{t^n_{j-1}}^{t^n_{j}}\int_{\mathbb{R}^{d_2}}f(x,t)q^{\zeta^*(\tau^\ell _k),\dot{\zeta} ^*(\tau ^\ell _k),\delta}(\bar{Y}^n_{j-1},dx)dt \\
    &- \sum_{j=n+km + l_0+1}^{n+(k+1)m} \int_{t^n_{j-1}}^{t^n_{j}}\int_{\mathbb{R}^{d_2}}f\left(x,t^n_{j-1}\right)q^{\zeta^*(\tau ^\ell _k),\dot{\zeta} ^*(\tau ^\ell _k),\delta}(\bar{Y}^n_{j-1},dx)dt,\\
    C_3 ^k (m) &= \sum_{j=n+km + l_0+1}^{n+(k+1)m} \varepsilon_{j}\int_{\mathbb{R}^{d_2}}f\left(x,t^n_{j-1}\right)q^{\zeta^*(\tau ^\ell _k),\dot{\zeta} ^*(\tau ^\ell _k),\delta}(\bar{Y}^n_{j-1},dx) \\
    &- \sum_{j=n+km + l_0+1}^{n+(k+1)m} \varepsilon_{j}\int_{\mathbb{R}^{d_2}}\int_{\mathbb{R}^{d_2}} f(x,t^n_{j-1})q^{\zeta ^*(\tau ^\ell _k),\dot{\zeta} ^*(\tau ^\ell _k),\delta}(y,dx) \mu^{\zeta ^*(\tau ^\ell _k),\dot{\zeta} ^*(\tau ^\ell _k),\delta}(dy),\\
    C_4 ^k (m) &= \sum_{j=n+km + l_0+1}^{n+(k+1)m} \varepsilon_{j}\int_{\mathbb{R}^{d_2}} f(y,t^n_{j-1}) \mu^{\zeta ^*(\tau ^\ell _k),\dot{\zeta} ^*(\tau ^\ell _k),\delta}(dy) \\
    &- \int_{[\tau ^\ell _k ,\tau^\ell_{k+1}]} \int_{\mathbb{R}^{d_2}} f(y,t) \mu^{\zeta^*(\tau ^\ell _k),\dot{\zeta} ^*(\tau ^\ell _k),\delta}(dy)dt.
\end{align*}
With these definitions, we now rewrite $\sum_{j=n+km + 1}^{n+(k+1)m} \int_{t^n_{j-1}}^{t^n_{j}}\int_{\mathbb{R}^{d_2}}f(y,t)\gamma^m(dy|t)dt$ in terms of $C_i ^k (m)$, $i=1,\dots, 4$, and $\int _{\bR ^{d_2} \times [t^n _{n +km}, t^n _{n + (k+1)m})} f(y,t) \gamma (dy dt)$. First, we split the sum over $j$ into two terms according to $l_0$:

\begin{align*}
    &  \sum_{j=n+km + 1}^{n+(k+1)m} \int_{t^n_{j-1}}^{t^n_{j}}\int_{\mathbb{R}^{d_2}}f(y,t)\gamma^m(dy|t)dt \\
    & \quad = \sum_{j=n+km + l_0 +1 } ^{n+(k+1)m} \int_{t^n_{j-1}}^{t^n_{j}}\int_{\mathbb{R}^{d_2}}f(y,t) q^{\zeta ^*(\tau^\ell _k), \dot \zeta ^* (\tau^\ell _k), \delta } (\bar Y ^n _{j-1}, dy) dt + C^k _1 (m).
\end{align*}
Next, for each interval $[t^n _{j-1}, t^n _j)$, we freeze the time-variable $t$ inside $f(y,t)$ at $t^n _{j-1}$:
\begin{align*}
    & \sum_{j=n+km + l_0 +1 } ^{n+(k+1)m} \int_{t^n_{j-1}}^{t^n_{j}}\int_{\mathbb{R}^{d_2}}f(y,t) q^{\zeta ^*(\tau^\ell _k), \dot \zeta ^* (\tau^\ell _k), \delta } (\bar Y ^n _{j-1}, dy) dt + C^k _1 (m) \\
    & \quad = \sum_{j=n+km + l_0+1}^{n+(k+1)m} \int_{t^n_{j-1}}^{t^n_{j}}\int_{\mathbb{R}^{d_2}}f\left(y,t^n_{j-1}\right)q^{\zeta^*(\tau ^\ell _k),\dot{\zeta} ^*(\tau ^\ell _k),\delta}(\bar{Y}^n_{j-1},dy)dt + C^k _1 (m) + C^k _2 (m) \\
    & \quad = \sum_{j=n+km + l_0+1}^{n+(k+1)m} \varepsilon _j \int_{\mathbb{R}^{d_2}}f\left(y,t^n_{j-1}\right)q^{\zeta^*(\tau ^\ell _k),\dot{\zeta} ^*(\tau ^\ell _k),\delta}(\bar{Y}^n_{j-1},dy) + C^k _1 (m) + C^k _2 (m),
\end{align*}
where in the second step we have used that the integral over the time variable is equal to $t^n _j - t^n _{j-1} = \varepsilon _j$. Next, by averaging over the controlled variable $\bar Y ^n _{j-1}$, we can write the last display as,
\begin{align*}
    & \sum_{j=n+km + l_0+1}^{n+(k+1)m} \varepsilon _j \int_{\mathbb{R}^{d_2}}f\left(y,t^n_{j-1}\right)q^{\zeta^*(\tau ^\ell _k),\dot{\zeta} ^*(\tau ^\ell _k),\delta}(\bar{Y}^n_{j-1},dy) + C^k _1 (m) + C^k _2 (m) \\
    & \quad = \sum_{j=n+km + l_0+1}^{n+(k+1)m} \varepsilon_{j}\int_{\mathbb{R}^{d_2}}\int_{\mathbb{R}^{d_2}} f(y,t^n_{j-1})q^{\zeta^*(\tau^\ell _k),\dot{\zeta} ^*(\tau^\ell _k),\delta}(y,dz) \mu^{\zeta^*(\tau^\ell _k),\dot{\zeta} ^*(\tau^\ell _k),\delta}(dy) \\
    & \qquad \quad + C^k _1(m)+ C^k _2(m) + C^k _3(m).
\end{align*}
Because $\mu^{\zeta^*(\tau^\ell _k),\dot{\zeta} ^*(\tau^\ell _k),\delta}$ is invariant for $q^{\zeta^*(\tau^\ell _k),\dot{\zeta} ^*(\tau^\ell _k),\delta}$, we have,
\begin{align*}
    &\int_{\mathbb{R}^{d_2}}\int_{\mathbb{R}^{d_2}} f(y,t^n_{j-1})q^{\zeta^*(\tau^\ell _k),\dot{\zeta} ^*(\tau^\ell _k),\delta}(y,dz) \mu^{\zeta^*(\tau^\ell _k),\dot{\zeta} ^*(\tau^\ell _k),\delta}(dy) \\ & \quad = \int _{\bR ^{d_2}} f(y, t^n _{j-1}) \mu^{\zeta^*(\tau^\ell _k),\dot{\zeta} ^*(\tau^\ell _k),\delta}(dy).
\end{align*}
Moreover, from the definition of $C^k _4 (m)$, we have,
\begin{align*}
    &\sum_{j=n+km + l_0+1}^{n+(k+1)m} \varepsilon_{j} \int _{\bR ^{d_2}} f(y, t^n _{j-1}) \mu^{\zeta^*(\tau^\ell _k),\dot{\zeta} ^*(\tau^\ell _k),\delta}(dy) \\ & \quad = \int _{\tau ^\ell _{k}} ^{\tau ^\ell _{k+1}} \int _{\bR ^{d_2}} f(y, t^n _{j-1}) \mu^{\zeta^*(\tau^\ell _k),\dot{\zeta} ^*(\tau^\ell _k),\delta}(dy)dt + C^k _4 (m)
\end{align*}
Combining the steps above, and the definition of $\gamma$, we can express the difference \eqref{eq:fixk} as,
\begin{align*}
    \sum_{j=n+km + 1}^{n+(k+1)m} \int_{t^n_{j-1}}^{t^n_{j}}\int_{\mathbb{R}^{d_2}}f(y,t)\gamma^m(dy|t)dt - \int _{\bR ^{d_2}} \int _{\tau ^\ell _k} ^{\tau ^\ell _{k+1}} f(y,t) \gamma (dy dt) = \sum _{i=1} ^4 C^k _i (m).
\end{align*}
We now consider the $C^k _i (m)$-terms, for a fixed $k \in \{0,1, \dots, \ell\}$, as we let $m$ go to infinity. 

For $C^k _1 (m)$, because $f$ is bounded, the sum only contains a finite number of terms, and 
\[ 
t^n _{n +km + l_0} - t^n _{n + km +1} \to 0, \ m \to \infty,
\]
we have that $C^k _1 (m) \to 0$.

For $C^k _2 (m)$, we can write this term as, 
\[
C_2 ^k (m) = \!\!\!\sum_{j=n+km + l_0+1} ^{n+(k+1) m} \int_{t^n_{j-1}}^{t^n_{j}}\int_{\mathbb{R}^{d_2}}\left( f(y,t) - f(y, t^n _{j-1}) \right)q^{\zeta^*(\tau^\ell _k),\dot{\zeta} ^*(\tau ^\ell _k),\delta}(\bar{Y}^n_{j-1},dx)dt .
\]
Using the uniform continuity of $f$, these terms can be made arbitrarily small.

Next, for $C^k _3(m)$, arguments analogous to those used in the proof of Lemma \ref{lem:bounded_rel_ent}, based on the $L^1$-ergodic theorem, gives $C^k _3 (m) \to 0$ as $m \to \infty$; due to the similarity with the previous proof, we omit the details.


For $C^k _4 (m)$, we utilise Riemann integrability of the function $\hat f : [0,T] \to \bR$ defined by,
\[
    t \mapsto \hat f (t) = \int _{\bR ^{d_2}} f(y,t) \mu ^{\zeta ^*(\tau ^\ell _k), \dot \zeta ^* (\tau ^\ell _k), \delta} (dy).
\]
Noting that $t^n _{n + (k+1)m} \to \tau ^\ell _{k+1}$ and $t^n _{n+km +l_0 +1} \to \tau ^\ell _k$, 
\[
    \sum _{j=n +km + l_0 +1} ^{n + (k+1)m} \varepsilon _j \int _{\bR ^{d_2}} f(y, t^n _{j-1}) \mu ^{\zeta ^* (\tau ^\ell _k), \dot \zeta ^*(\tau ^\ell _k), \delta } (dy) ,
\]
is a Riemann sum and converges to $\int _{\tau ^\ell _k} ^{\tau ^\ell _{k+1}} \hat f(t) dt$ as $m \to \infty$. Thus $C^k _4 (m) \to 0$ as $m \to \infty$. 

We have established that, for each $k \in \{0, \dots, \ell\}$, $\sum _{i=1} ^4 C^k _i (m) \to 0$, as $m \to \infty$. It follows that,
\[
    \sum _{k=0} ^{\ell} \sum _{i=1} ^4 C^k _i (m) \to 0, \ \ m \to \infty.
\]
By extension, as $m \to \infty$,
\begin{align*}
\sum _{k=0} ^{\ell} \sum_{j=n+km + 1}^{n+(k+1)m} \int_{t^n_{j-1}}^{t^n_{j}}\int_{\mathbb{R}^{d_2}}f(y,t)\gamma^m(dy|t)dt &\to  \sum _{k=0} ^{\ell} \int _{\tau ^\ell _k} ^{\tau ^\ell _{k+1}} \int _{\bR ^{d_2}} f(y,t) \gamma (dy |t) dt \\
&= \int _{0} ^{T} \int _{\bR ^{d_2}} f(y,t) \gamma (dy dt).
\end{align*}
It remains to consider the term, 
\begin{align}
\label{eq:finalTerm}
\int _{t^n _{n + \ell m}} ^T \int _{\bR ^{d_2}} f(y,t ) \gamma ^m (dy |t)dt,
\end{align}
in the limit as $m \to \infty$. Since $f$ is bounded, $\gamma ^m (\cdot |t)$ is a probability measure for each $t \in [0,T]$, and $t^n _{n +\ell m} \to \tau^\ell _\ell = T$, as $m \to \infty$, we have that \eqref{eq:finalTerm} vanishes in this limit. Thus, we have shown that, w.\ p.\ 1, for arbitrary bounded and uniformly continuous $f:\bR^{d_2} \times [0,T] \to \bR$,
\[
\int _{\bR ^{d_2} \times [0, T]} f(y,t)\gamma ^m(dydt) \to  \int_{\bR ^{d_2} \times [0, T]} f(y,t) \gamma (dy dt), \ \ m \to \infty.
\]
That is, w.\ p.\ 1 we have the weak convergence $\gamma ^m \to \gamma$, as $m \to \infty$. This completes the proof.
\end{proof}

The next step is to prove the convergence of $\bar x ^\ell$ when taking $\delta \to 0$ and $\ell \to \infty$, in that order. We have the following result.

\begin{lemma}\label{lem:limit_delta} 
Assume \ref{ass:Lipschitz}-\ref{ass:limith} hold and let $\{\bar x^{\ell}\}$ be the process defined in Lemma \ref{lem:limit_m}. Then, $\{\bar x^{\ell}\}$ converges to $\zeta^*$, on the time interval $[0,\hat{S}]$, in the limit as $\delta\to 0$ and $\ell\to\infty$. 
\end{lemma}
Before embarking on the proof of Lemma \ref{lem:limit_delta}, we show two results that are used in the proof: the integrability of $g(x,\cdot)$ with respect to $\pi _x $, for each $x \in \bR ^{d_1}$ (Lemma \ref{lem:g_integrable}), and that if the process $\bar{x}^l$ converges to $\zeta^*$ on $[0,\hat{S}]$, then it also converges to $\zeta^*$ on $[0,T]$, see Corollary \ref{cor:barX}. 

\begin{lemma} 
\label{lem:g_integrable}
Under \ref{ass:Lipschitz}-\ref{ass:limith}, for any $x\in\mathbb{R}^{d_1}$, the function $y \mapsto g(x,y)$ is integrable with respect to $\pi_x$.
\end{lemma}
\begin{proof}
Since $1+x\leq e^x$ for all $x\in\mathbb{R}$, from \ref{ass:logMGF} we have that for all $x,\alpha\in\mathbb{R}^{d_1}$ and $y \in \bR ^{d_2}$,
\[
    1+\int _{\bR ^{d_2}} \langle\alpha, g(x,z)\rangle \rho_x(y,dz) \leq \sup_y\int _{\bR ^{d_2}} e^{\langle \alpha, g(x,z)\rangle}\rho_x(y,dz) <\infty.
\]
By taking $\alpha$ as the unit vectors $e_i$, for every $i=1, \dots, d_1$, in the last display, the upper bound implies the finiteness of every component of 
\[
    \sup_y \int _{\bR ^{d_2}}  g(x,z) \rho_x(y,dz).
\]
Moreover, since $\pi_x\rho_x = \pi_x$, we have, 
\[
    \int _{\bR ^{d_2}}  g(x,z) \pi_x(dz) = \int _{\bR ^{d_2}} \left(\int _{\bR ^{d_2}}  g(x,z) \rho_x(y,dz)\right) \pi_x(dy).
\]
Therefore, every component of the left integral is finite, which proves the claim.
\end{proof}

\begin{corollary}
\label{cor:barX}
   Assume that $\bar{x}^l$ converges to $\zeta^*$ on $[0,\hat{S}]$, then $\hat{S} = T$. 
\end{corollary}
\begin{proof}
    Assume that $\hat{S}<T$. By the convergence of $\bar{x}^l$ we have that $\|\bar{x}^\ell(\hat{S})- \zeta^*(\hat{S})\|$ can be made arbitrarily small for small enough $\delta$ and large enough $\ell$. The trajectory, $\zeta^*$, is continuous by definition and by Theorem \ref{thm:limit}, $\bar x^\ell$ is continuous on $[0,T]$. From the definition of $\hat{S}$ we have,
    \[
    \lim_{t\to \hat{S}^+}\|x^\ell(t)-\zeta^*(t)\|\geq 1.
    \]
    However this contradicts the continuity of $\zeta^*$ and $x^\ell$ and we conclude that $\hat{S} = T$.
\end{proof}

We now move to the proof of Lemma \ref{lem:limit_delta}. The proof uses arguments similar to those used in Section \ref{sec:limit} to prove Theorem \ref{thm:limit} (specifically, similar to arguments used in the proof of Lemma \ref{lem:1}). For simplicity the proof is done with $\hat{S} = T$, the proof in the case $\hat{S}<T$ is completely analogous. 

\begin{proof}[Proof of Lemma \ref{lem:limit_delta}]
As already noted, by construction of the $\nu ^{\zeta ^* (t), \dot \zeta ^* (t)}$-measures, for all $t \in [0,T]$, $\zeta ^*$ satisfies
\[
    \zeta ^*(t) = x_0 + \int _0 ^t \int _{\bR ^{d_2}} g(\zeta ^*(s), y) \nu ^{\zeta ^* (t), \dot \zeta ^* (t)} (dy) ds,
\]
and Lemma \ref{lem_unique} ensures that the solution is unique. To show the claimed convergence, we consider the difference between $\bar x^\ell$ and $\zeta ^*$:
\begin{align*}
   & \norm{\bar x^\ell - \zeta ^*}_\infty = \sup _{t \in [0,T]} \norm{\bar x ^\ell (t) - \zeta ^*(t)} \\
    &= \sup_{t\in[0,T]}\left\|\sum_{i=0}^{k-1}\int_{\tau^\ell_i}^{\tau^\ell_{i+1}}\int g( \zeta ^*(s),y)\mu^{\zeta ^*(\tau^\ell_i),\dot \zeta ^*(\tau^\ell_i),\delta}(dy)ds + \int_{\tau^\ell_k}^t\int g( \zeta ^*(s),y)\mu^{\zeta ^*(\tau^\ell_k),\dot \zeta ^*(\tau^\ell_k),\delta}(dy)ds\right. \\
    & \left. \qquad \qquad \quad - \int_0^t\int g(\zeta ^*(s),y)\nu^{\zeta ^*(s),\dot{\zeta ^*}(s)}(dy)ds \right\| \\
    &\leq \sup_{t\in[0,T]}\left\|\sum_{i=0}^{k-1}\int_{\tau^\ell_i}^{\tau^\ell_{i+1}}\int g( \zeta ^*(s),y)\mu^{\zeta ^*(\tau^\ell_i),\dot \zeta ^* (\tau^\ell_i),\delta}(dy)ds + \int_{\tau^\ell_k}^t\int g( \zeta ^*(s),y)\mu^{\zeta ^*(\tau^\ell_k),\dot \zeta ^*(\tau^\ell_k),\delta}(dy)ds\right.\\
    &\qquad\qquad \quad\left.- \sum_{i=0}^{k-1}\int_{\tau^\ell_i}^{\tau^\ell_{i+1}}\int g( \zeta ^*(s),y)\nu^{\zeta ^*(\tau^\ell_i),\dot \zeta ^*(\tau^\ell_i)}(dy)ds - \int_{\tau^\ell_k}^t\int g( \zeta ^*(s),y)\nu^{\zeta ^*(\tau^\ell_k),\dot\zeta ^*(\tau^\ell_k)}(dy)ds\right\|\\
    &\qquad +\sup_{t\in[0,T]}\left\| \sum_{i=0}^{k-1}\int_{\tau^\ell_i}^{\tau^\ell_{i+1}}\int g( \zeta ^*(s),y)\nu^{\zeta ^*(\tau^\ell_i),\dot \zeta ^*(\tau^\ell_i)}(dy)ds + \int_{\tau^\ell_k}^t\int g( \zeta ^*(s),y)\nu^{\zeta ^*(\tau^\ell_k),\dot \zeta ^*(\tau^\ell_k)}(dy)ds\right.\\
    &\left. 
    \qquad\qquad\qquad\qquad - \int_0^t\int g(\zeta ^*(s),y)\nu^{\zeta ^*(s),\dot \zeta ^*(s)}(dy)ds
    \right\|.
\end{align*}
We now treat the two suprema in the upper bound separately, and start by considering a fixed but arbitrary $t \in [0,T]$. For the terms inside the first supremum, for any $i \in \{0, 1, \dots, k-1\}$, we have the upper bound,
\begin{align*}
   & \left \| \int_{\tau^\ell_i}^{\tau^\ell_{i+1}}\int g( \zeta ^*(s),y)\mu^{\zeta ^*(\tau^\ell_i),\dot \zeta ^* (\tau^\ell_i),\delta}(dy)ds - \int_{\tau^\ell_i}^{\tau^\ell_{i+1}}\int g( \zeta ^*(s),y)\nu^{\zeta ^*(\tau^\ell_i),\dot \zeta ^*(\tau^\ell_i)}(dy)ds \right \| \\
   & \quad \leq \left\| \int_{\tau^\ell_i}^{\tau^\ell_{i+1}}\int g( \zeta ^*(\tau ^\ell _i),y)\mu^{\zeta ^*(\tau^\ell_i),\dot \zeta ^* (\tau^\ell_i),\delta}(dy)ds - \int_{\tau^\ell_i}^{\tau^\ell_{i+1}}\int g( \zeta ^*(\tau ^\ell _i),y)\nu^{\zeta ^*(\tau^\ell_i),\dot \zeta ^*(\tau^\ell_i)}(dy)ds \right \| \\
   & \qquad + \left \| \int_{\tau^\ell_i}^{\tau^\ell_{i+1}}\int \left ( g( \zeta ^*(\tau ^\ell _i),y) - g(\zeta ^* (s),y) \right)\mu^{\zeta ^*(\tau^\ell_i),\dot \zeta ^* (\tau^\ell_i),\delta}(dy)ds \right \| \\
   & \qquad + \left \| \int_{\tau^\ell_i}^{\tau^\ell_{i+1}}\int \left ( g( \zeta ^*(\tau ^\ell _i),y) - g(\zeta ^* (s),y) \right)\nu^{\zeta ^*(\tau^\ell_i),\dot \zeta ^* (\tau^\ell_i)}(dy)ds \right \|.
\end{align*}
Similarly, for the term involving integrals from $\tau _k ^\ell$ to $t$,
\begin{align*}
    & \left\| \int_{\tau^\ell_k}^t\int g( \zeta ^*(s),y)\mu^{\zeta ^*(\tau^\ell_k),\dot \zeta ^*(\tau^\ell_k),\delta}(dy)ds - \int_{\tau ^\ell _k} ^t\int g(\zeta ^*(s),y)\nu^{\zeta ^*(\tau ^\ell _k),\dot{\zeta ^*}(\tau ^\ell _k)}(dy)ds \right\| \\
    & \quad \leq \left\| \int_{\tau^\ell_k}^t\int g( \zeta ^*(\tau ^\ell _k),y)\mu^{\zeta ^*(\tau^\ell_k),\dot \zeta ^*(\tau^\ell_k),\delta}(dy)ds - \int_{\tau ^\ell _k} ^t\int g(\zeta ^*(\tau ^\ell _k),y)\nu^{\zeta ^*(\tau ^\ell _k),\dot{\zeta ^*}(\tau ^\ell _k)}(dy)ds \right\| \\
    & \qquad + \left\| \int_{\tau^\ell_k}^t\int \left( g( \zeta ^*(\tau ^\ell _k),y) - \int_{\tau ^\ell _k}^t\int g(\zeta ^*(s),y) \right) \mu^{\zeta ^*(\tau^\ell_k),\dot \zeta ^*(\tau^\ell_k),\delta}(dy)ds \right\| \\
    & \qquad + \left\| \int_{\tau^\ell_k}^t\int \left( g( \zeta ^*(\tau ^\ell _k),y) - \int_{\tau ^\ell _k} ^t\int g(\zeta ^*(s),y) \right) \nu^{\zeta ^*(\tau^\ell_k),\dot \zeta ^*(\tau^\ell_k),\delta}(dy)ds \right\|. 
\end{align*}
From the definitions of $\mu ^{\zeta ^*(\tau^\ell_i),\dot \zeta ^* (\tau^\ell_i),\delta}$ and $\nu^{\zeta ^*(\tau^\ell_k),\dot \zeta ^*(\tau^\ell_k),\delta}$, we have that,
\begin{align*}
    & \left\| \int_{\tau^\ell_i}^{\tau^\ell_{i+1}}\int g( \zeta ^*(\tau ^\ell _i),y)\mu^{\zeta ^*(\tau^\ell_i),\dot \zeta ^* (\tau^\ell_i),\delta}(dy)ds - \int_{\tau^\ell_i}^{\tau^\ell_{i+1}}\int g( \zeta ^*(\tau ^\ell _i),y)\nu^{\zeta ^*(\tau^\ell_i),\dot \zeta ^*(\tau^\ell_i)}(dy)ds \right \| \\
    & \quad = | \tau ^\ell _{i+1} - \tau ^\ell _i | \frac{\delta}{2}  \left\| \int g( \zeta ^*(\tau ^\ell _i),y) \pi _{\zeta ^* (\tau ^\ell _i)} (dy) - \int g( \zeta ^*(\tau ^\ell _i),y) \nu^{\zeta ^*(\tau^\ell_i),\dot \zeta ^*(\tau^\ell_i)}(dy)  \right\| \\
    & \quad = | \tau ^\ell _{i+1} - \tau ^\ell _i | \frac{\delta}{2}  \left\| \int g( \zeta ^*(\tau ^\ell _i),y) \pi _{\zeta ^* (\tau ^\ell _i)} (dy) - \dot{\zeta} ^* (\tau ^\ell _i)  \right\|.
\end{align*}
The integral inside the norm is finite by Lemma \ref{lem:g_integrable}. 

Next, by the uniform Lipschitz property for $g$, we have, 
\begin{align*}
    & \left \| \int_{\tau^\ell_i}^{\tau^\ell_{i+1}}\int \left ( g( \zeta ^*(\tau ^\ell _i),y) - g(\zeta ^* (s),y) \right)\mu^{\zeta ^*(\tau^\ell_i),\dot \zeta ^* (\tau^\ell_i),\delta}(dy)ds \right \| \\
    & \quad \leq L_g \int _{\tau ^\ell _i} ^{\tau ^\ell _{i+1}} \norm{\zeta ^*(s) - \zeta ^* (\tau ^\ell _i)}ds.
\end{align*}
In precisely the same way we have,
\begin{align*}
    & \left \| \int_{\tau^\ell_i}^{\tau^\ell_{i+1}}\int \left ( g( \zeta ^*(\tau ^\ell _i),y) - g(\zeta ^* (s),y) \right)\nu^{\zeta ^*(\tau^\ell_i),\dot \zeta ^* (\tau^\ell_i)}(dy)ds \right \| \\
    & \quad \leq L_g \int _{\tau ^\ell _i} ^{\tau ^\ell _{i+1}} \norm{\zeta ^*(s) - \zeta ^* (\tau ^\ell _i)}ds.
\end{align*}
Combining these inequalities yields the upper bound,
\begin{align*}
    & \left \| \int_{\tau^\ell_i}^{\tau^\ell_{i+1}}\int g( \zeta ^*(s),y)\mu^{\zeta ^*(\tau^\ell_i),\dot \zeta ^* (\tau^\ell_i),\delta}(dy)ds - \int_{\tau^\ell_i}^{\tau^\ell_{i+1}}\int g( \zeta ^*(s),y)\nu^{\zeta ^*(\tau^\ell_i),\dot \zeta ^*(\tau^\ell_i)}(dy)ds \right \| \\
    & \quad \leq ( \tau ^\ell _{i+1} - \tau ^\ell _i )  \frac{\delta}{2}  \left\| \int g( \zeta ^*(\tau ^\ell _i),y) \pi _{\zeta ^* (\tau ^\ell _i)} (dy) - \dot{\zeta} ^* (\tau ^\ell _i)  \right\| + 2 L_g \int _{\tau ^\ell _i} ^{\tau ^\ell _{i+1}} \norm{\zeta ^*(s) - \zeta ^* (\tau ^\ell _i)}ds .
\end{align*}

We can use the same arguments as above once more to obtain an upper bound for the term involving integrals from $\tau ^\ell _k$ to $t$:
\begin{align*}
 & \left\| \int_{\tau^\ell_k}^t\int g( \zeta ^*(s),y)\mu^{\zeta ^*(\tau^\ell_k),\dot \zeta ^*(\tau^\ell_k),\delta}(dy)ds - \int_{\tau ^\ell _k} ^t\int g(\zeta ^*(s),y)\nu^{\zeta ^*(\tau ^\ell _k),\dot{\zeta ^*}(\tau ^\ell _k)}(dy)ds \right\| \\
 & \quad \leq ( t - \tau ^\ell _k ) \frac{\delta}{2}  \left\| \int g( \zeta ^*(\tau ^\ell _k),y) \pi _{\zeta ^* (\tau ^\ell _k)} (dy) - \dot{\zeta} ^* (\tau ^\ell _k)  \right\| + 2L_g \int _{\tau ^\ell _k} ^{t} \norm{\zeta ^*(s) - \zeta ^* (\tau ^\ell _k)}ds.
\end{align*}
Combining the upper bounds yields,
\begin{align*}
& \sup_{t\in[0,T]}\left\|\sum_{i=0}^{k-1}\int_{\tau^\ell_i}^{\tau^\ell_{i+1}}\int g( \zeta ^*(s),y)\mu^{\zeta ^*(\tau^\ell_i),\dot \zeta ^* (\tau^\ell_i),\delta}(dy)ds \right. 
\\ & \qquad + \int_{\tau^\ell_k}^t\int g( \zeta ^*(s),y)\mu^{\zeta ^*(\tau^\ell_k),\dot \zeta ^*(\tau^\ell_k),\delta}(dy)ds\\
&\qquad - \sum_{i=0}^{k-1}\int_{\tau^\ell_i}^{\tau^\ell_{i+1}}\int g( \zeta ^*(s),y)\nu^{\zeta ^*(\tau^\ell_i),\dot \zeta ^*(\tau^\ell_i)}(dy)ds  
\\ & \qquad \left.- \int_{\tau^\ell_k}^t\int g( \zeta ^*(s),y)\nu^{\zeta ^*(\tau^\ell_k),\dot\zeta ^*(\tau^\ell_k)}(dy)ds\right\|\\ 
& \quad \leq \sup _{t \in [0,T]} \left\{ \sum _{i=1} ^{k-1} \left( ( \tau ^\ell _{i+1} - \tau ^\ell _i ) \frac{\delta}{2}  \left\| \int g( \zeta ^*(\tau ^\ell _i),y) \pi _{\zeta ^* (\tau ^\ell _i)} (dy) - \dot{\zeta} ^* (\tau ^\ell _i)  \right\| \right.\right.\\ & \qquad \qquad \left. + 2L_g \int _{\tau ^\ell _i} ^{\tau ^\ell _{i+1}} \norm{\zeta ^*(s) - \zeta ^* (\tau ^\ell _i)}ds \right)  \\
& \qquad + \left. ( t - \tau ^\ell _k ) \frac{\delta}{2}  \left\| \int g( \zeta ^*(\tau ^\ell _k),y) \pi _{\zeta ^* (\tau ^\ell _k)} (dy) - \dot{\zeta} ^* (\tau ^\ell _k)  \right\|  + 2L_g \int _{\tau ^\ell _k} ^{t} \norm{\zeta ^*(s) - \zeta ^* (\tau ^\ell _k)}ds \right\},
\end{align*}
where the value for $k$ depends on $t$. 

To deal with the supremum over $t$, we note that increasing $t$ will only add more non-negative terms, and the terms corresponding to time-differences will be maximal for $t = T$. This results in $k = \ell + 1$ and we have, 
\begin{align*}
& \sup_{t\in[0,T]}\left\|\sum_{i=0}^{k-1}\int_{\tau^\ell_i}^{\tau^\ell_{i+1}}\int g( \zeta ^*(s),y)\mu^{\zeta ^*(\tau^\ell_i),\dot \zeta ^* (\tau^\ell_i),\delta}(dy)ds + \int_{\tau^\ell_k}^t\int g( \zeta ^*(s),y)\mu^{\zeta ^*(\tau^\ell_k),\dot \zeta ^*(\tau^\ell_k),\delta}(dy)ds\right.\\
    &\qquad \left.- \sum_{i=0}^{k-1}\int_{\tau^\ell_i}^{\tau^\ell_{i+1}}\int g( \zeta ^*(s),y)\nu^{\zeta ^*(\tau^\ell_i),\dot \zeta ^*(\tau^\ell_i)}(dy)ds - \int_{\tau^\ell_k}^t\int g( \zeta ^*(s),y)\nu^{\zeta ^*(\tau^\ell_k),\dot\zeta ^*(\tau^\ell_k)}(dy)ds\right\|\\ 
    & \quad \leq  \sum _{i=1} ^{\ell} \left( (\tau ^\ell _{i+1} - \tau ^\ell _i ) \frac{\delta}{2}  \left\| \int g( \zeta ^*(\tau ^\ell _i),y) \pi _{\zeta ^* (\tau ^\ell _i)} (dy) - \dot{\zeta} ^* (\tau ^\ell _i)  \right\| + 2 L_g \int _{\tau ^\ell _i} ^{\tau ^\ell _{i+1}} \norm{\zeta ^*(s) - \zeta ^* (\tau ^\ell _i)}ds \right),
\end{align*}
where we have also used that $\tau ^\ell _{\ell+1} = T$. 

For the second supremum, we split it according to,
\begin{align*}
& \sup_{t\in[0,T]}\left\| \sum_{i=0}^{k-1}\int_{\tau^\ell_i}^{\tau^\ell_{i+1}}\int g( \zeta ^*(s),y)\nu^{\zeta ^*(\tau^\ell_i),\dot \zeta ^*(\tau^\ell_i)}(dy)ds + \int_{\tau^\ell_k}^t\int g( \zeta ^*(s),y)\nu^{\zeta ^*(\tau^\ell_k),\dot \zeta ^*(\tau^\ell_k)}(dy)ds\right.\\
    & \left. \qquad - \int_0^t\int g(\zeta ^*(s),y)\nu^{\zeta ^*(s),\dot \zeta ^*(s)}(dy)ds   \right\| \\
    &\quad \leq \sup_{t \in [0,T]} \Bigg\{ \Bigg\| \sum_{i=0}^{k-1}\int_{\tau^\ell_i}^{\tau^\ell_{i+1}}\int g( \zeta ^*(s),y)\nu^{\zeta ^*(\tau^\ell_i),\dot \zeta ^*(\tau^\ell_i)}(dy)ds \\ & \qquad - \sum_{i=0}^{k-1}\int_{\tau^\ell_i}^{\tau^\ell_{i+1}}\int g( \zeta ^*(\tau ^\ell _i),y)\nu^{\zeta ^*(\tau^\ell_i),\dot \zeta ^*(\tau^\ell_i)}(dy)ds \\
    & \qquad  + \int_{\tau^\ell_k}^t\int g( \zeta ^*(s),y)\nu^{\zeta ^*(\tau^\ell_k),\dot \zeta ^*(\tau^\ell_k)}(dy)ds - \int_{\tau^\ell_k}^t\int g( \zeta ^*(\tau ^\ell _k),y)\nu^{\zeta ^*(\tau^\ell_k),\dot \zeta ^*(\tau^\ell_k)}(dy)ds \Bigg\|    \\
    & \qquad + \Bigg\| \sum_{i=0}^{k-1}\int_{\tau^\ell_i}^{\tau^\ell_{i+1}}\int g( \zeta ^*(\tau ^\ell _i),y)\nu^{\zeta ^*(\tau^\ell_i),\dot \zeta ^*(\tau^\ell_i)}(dy)ds \\ & \qquad + \int_{\tau^\ell_k}^t\int g( \zeta ^*(\tau ^\ell _k),y)\nu^{\zeta ^*(\tau^\ell_k),\dot \zeta ^*(\tau^\ell_k)}(dy)ds  \\
    &  \qquad - \int_0^t\int g(\zeta ^*(s),y)\nu^{\zeta ^*(s),\dot \zeta ^*(s)}(dy)ds \Bigg\| \Bigg\}. 
\end{align*}
Similar to the above, we start by treating the terms inside the supremum to obtain suitable upper bounds. In this direction the second norm-term is the easiest to treat. From the definitions of the $\nu ^{\zeta ^* (\tau ^\ell _i), \dot{\zeta} ^* (\tau ^\ell _i)}$-measures and the properties of $\zeta ^*$,
\begin{align*}
& \left\| \sum_{i=0}^{k-1}\int_{\tau^\ell_i}^{\tau^\ell_{i+1}}\int g( \zeta ^*(\tau ^\ell _i),y)\nu^{\zeta ^*(\tau^\ell_i),\dot \zeta ^*(\tau^\ell_i)}(dy)ds + \int_{\tau^\ell_k}^t\int g( \zeta ^*(\tau ^\ell _k),y)\nu^{\zeta ^*(\tau^\ell_k),\dot \zeta ^*(\tau^\ell_k)}(dy)ds \right. \\
    & \left. \qquad   - \int_0^t\int g(\zeta ^*(s),y)\nu^{\zeta ^*(s),\dot \zeta ^*(s)}(dy)ds \right\| \\
    & \quad =  \left\| \sum_{i=0}^{k-1} (\tau ^\ell _{i+1} - \tau ^\ell _i) \dot{\zeta} ^* (\tau ^\ell _i) + (t - \tau ^\ell _k) \dot{\zeta} ^* (\tau ^\ell _k)- \zeta ^* (t) \right\|.
\end{align*}
This term will converge to 0 uniformly in $t$ as $l$ grows, due to the properties of $\zeta ^*$ (see Lemma \ref{lem:piecewise_conti}).

For the other term inside the supremum, we have the upper bound,
\begin{align*}
&\left\| \sum_{i=0}^{k-1}\int_{\tau^\ell_i}^{\tau^\ell_{i+1}}\int g( \zeta ^*(s),y)\nu^{\zeta ^*(\tau^\ell_i),\dot \zeta ^*(\tau^\ell_i)}(dy)ds - \sum_{i=0}^{k-1}\int_{\tau^\ell_i}^{\tau^\ell_{i+1}}\int g( \zeta ^*(\tau ^\ell _i),y)\nu^{\zeta ^*(\tau^\ell_i),\dot \zeta ^*(\tau^\ell_i)}(dy)ds \right.\\
    & \qquad \left. + \int_{\tau^\ell_k}^t\int g( \zeta ^*(s),y)\nu^{\zeta ^*(\tau^\ell_k),\dot \zeta ^*(\tau^\ell_k)}(dy)ds - \int_{\tau^\ell_k}^t\int g( \zeta ^*(\tau ^\ell _k),y)\nu^{\zeta ^*(\tau^\ell_k),\dot \zeta ^*(\tau^\ell_k)}(dy)ds \right\|  \\
    & \quad \leq \sum _{i=0} ^{k-1} \int_{\tau^\ell_i}^{\tau^\ell_{i+1}}\int \left\| g( \zeta ^*(s),y) - g(\zeta ^* (\tau ^\ell _i),y) \right\|\nu^{\zeta ^*(\tau^\ell_i),\dot \zeta ^*(\tau^\ell_i)}(dy)ds \\
    &\qquad \quad + \int _{\tau ^\ell _k} ^t \int \left\| g(\zeta^* (s),y) - g(\zeta^* (\tau ^\ell _k), y) \right\| \nu^{\zeta ^*(\tau^\ell_i),\dot \zeta ^*(\tau^\ell_i)}(dy)ds \\
    & \quad \leq \sum _{i=0} ^{k-1} L_g \int _{\tau _i ^\ell} ^{\tau ^\ell _{i+1}} \norm{\zeta ^* (s) - \zeta ^* (\tau ^\ell _i)} ds +L_g  \int _{\tau ^\ell _k} ^t \norm{g(\zeta^* (s)) - g(\zeta ^* (\tau ^\ell _k))} ds.
\end{align*}
Similar to before, we see that the supremum is achieved at $t=T$, and thus $k=\ell$. Together with the preceding calculations this yields the upper bound,
\begin{align*}
 \norm{\bar x^\ell - \zeta ^*}_\infty &= \sup _{t \in [0,T]} \norm{\bar x ^\ell (t) - \zeta ^*(t)} \\
  &  \leq \sum _{i=1} ^{\ell-1} \left( (\tau ^\ell _{i+1} - \tau ^\ell _i ) \frac{\delta}{2}  \left\| \int g( \zeta ^*(\tau ^\ell _i),y) \pi _{\zeta ^* (\tau ^\ell _i)} (dy) - \dot{\zeta} ^* (\tau ^\ell _i)  \right\| \right. \\
  & \qquad \qquad \left. + 3 L_g \int _{\tau ^\ell _i} ^{\tau ^\ell _{i+1}} \norm{\zeta ^*(s) - \zeta ^* (\tau ^\ell _i)}ds \right). 
\end{align*}
Note that for any $\ell$, by sending $\delta$ to 0, we have,
\begin{align*}
    \sum _{i=1} ^{\ell-1} (\tau ^\ell _{i+1} - \tau ^\ell _i ) \frac{\delta}{2}  \left\| \int g( \zeta ^*(\tau ^\ell _i),y) \pi _{\zeta ^* (\tau ^\ell _i)} (dy) - \dot{\zeta} ^* (\tau ^\ell _i)  \right\| \to 0.
\end{align*}
Next, by the uniform continuity of $\zeta ^*$, for $\ell$ large enough, we have that for any $\tilde \delta>0$, 
\[
\max_{i\in\{0,\ell-1\}}\sup_{s\in [\tau^\ell_{i+1},\tau^\ell_i]}\|\zeta^*(s)-\zeta^*(\tau ^\ell _i)\|<\tilde \delta.
\]
Using the expression in the last display, we obtain, 
\[
\sum _{i=1} ^{\ell-1} 3 L_g \int _{\tau ^\ell _i} ^{\tau ^\ell _{i+1}} \norm{\zeta ^*(s) - \zeta ^* (\tau ^\ell _i)}ds < \sum _{i=1} ^{\ell-1} 3\tilde \delta L_g(\tau ^\ell _{i+1} - \tau ^\ell _{i}) = 3\tilde \delta L_g T,
\]
where the right-hand side can be made arbitrarily small.
\end{proof}
We now have all results needed to prove Theorem \ref{thm:convX} in place. Although a detailed outline of the proof has been given throughout this section, we conclude by collecting the steps in a brief formal proof.
\begin{proof}[Proof of Theorem \ref{thm:convX}]
    By Lemma \ref{lem:bounded_rel_ent}, the controlled measures $\{ \bar \nu ^n\}$ have bounded expected running cost. The conditions of Theorem \ref{thm:limit} are thus satisfied, and it follows that every subsequence of $(\bar \nu ^n, \bar X ^n)$ has a convergent subsequence. Furthermore, the corresponding limit point $(\bar \nu, \bar X)$ is characterized by Theorem \ref{thm:limit} and it remains to show that this limit is of the claimed form. By Lemmas \ref{lem:limit_m}-\ref{lem:limit_delta} and Corollary \ref{cor:barX}, and noting the role of $\ell$ and $m$, and their asymptotics, as $m \to \infty$, $\delta \to 0$ and $\ell \to \infty$,  $\bar X ^n$ converges in probability to $\zeta ^*$. Thus, $\bar X = \zeta ^*$. Moreover, these results combined with the characterization from Theorem \ref{thm:limit} also give the form for $\bar \nu$, as we have that $\zeta ^*$ is the unique solution (see Lemma \ref{lem_unique}) to \eqref{eq:ODEzeta} : it holds that $\bar \nu (\cdot|t)= \nu ^{\zeta ^* (t), \dot \zeta ^* (t)} (\cdot)$. This completes the proof.
\end{proof}

\subsection{Proof of Laplace lower bound}
\label{sec:proofLower}
With the results of Sections \ref{sec:optimalControl}--\ref{sec:conv}, we can now complete the proof of the Laplace lower bound of Theorem \ref{thm:lower}. We start by proving an ancillary result that is used in the proof. 

In Lemma \ref{lem:bounded_rel_ent} we proved that the expected running cost associated with $\{ \bar \nu ^n \} $ is bounded by showing that the sums over the two terms in the alternative representation \eqref{eqn_relative_entropy} are bounded. Using Lemmas  \ref{lem:limit_m} and \ref{lem:limit_delta}, we can also show that the sum over the second term appearing in \ref{eqn_relative_entropy} is negligible in the limit under consideration.
\begin{lemma}\label{lem:negligible}
Assume \ref{ass:Lipschitz}-\ref{ass:limith} hold. Then, 
\begin{align*}
    \limsup_{\ell\to\infty} \limsup_{\delta\to 0}\limsup_{m\to\infty}E\Bigg[\frac{1}{\beta_n} \sum_{k=0}^{{\ell}} \sum_{j=0}^{m-1}\int &\left(\log\frac{d\rho_{\zeta^*(\tau^\ell_k)}(\bar Y^n_{n+km+j},\cdot)}{d\rho_{\bar X^n_{n+km+j}}(\bar Y^n_{n+km+j},\cdot)}\right) \\ & \times q^{\zeta^*(\tau^\ell_k),\dot{\zeta}^*(\tau^\ell_k),\delta}(\bar{Y}^n_{n+km+j},dy)\Bigg] = 0.
\end{align*}
\end{lemma}
\begin{proof}
Recall that $\rho_x(y,dz)=\eta_x(y,z)\lambda(dz)$ where $\eta_x(y,z)$ is continuous in $x$, uniformly in $y,z$. This implies that 
\[
    \log\frac{d\rho_x(y,\cdot)}{d\rho_w(y,\cdot)} = \log \frac{\eta_x(y,\cdot)}{\eta_w(y,\cdot)} \to 0 \mbox{ as }x\to w.
\]
It now suffices to prove that $\{\bar{X}^n\}$ converges to $\zeta^*(\tau^\ell_k)$, which holds by Lemmas \ref{lem:limit_m}--\ref{lem:limit_delta}.
\end{proof}

\begin{proof}[Proof of Theorem \ref{thm:lower}]
With the continuity of $F$, Lemma \ref{lem:negligible} and Lemma \ref{lem:limit_m}, we have that for any $\varepsilon>0$, we can take $\delta>0$ and $\ell<\infty$ small and large enough, respectively, so that we have the upper bound,
\begin{align*}
    &\limsup_{m\to\infty}-\frac{1}{\beta_n}\log E e^{-\beta_n F(X^n)}\\
    &\quad\leq   \limsup_{m\to\infty}E\left[F(\bar X^n) + \frac{1}{m(\ell+1)} \sum_{k=0}^{{\ell}}\sum_{j=0}^{m-1} R\left( \bar \nu^n_{km+j+1}(\cdot)\| \rho_{\bar{X}^n_{km+j}}(\bar Y^n_{km+j},\cdot)\right)\right]\\
    &\leq \limsup_{m\to\infty}E\left[F(\bar X^n) + \frac{1}{m(\ell+1)} \sum_{k=0}^{{\ell}}\sum_{j=l_0}^{m-1} R\left(q^{\zeta^*(\tau^\ell_k),\dot{\zeta}^*(\tau^\ell_k),\delta}(\bar Y^n_{n+km+j},\cdot)\| \rho_{\zeta^*(\tau^\ell_k}(\bar Y^n_{km+j},\cdot)\right)\right] + \varepsilon\\
    &\quad\leq   E\left[ F(\hat X^\ell) + \frac{1}{\ell+1} \sum_{k=0}^{{\ell}}\int R\left( q^{\zeta^*(\tau^\ell_k),\dot{\zeta}^*(\tau^\ell_k),\delta}(y,\cdot)\| \rho_{\zeta^*(\tau^\ell_k)}(y,\cdot)\right)\mu^{\zeta^*(\tau^\ell_k),\dot{\zeta}^*(\tau^\ell_k),\delta}(dy)\right] + \varepsilon. 
\end{align*}
From Lemma \ref{lem:optimal_control},
\begin{align*}
    &E\left[  \int R\left( q^{\zeta^*(\tau^\ell_k),\dot{\zeta}^*(\tau^\ell_k),\delta},(y,\cdot)\| \rho_{\zeta^*(\tau^\ell_k)}(y,\cdot)\right)\mu^{\zeta^*(\tau^\ell_k),\dot{\zeta}^*(\tau^\ell_k),\delta}(dy)\right]\\
    &\quad= E\left[ R\left( \mu^{\zeta^*(\tau^\ell_k),\dot{\zeta}^*(\tau^\ell_k),\delta}\otimes q^{\zeta^*(\tau^\ell_k),\dot{\zeta}^*(\tau^\ell_k),\delta}(\cdot,\cdot)\| \mu^{\zeta^*(\tau^\ell_k),\dot{\zeta}^*(\tau^\ell_k),\delta}\otimes \rho_{\zeta^*(\tau^\ell_k)}(\cdot,\cdot)\right) \right]
    \\
    &\quad\leq E\left[L\left( \zeta^*(\tau^\ell_k),\dot{\zeta}^*(\tau^\ell_k)\right)\right] + \varepsilon.
\end{align*}
Thus, 
\begin{align*}
    \limsup_{m\to\infty}-\frac{1}{\beta_n}\log E e^{-\beta_n F(X^n)}\leq E\left[ F(\hat X^{\ell}) + \frac{1}{\ell+1} \sum_{k=0}^{{\ell}} L\left( \zeta^*(\tau^\ell_k),\dot{\zeta}^*(\tau^\ell_k)\right)\right]+2\varepsilon.
\end{align*}
We can rewrite the sum inside the expectation as a Riemann sum,
\begin{align*}
    \frac{1}{\ell+1} \sum_{k=0}^{{\ell}} L\left( \zeta^*(\tau^\ell_k),\dot{\zeta}^*(\tau^\ell_k)\right) &=  \sum_{k=0}^{{\ell}} \frac{1}{\ell(\tau^\ell_{k+1}-\tau^\ell_k)}L\left( \zeta^*(\tau^\ell_k),\dot{\zeta}^*(\tau^\ell_k)\right)(\tau^\ell_{k+1}-\tau^\ell_k).
\end{align*}
From \ref{ass:limith}, 
\begin{align*}
    \ell(\tau^\ell_{k+1}-\tau^\ell_k) &= \lim_{n\to\infty} \sum_{i = n + km+1}^{n+(k+1)m} \varepsilon_i\ell \\
    &= \lim_{n\to\infty} \sum_{i = n + km+1}^{n+(k+1)m} \varepsilon_i\beta_n\frac{\ell}{\beta_n} \\
    &\geq \lim_{n\to\infty}\varepsilon_{n+(k+1)m}\beta_n \\
    &= h(\tau^\ell_{k+1}).
\end{align*}
It follows that $\frac{1}{\ell(\tau^\ell_{k+1}-\tau^\ell_k)}$ is bounded from above by $\frac{1}{h(\tau^\ell_{k+1})}$.

To finish the proof, we note that because $\delta$ is arbitrary, $F$ is continuous, $\zeta^*$ is piecewise linear with finitely many pieces, see Lemma \ref{lem:piecewise_conti}, using Lemma \ref{lem:limit_delta} and \eqref{eq:boundFI}, we have, 
\begin{align*}
    &\limsup_{\ell\to\infty}\limsup_{m\to\infty}-\frac{1}{\beta_n}\log E e^{-\beta_n F(X^n)}\\
    &\quad\leq \limsup_{\ell\to\infty}E\left[ \limsup_{\delta\to 0}F(\hat X^{\ell}) +  \sum_{k=0}^{{\ell}-1} \frac{1}{h(\tau^\ell_{k+1})} L\left( \zeta^*(\tau^\ell_k),\dot{\zeta}^*(\tau^\ell_k)\right)(\tau^l_{k+1}-\tau^\ell_k)\right] + 2\varepsilon\\
    &\quad \leq F(\zeta^*) + \int_0^T \frac{1}{h(t)}L(\zeta^*(t),\dot{\zeta}^*(t))dt +2\varepsilon\\
    &\quad = F(\zeta^*)+I(\zeta^*) + 2\varepsilon \\
    & \quad \leq \inf_{\varphi}(F(\varphi)+I(\varphi)) + 3\varepsilon.
\end{align*}
Because $\varepsilon >0$ is arbitrary, and noting the role of $\ell$ and $m$ and their asymptotics, this proves the upper bound,
\[
    \limsup_{n\to\infty}-\frac{1}{\beta_n}\log E \left[e^{-\beta_n F(X^n)}\right] \leq \inf_{\varphi}(F(\varphi)+I(\varphi)).
\]
This completes the proof of Theorem \ref{thm:lower}.

\end{proof}
\section{Proof of Theorem \ref{thm:limit}}
\label{sec:limit}
In this section we carry out the proof of Theorem \ref{thm:limit}, the convergence result for $(\tilde\nu ^n, \tilde X ^n)$ when $\{ \tilde \nu ^n \}$ is a generic sequence of control measures satisfying bounded expected running cost. The first step is proving the following uniform integrability property; the tightness of $\{ \tilde \nu ^n \}$ is an immediate consequence of Lemma \ref{lem:ui}.
\begin{lemma}\label{lem:ui}
Under \ref{ass:Lipschitz}-\ref{ass:limith}, if $\{\tilde \nu^n\}$ has bounded running cost, i.e.,
\begin{equation}
    \sup_nE\left[\frac{1}{\beta_n}\sum_{i=n}^{\beta_n+n-1}R(\tilde\nu_i^n(\cdot)||\rho_{\tilde{X}_i^n}(\bar{Y}^n_i,\cdot))\right]<\infty,
\end{equation}
then it satisfies the uniform integrability property,
\begin{equation}\label{eqn:ui}
    \lim_{C\to\infty}\sup_n E\left[\int_0^T\int_{\|g(\tilde X^n(t),z)\|>C}\|g(\tilde{X}^n(t),z)\|\tilde \nu^n(dz\times dt)\right] = 0.
\end{equation}
\end{lemma}
\begin{proof}
The proof uses the inequality 
\[
ab\leq e^{\sigma a} + \frac{1}{\sigma}(b\log(b) - b+1),
\]
with $a = \|g(\tilde X ^n (t),z)\|$ and $b = \frac{d \tilde \nu^n_i(\cdot)}{d\rho_{\tilde{X}^n_i}(\tilde{Y}^n_i,\cdot)}$. For $t \in [0,T]$, and fixed $C$ and $n$, we have, 
\begin{align*}
    &\int_{\|g (\tilde X ^n (t), z)\|>C} \|g (\tilde X ^n (t), z)\|d \tilde\nu^n_i(dz) \\
    & \quad = \int_{\|g (\tilde X ^n (t), z)\|>C} \|g (\tilde X ^n (t), z)\|\frac{d\tilde\nu^n_i(z)}{d\rho_{\tilde{X}^n_i}(\tilde{Y}^n_i,z)}\rho_{\tilde{X}^n_i}(\tilde{Y}^n_i,dz)\\
    & \quad \leq \int_{\|g (\tilde X ^n (t), z)\|>C} e^{\sigma\|g (\tilde X ^n (t), z)\|} \rho_{\tilde{X}^n_i}(\tilde{Y}^n_i,dz) \\
    &\qquad  + \frac{1}{\sigma}\int_{\|g (\tilde X ^n (t), z)\|>C}\left(\frac{d\tilde\nu^n_i(z)}{d\rho_{\tilde{X}^n_i}(\tilde{Y}^n_i,z)}\log \left(\frac{d\tilde\nu^n_i(z)}{d\rho_{\tilde{X}^n_i}(\tilde{Y}^n_i,z)}\right) -\frac{d\tilde\nu^n_i(z)}{d\rho_{\tilde{X}^n_i}(\tilde{Y}^n_i,z)} +1  \right)\rho_{\tilde{X}^n_i}(\tilde{Y}^n_i,dz)\\
    &\quad \leq  \int_{\|g (\tilde X ^n (t), z)\|>C} e^{\sigma\|g (\tilde X ^n (t), z)\|} \rho_{\tilde{X}^n_i}(\tilde{Y}^n_i,dz) + \frac{1}{\sigma}R(\tilde\nu^n_i(\cdot)||\rho_{\tilde{X}^n_i}(\tilde{Y}^n_i,\cdot))\\
    &\quad \leq e^{-\sigma C}\sup_x\sup_y \int e^{2\sigma \|g (x, z)\|}\rho_x(y,dz) + \frac{1}{\sigma}R(\tilde\nu^n_i(\cdot)||\rho_{\tilde{X}^n_i}(\tilde{Y}^n_i,\cdot)).
\end{align*}
The first term in the last display is finite by \ref{ass:logMGF}. This upper bound now bound yields an upper bound on the corresponding expectation when we also integrate over time:
\begin{align*}
    &E\left[\int_0^T\int_{\|g(\tilde X^n(t),z)\|>C}\|g(\tilde{X}^n(t),z)\|\tilde \nu^n(dz\times dt)\right]\\
    &\leq E\left[\sum_{i=n}^{\beta_n+n-1} \int_{t_i,t_{i+1}} \frac{1}{h_n(t)} e^{-\sigma C}\sup_x\sup_y \int e^{2\sigma \|z\|}\rho_x(y,dz) + \frac{1}{h_n(t)}\frac{1}{\sigma}R(\tilde\nu^n_i(\cdot)||\rho_{\tilde{X}^n_i}(\tilde{Y}^n_i,\cdot))dt\right] \\
    &=\leq  e^{-\sigma C}\sup_x\sup_y \int e^{2\sigma \|z\|}\rho_x(y,dz) + \frac{1}{\sigma}E\left[\frac{1}{\beta_n}\sum_{i=n}^{\beta_n+n-1}R(\tilde\nu^n_i(\cdot)||\rho_{\tilde{X}^n_i}(\tilde{Y}^n_i,\cdot))\right].
\end{align*}
From \eqref{eq:finite_control}, the relative-entropy term is bounded in $n$. Therefore, since the first term has no dependence on $n$, taking $C\to \infty$, followed by $\sigma \to \infty$ yields the claimed convergence. 
\end{proof}

The proof of Theorem \ref{thm:limit} is dived into several steps and the arguments follow closely those used in \cite [Section 5.3]{dupell4}. As already mentioned, the tightness of $\{\tilde\nu^n\}$ follows from Lemma \ref{lem:ui}. Next, we consider the stochastic process $S^n$ defined as
\begin{equation}
\label{eq:Sn}
    S^n(t) = x + \int_0^t \int_{\mathbb{R}^{d_2}}g(\hat{X}^n(s),y)\tilde\nu^n(dy|s)ds,
\end{equation}
where $\hat{X}^n(t)$ is the piecewise constant function that takes the values $\hat{X}^n(t_{n+k}-t_n) = \tilde{X}^n_k$. This intermediate process $S^n$ is used to bridge the gap between $\tilde{X}^n$ and the limit $\tilde{X}$. 

A detailed outline of the remainder of the proof is as follows. First, in Lemma \ref{mcon} we prove that $\{S^n\}$ is tight. Next, in Lemma \ref{lem_3} we show the convergence
\begin{align}
\label{eq:SXlim}
   P\left(\sup_{t\in[0,T]}\|\tilde{X}^n(t)-S^n(t)\|>\varepsilon\right) = 0 \text{ for any }\varepsilon>0.
\end{align}
Together with Lemma \ref{mcon}, Lemma \ref{lem_3} then gives tightness of $\{\tilde{X}^n\}$. By Prohorov's theorem $\{ S^n \}$ and $\{ \tilde X ^n \}$ both have convergent subsequences, with some limit $S$. The final step in the proof, which is carried out in Lemma \ref{lem:1}, is to show that w.\ p.\ 1, $S = \tilde X$, where $\tilde{X}$ is defined in \eqref{eqn:x_bar}.  


\begin{lemma}
\label{mcon}
Define the modulus of continuity of $S^n$ as, for any $\delta >0$,
\begin{equation*}
    w^n(\delta) = \sup_{|s-t|<\delta}\|S^n(t)-S^n(s)\|.
\end{equation*}
The following statements hold,  
\begin{enumerate}
    \item[(a)] for all $\varepsilon>0$ and $\eta>0$ there exists a $\delta>0$ such that $P(w^n(\delta)>\varepsilon)<\eta$ for all $n$.
    \item[(b)] $S^n$ is tight.
\end{enumerate}
\end{lemma}
\begin{proof}
For part (a), given any $\varepsilon>0$ and $\eta>0$, from \eqref{eqn:ui} we can choose $C>0$ such that,
\begin{align*}
 &\sup_n E\left[\int_0^T\int_{\|g(S^n(t),y)\|>C}\|g(\hat X^n(t),y)\|\tilde\nu^n(dy\times dt)\right] 
    \leq \frac{\eta\varepsilon}{2e^T}.
\end{align*}
With $\delta = \varepsilon/({2C})$, Markov's inequality yields,
\begin{align*}
    &P\left(w^n(\delta)>\varepsilon\right) = P\left( \sup_{|s-t|<\delta}\|S^n(t)-S^n(s)\|>\varepsilon\right) \\
    & \quad \leq P\left(\sup_{|s-t|<\delta}\int_s^t\int_{\mathbb{R}^{d_2}} \|g(\hat X^n(r),y)\|\tilde\nu^n(dy|r)dr>\varepsilon\right)\\
    &\quad \leq P\left(\sup_{|s-t|<\delta}\int_s^t\int_{\|g(\hat X^n(r),y)\|>C} \|g(\hat X^n(r),y)\|\tilde\nu^n(dy|r)dr\right.\\
    &\qquad\qquad+ \left.\sup_{|s-t|<\delta}\int_s^t\int_{\|g(\hat X^n(r),y)\|\leq C}\|g(\hat X^n(r),y)\|\tilde\nu^n(dy|r)dr>\varepsilon\right)\\
    &\quad \leq P\left(\sup_{|s-t|<\delta}\int_s^t\int_{\|g(\hat X^n(r),y)\|>C} \|g(\hat X^n(r),y)\|\frac{h^n(r)}{h^n(r)}\tilde\nu^n(dy|r)dr + C\delta>\varepsilon\right) \\
    &\quad = P\left(\sup_{|s-t|<\delta}\int_s^t\int_{\|g(\hat X^n(r),y)\|>C} \|g(\hat X^n(r),y)\|h^n(r)\tilde\nu^n(dy\times dr) >\frac{\varepsilon}{2T}\right)\\
    &\quad \leq P\left(\int_0^T\int_{\|g(\hat X^n(r),y)\|>C} \|g(\hat X^n(r),y)\|\tilde\nu^n(dy\times dr) >\frac{\varepsilon}{2e^T}\right)\\
    &\quad \leq\frac{2e^T}{\varepsilon}E\left[\int_0^T\int_{\|g(\hat X^n(r),y)\|>C} \|g(\hat X^n(r),y)\|\tilde\nu^n(dy\times dr)\right] \\
    & \quad < \eta.
\end{align*}
This establishes (a). For (b), because of (a) and that $S^n(0)=x$ for all $n$, the tightness of $\{S^n\}$ follows from Theorem A.3.22 in \cite{dupell4}. 
\end{proof}
\begin{lemma}\label{lem_3}
For any $\varepsilon>0$,
\begin{equation*}
    P\left(\sup_{t\in[0,T]}\|S^n(t)-\tilde{X}^n(t)\|>\varepsilon\right) \to 0,
\end{equation*}
as $n\rightarrow\infty$.
\end{lemma}
\begin{proof}
We use the notation $t^n_j = t_{j}-t_n$. Because $\tilde{X}^n$ is the piecewise linear interpolation of the random vector $\{\tilde{X}^n_j\}=\{\tilde{X}^n(t^n_j)\}$, we have, 
\begin{align*}
   \sup_{t\in [0,T]}\|S^n(t)-\tilde{X}^n(t)\| 
   &\leq \max_{k\in J}\sup_{t\in[t^n_k,t^n_{k+1}]}\|S^n(t)-\tilde{X}^n(t)\| \\
   & \leq\max_{k\in J} w^n(\varepsilon^n_k)+ \max_{k\in J}\|S^n(t^n_k)-\tilde{X}^n(t^n_k)\|\\
    &\leq w^n(\varepsilon_{n})+ \max_{k\in J }\|S^n(t^n_k)-\tilde{X}^n(t^n_k)\|,
\end{align*}
where $J\doteq\{n,\dots, m(T+t_n) \}$. Lemma \ref{mcon} implies that $w^n(\varepsilon_n)\overset{p}{\to} 0$. Hence, it suffices to show that $\max_{k\in J }\|S^n(t^n_k)-\tilde{X}^n(t^n_k)\| \overset{p}{\to} 0$, as $n\to \infty$. By Markov's inequality, For any $\varepsilon>0$, 
\begin{align*}
    P\left(\max_{k\in J }\|S^n(t^n_k)-\tilde{X}^n(t^n_k)\|>\varepsilon\right)
    \leq \frac{1}{\varepsilon}E\left[\max_{k\in J }\|S^n(t^n_k)-\tilde{X}^n(t^n_k)\|\right].
\end{align*}
It remains to show that the expectation on the right-hand side of the last display converges to $0$. Given $\theta>0$, we define a variable $\Lambda^n_j$ as a truncation of $\varepsilon_{j}g(X^n_j,\bar{Y}^n_j)$:
\begin{align*}
    \Lambda^n_j = \begin{cases}\tilde{X}^n_{j+1}-\tilde{X}^n_j \quad&\text{if } \|\tilde{X}^n_{j+1}-\tilde{X}^n_j\|<\theta, \\
    0 \quad&\text{if } \|\tilde X^n_{j+1}-\tilde X^n_j\|\geq\theta. \end{cases}
\end{align*}
Note that because $\tilde{X}^n_{j+1} = \tilde{X}^n_j +\varepsilon_{j} g(\tilde{X}^n_j,\Bar{Y}^n_j)$, this is indeed a truncated version of $\varepsilon_{j}g(X^n_j,\bar{Y}^n_j)$. From the definitions of $S^n$ and $\tilde{X}^n$, and using that $\tilde\nu^n(dy|s)=\tilde\nu^n_j(dy)$ for $s\in[t^n_j,t^n_{j+1})$, we have,
\begin{align}
     & E\left[\max_{k\in J }\|S^n(t^n_k)-\tilde{X}^n(t^n_k)\|\right] \nonumber \\ 
      &\leq E\left[\max_{k\in J }\left\|x + \sum_{j=n}^{k-1}\varepsilon_{j}\int_{\mathbb{R}^{d_2}}g(\hat X^n(t^n_j),y)\tilde\nu^n_j(dy) -\tilde{X}^n(t^n_k)\right\|\right] \nonumber \\ 
    &\leq E\left[\max_{k\in J }\left\|x + \sum_{j=n}^{k-1}\Lambda^n_j -\tilde{X}^n(t^n_k)\right\|\right] \label{eq:Bnd1}  \\
    & \quad + E\left[\max_{k\in J }\left\| \sum_{j=n}^{k-1}\left(\Lambda^n_j -\varepsilon_{j}\int_{\mathbb{R}^{d_2}}g(\tilde{X}^n_j,y)1_{\{\|\varepsilon_{j}g(\tilde{X}^n_j,y)\|<\theta \}}\tilde\nu^n_j(dy) \right)\right\|\right]\label{eq:Bnd2} 
\\
    & \;\; +\!E \! \! \left[\! \max_{k\in J }\! \left\| \sum_{j=n}^{k-1}\! \varepsilon_{j}\! \left(\int_{\mathbb{R}^{d_2}}g(\tilde{X}^n_j,y)1_{\{\|\varepsilon_{j}\! \! g(\tilde{X}^n_j,y)\|<\theta \}}\tilde\nu^n_j(dy) \! -\! \! \int_{\mathbb{R}^{d_2}}\! \! g(\bar X^n_j,y)\tilde\nu^n_j(dy)\right)\! \right\|\right]\! . \label{eq:Bnd3}
\end{align}
Therefore, to complete the proof, it suffices to show that the three terms on the right-hand side of the inequality converge to $0$ as $n\rightarrow\infty$. 

For the first and third terms, \eqref{eq:Bnd1} and \eqref{eq:Bnd3}, we show that both are bounded from above by $c^n(\theta)$, defined as, 
\begin{equation*}
    c^n(\theta) \doteq E\left[\int_0^T\int_{\mathbb{R}^{d_2}}\|g(\tilde{X}^n_j,y)\|1_{\{\|\varepsilon_{j}g(\tilde{X}^n_j,y)\|>\theta\}}\tilde\nu^n(dy|t)dt\right].
\end{equation*}
In the limit as $n \to \infty$, we have $c^n(\theta) \to 0$ because of the uniform integrability property, see Lemma \ref{lem:ui}, and $\varepsilon_{j}\to 0$. 
For the first term \eqref{eq:Bnd1}, we write $\tilde{X}^n(t^n_k)-x$ as a telescoping sum,
\begin{align*}
    & E\left[\max_{k\in J }\left\|x + \sum_{j=n}^{k-1}\Lambda^n_j -\tilde{X}^n(t^n_k)\right\|\right]
    \\ & \quad \leq E\left[\sum_{j=n}^{m(T+t_n)}\|\tilde{X}^n(t^n_{j+1})- \tilde{X}^n(t^n_{j})-\Lambda_j^n\|\right]\\
    & \quad \leq E\left[\sum_{j=n}^{\beta_n+n}\varepsilon_{j}\|g(\tilde{X}^n_j,\bar{Y}^n_j)\|1_{\{\|\varepsilon_{j}g(\tilde{X}^n_j,\bar{Y}^n_j)\|\geq\theta\}}\right]\\
    &\quad = E\left[\sum_{j=n}^{\beta_n+n}\int_{t_j^n}^{t^n_{j+1}}\int_{\mathbb{R}^{d_2}}\|g(\tilde{X}^n_j,y)\|1_{\{\|\varepsilon_{j} g(\tilde{X}^n_j,y)\|\geq\theta\}}\tilde\nu^n_j(dy)dt\right]\\
    &\quad \leq c^n(\theta).
\end{align*}
For \eqref{eq:Bnd3}, 
\begin{align*}
     &E\left[\max_{k\in J }\left\| \sum_{j=n}^{k-1}\varepsilon_{j}\left(\int_{\mathbb{R}^{d_2}}g(\tilde{X}^n_j,y)1_{\{\|\varepsilon_{j}g(\tilde{X}^n_j,y)\|<\theta \}}\tilde\nu^n_j(dy)-\int_{\mathbb{R}^{d_2}}g(\tilde{X}^n_j,y)\tilde\nu^n_j(dy) \right)\right\|\right]\\
     &\quad\leq E\left[\sum_{j=n}^{m(T+t_n)} \varepsilon_{j}\int_{\mathbb{R}^{d_2}}\| g(\tilde{X}^n_j,y)\|1_{\{\|\varepsilon_{j}g(\tilde{X}^n_j,y)\|\geq\theta \}}\tilde\nu^n_j(dy)\right]\\
     &\quad = E\left[\sum_{j=n}^{\beta_n+n}\int_{t_j^n}^{t^n_{j+1}}\int_{\mathbb{R}^{d_2}}\|g(\tilde{X}^n_j,y)\|1_{\{\|\varepsilon_{j} g(\tilde{X}^n_j,y)\|\geq\theta\}}\tilde\nu^n(dy|t)dt\right]\\
     &\quad \leq c^n(\theta).
\end{align*}

For \eqref{eq:Bnd2}, the second term in the upper bound for $E\left[ \max_{k\in J }\|S^n(t^n_k)-\tilde{X}^n(t^n_k)\|  \right]$, define the sequence $\{M^n_j\}$ by,
\begin{equation*}
    M^n_{j+1}\doteq \Lambda_j^n - \varepsilon_{j}\int_{\mathbb{R}^{d_2}}g(\tilde{X}^n_j,y)1_{\{\|\varepsilon_{j}g(\tilde{X}^n_j,y)\|<\theta\}}\tilde\nu^n_j(dy).
\end{equation*}
This is a martingale difference sequence with respect to the $\sigma$-algebra $\bar{\mathcal{F}}^n_j\doteq \sigma(\tilde{X}^n_n,\dots,\tilde{X}^n_{n+j})$. Indeed, by the definition of $\tilde\nu^n_j$,
\begin{align*}
    E[\Lambda_j^n|\bar{\mathcal{F}}^n_j] &= E\left[(\tilde{X}^n_{j+1}-\tilde{X}^n_j)1_{\|\tilde{X}^n_{j+1}-\tilde{X}^n_j\|< \theta}|\bar{\mathcal{F}}^n_j\right]\\
    &= \varepsilon_{j} E\left[g(X^n_j,\bar{Y}^n_j)1_{\{\|\varepsilon_{j}g(X^n_j,\bar{Y}^n_j)\|< \theta\}}\bar{\mathcal{F}}^n_j\right]\\
    &= \varepsilon_{j}\int_{\mathbb{R}^{d_2}}g(\tilde{X}^n_j,y)1_{\{\|\varepsilon_{j}g(\tilde{X}^n_j,y)\|<\theta\}}\tilde\nu^n_j(dy).
\end{align*}
Therefore, $\{(\sum_{j=n}^{k}M^n_{j},\bar{\mathcal{F}}^n_{k-1}) \}$ is a martingale and for any $i\neq j$, $E\left[\left\langle  M^n_i,M^n_j\right\rangle\right]=0$. Moreover, we can rewrite \eqref{eq:Bnd2} as,
\[
    E\left[\max_{k\in J }\left\| \sum_{j=n}^{k-1}\left(\Lambda^n_j -\varepsilon_{j}\int_{\mathbb{R}^{d_2}}g(\tilde{X}^n_j,y)1_{\{\|\varepsilon_{j}g(\tilde{X}^n_j,y)\|<\theta \}}\tilde\nu^n_j(dy) \right)\right\|\right]
    = E\left[\max_{k\in J }\left\| \sum_{j=n}^{k-1}M^n_{j+1}\right\|\right].
\]
For the expression in the last display, we have the upper bound,
\begin{align*}
    E\left[\max_{k\in J }\left\| \sum_{j=n}^{k-1}M^n_{j+1}\right\|\right] &\leq \left(E\left[\max_{k\in J }\left\| \sum_{j=n}^{k-1}M^n_{j+1}\right\|^2\right]\right)^{1/2} \\
    &\leq 2 \left(E\left[\left\| \sum_{j=n}^{\beta_n +n}M^n_{j+1}\right\|^2\right]\right)^{1/2},
\end{align*}
where the first inequality comes from the Cauchy-Schwarz inequality and the second from Doob's submartingale inequality. Furthermore, because,
\begin{align*}
    E[\Lambda_j^n|\bar{\mathcal{F}}^n_j] = \varepsilon_{j}\int_{\mathbb{R}^{d_2}}g(\tilde{X}^n_j,y)1_{\{\|\varepsilon_{j}g(\tilde{X}^n_j,y)\|<\theta\}}\tilde\nu^n_j(dy),
\end{align*}
and $\|\Lambda_j^n\|\leq \theta$, we have the upper bound,
\begin{align*}
    E\left[\left\| \sum_{j=n}^{\beta_n+n}M^n_{j+1}\right\|^2\right]
    &=\sum_{j=n}^{\beta_n+n}E\left[\left\| \Lambda_j^n - \varepsilon_{j}\int_{\mathbb{R}^{d_2}}g(\tilde{X}^n_j,y)1_{\{\|\varepsilon_{j}g(\tilde{X}^n_j,y)\|<\theta\}}\tilde\nu^n_j(dy)\right\|^2\right]\\
    &\leq \sum_{j=n}^{\beta_n+n}E\left[\| \Lambda_j^n \|^2\right]
    \leq\theta\sum_{j=n}^{\beta_n+n}E\left[\| \Lambda_j^n \|\right]\\
    &\leq \theta\sum_{j=n}^{\beta_n+n}E\left[ \int_{t_j^n}^{t^n_{j+1}}\int_{\mathbb{R}^{d_2}}\|g(\tilde{X}_j^n,y)\|\tilde\nu^n(dy|t)dt\right]\\
    &\leq \theta E\left[ \int_{0}^{T}\int_{\mathbb{R}^{d_2}}\|g(\tilde{X}_j^n,y)\|h^n(t)\tilde\nu^n(dy\times dt)\right].
\end{align*}
The expectation in the last display is bounded by a constant, independent of $n$, due to the uniform integrability property. Therefore, this term vanishes as $\theta\to 0$. 

Combining the bounds for \eqref{eq:Bnd1}--\eqref{eq:Bnd3}, we conclude that,
\begin{align*}
    E\left[\max_{k\in J }\|S^n(t^n_k)-\tilde{X}^n(t^n_k)\|\right] \leq 2c^n (\theta) + \theta E\left[ \int_{0}^{T}\int_{\mathbb{R}^{d_2}}\|g(\tilde{X}_j^n,y)\|h^n(t)\tilde\nu^n(dy\times dt)\right].
\end{align*}
The proof is completed by sending $n \to \infty$ followed by $\theta\to 0$.
\end{proof}

\begin{lemma}\label{lem:1}
Let $S^n$ and $\tilde{X}$ be defined by \eqref{eq:Sn} and \eqref{eqn:x_bar}, respectively. Then, for any $\varepsilon>0$,
\begin{equation*}
    P\left(\sup_{t\in[0,T]}\|S^n(t)-\tilde{X}(t)\|>\varepsilon\right) \to 0,
\end{equation*}
as $n\rightarrow\infty$.
\end{lemma}
\begin{proof}
By Markov's inequality, for any $\varepsilon>0$,
\begin{align*}
    P\left(\sup_{t\in[0,T]}\|S^n(t)-\tilde{X}(t)\|>\varepsilon\right) 
    \leq \frac{1}{\varepsilon}E\left[\sup_{t\in[0,T]}\|S^n(t)-\tilde{X}(t)\|\right].
\end{align*}
It remains to show that the expectation on the right-hand side converges to $0$. To show this, we notice that,
\begin{align*}
    &E\left[\sup_{t\in[0,T]}\|S^n(t)-\tilde{X}(t)\|\right]\nonumber\\
&=E\left[\sup_{t\in[0,T]}\left\|\int_{0}^t\int_{\mathbb{R}^{d_2}} g(\hat X^n(s),y)d\tilde\nu^n(y|s)ds - \int_{0}^t\int_{\mathbb{R}^{d_2}}g(\tilde{X}(s),y)d\tilde\nu(y|s)ds\right\|\right] \nonumber\\
    &\leq E\left[\sup_{t\in[0,T]}\left\|\int_{0}^t\int_{\mathbb{R}^{d_2}}g(\hat X^n(s),y)d\tilde\nu^n(y|s)ds - \int_{0}^t\int_{\mathbb{R}^{d_2}}g(\hat X^n,y)d\tilde\nu(y|s)ds\right\|\right]\nonumber\\
&\quad+E\left[\sup_{t\in[0,T]}\int_{0}^t\int_{\mathbb{R}^{d_2}}\left\|g(\hat X^n,y) -  g(\tilde{X}(s),y)\right\|d\tilde\nu(y|s)ds\right]\nonumber\\
    &\leq E\left[\sup_{t\in[0,T]}\int_{\mathbb{R}^{d_2}}\|g(\hat X^n,y)\|\|h^n(s)-h(s)\|\tilde\nu^n(d\times ds)\right]\nonumber\\  &\quad+E\left[\sup_{t\in[0,T]}\left\|\int_{\mathbb{R}^{d_2}\times[0,t]} g(\hat X^n,y)h(s)\tilde\nu^n(dy\times ds) - \int_{\mathbb{R}^{d_2}\times[0,t]} g(\hat X^n,y)h(s)\tilde\nu(dy\times ds)\right\|\right]\nonumber\\
    &\quad+E\left[\sup_{t\in[0,T]}\int_{0}^t \int_{\mathbb{R}^{d_2}}K\|\hat X^n(s)-\tilde{X}(s)\|d\tilde\nu(y|s)ds\right]\nonumber\\
    &\leq E\left[\sup_{t\in[0,T]}\int_{\mathbb{R}^{d_2}}\|g(\hat X^n,y)\|\|h^n(s)-h(s)\|\tilde\nu^n(dy\times ds)\right] \\ 
&\quad+E\left[\sup_{t\in[0,T]}\left\|\int_{\mathbb{R}^{d_2}\times[0,t]} g(\hat X^n,y)h(s)\tilde\nu^n(dy\times ds) - \int_{\mathbb{R}^{d_2}\times[0,t]} g(\hat X^n,y)h(s)\tilde\nu(dy\times ds)\right\|\right] \\ 
    &\quad+E\left[\sup_{t\in[0,T]}\int_{0}^t \int_{\mathbb{R}^{d_2}}K\|\hat X^n(s)-S^n(s)\|d\tilde\nu(y|s)ds\right] \\
    &\quad+E\left[\sup_{t\in[0,T]}\int_{0}^t \int_{\mathbb{R}^{d_2}}K\|S^n(s)-\bar X(s)\|d\tilde\nu(y|s)ds\right].\nonumber
\end{align*}
We define $D_i (n)$, $i=1,2,3$, by
\begin{align}
\label{eqn:bounds_1}
    D_1 (n) &= E\left[\sup_{t\in[0,T]}\int_{\mathbb{R}^{d_2}}\|g(\hat X^n,y)\|\|h^n(s)-h(s)\|\tilde\nu^n(dy\times ds)\right],\\
\label{eqn:bounds_2}
    D_2 (n) &= E\left[\sup_{t\in[0,T]}\left\|\int_{\mathbb{R}^{d_2}\times[0,t]} g(\hat X^n,y)h(s)\tilde\nu^n(dy\times ds)\right.\right. \\ & \qquad \qquad \left.\left. - \int_{\mathbb{R}^{d_2}\times[0,t]} g(\hat X^n,y)h(s)\tilde\nu(dy\times ds)\right\|\right], \nonumber
\end{align}
and
\begin{align}
\label{eqn:bounds_3}
    D_3(n) = E\left[\sup_{t\in[0,T]}\int_{0}^t \int_{\mathbb{R}^{d_2}}K\|\hat X^n(s)-S^n(s)\|d\tilde\nu(y|s)ds\right].
\end{align}
With these definitions, we have the upper bound,
\begin{align*}
    E\left[\sup_{t\in[0,T]}\|S^n(t)-\tilde{X}(t)\|\right] &\leq \sum _{i=1} ^3 D_i(n) + E\left[\sup_{t\in[0,T]}\int_{0}^t \int_{\mathbb{R}^{d_2}}K\|S^n(s)-\bar X(s)\|d\tilde\nu(y|s)ds\right].
\end{align*}
By Gr\"onwall's inequality we have,
\begin{align*}
    E\left[\sup_{t\in[0,T]}\|S^n(t)-\tilde{X}(t)\|\right]\leq E\left[e^{KT} \sum _{i=1} ^3 D_i(n)\right].
\end{align*}
It remains to prove that the $D_i$ terms converge to $0$. 

First, $D_3(n)$ converges to zero due to Lemma \ref{lem_3}. Next, we consider $D_1(n)$, given in \eqref{eqn:bounds_1}. Due to the uniform integrability property \eqref{eqn:ui}, there exists some $C>0$ such that,
\[
    \sup_n E\left[\int_0^T\int_{\|g(\hat X^n(t),y)\|>C}\|g(x,y)\|\tilde\nu^n(dy\times dt)\right]
    \leq 1.
\]
With this $C>0$, we have,
\begin{align*}
    D_1 (n) &\leq E\left[\int_0^T\int_{\|g(\hat X^n(s),y)\|\leq C} \|g(\hat X^n(s),y)\|\|h^n(s)-h(s)\|\tilde\nu^n(dy\times ds)\right]\\
    &\quad +E\left[\int_0^T\int_{\|g(\hat X^n(s),y)\|> C} \|g(\hat X^n(s),y)\|\|h^n(s)-h(s)\|\tilde\nu^n(dy\times ds)\right]\\
    & \leq (CT+1) \sup_{t\in[0,T]}\|h^n(t)-h(t)\|.
\end{align*}
Thus, $D_1(n)$ converges to $0$ due to the uniform converge of $\{h^n\}$ to $h$ as $n\rightarrow\infty$, see Assumption \ref{ass:limith}. 

Finally, we consider $D_2(n)$ in \eqref{eqn:bounds_2}. Note that, for any $C>0$ and $n_0>0$,
\begin{align*}
    D_2(n) &\leq E\Bigg[\sup_{t\in[0, T]}\Bigg\|\int_{ \mathbb{R}^{d_2}\times[0,t]} g(\tilde{X}^{n_0}(s),y)h(s)\tilde\nu^n(dy\times ds) \\ & \qquad\qquad  - \int_{ \mathbb{R}^{d_2}\times[0,t]} g(\tilde{X}^{n_0}(s),y)h(s)\tilde\nu(dy\times ds)\Bigg\|\Bigg]\\
    &\quad +E\Bigg[\sup_{t\in[0, T]}\Bigg\|\int_{ \mathbb{R}^{d_2}\times[0,t]} g(\hat{X}^{n}(s),y)h(s)\tilde\nu^n(dy\times ds)\\ & \qquad\qquad  - \int_{ \mathbb{R}^{d_2}\times[0,t]} g(\tilde{X}^{n_0}(s),y)h(s)\tilde\nu^n(dy\times ds)\Bigg\|\Bigg]\\
     & \quad + E\Bigg[\sup_{t\in[0, T]}\Bigg\|\int_{ \mathbb{R}^{d_2}\times[0,t]} g(\tilde{X}^{n_0}(s),y)h(s)\tilde\nu(dy\times ds)\\  & \qquad\qquad - \int_{ \mathbb{R}^{d_2}\times[0,t]} g(\hat{X}^{n}(s),y)h(s)\tilde\nu(dy\times ds)\Bigg\|\Bigg]\\
        &\leq E\Bigg[\sup_{t\in[0, T]}\Bigg\|\int_{ \mathbb{R}^{d_2}\times[0,t]} g(\tilde{X}^{n_0}(s),y)h(s)\tilde\nu^n(dy\times ds)  \\ & \qquad\qquad- \int_{ \mathbb{R}^{d_2}\times[0,t]} g(\tilde{X}^{n_0}(s),y)h(s)\tilde\nu(dy\times ds)\Bigg\|\Bigg]\\
    & \quad +E\left[\sup_{t\in[0, T]}\int_{ \mathbb{R}^{d_2}\times[0,t]} \|\hat X^n(s) - \bar X^{n_0}(s)\|h(s)\tilde\nu^n(dy\times ds)\right]\\
    & \quad +E\left[\sup_{t\in[0, T]}\int_{ \mathbb{R}^{d_2}\times[0,t]} \|\bar X^{n_0}(s) - \hat X^{n_0}(s)\|h(s)\tilde\nu(dy\times ds)\right]
    \end{align*}
    Because $\tilde{X}^n$ and $\hat{X}^n$ converge to the same process, as $n \to \infty$, the last two terms in the last display can be made arbitrarily small by choosing large enough $n$ and $n_0$. For the first term in the upper bound, we split it up into two parts, one that is bounded where we can use the weak convergence of $\tilde\nu^n$ and one part that can be made arbitrarily small due to uniform integrability property:
    \begin{align*}
    &E\left[\sup_{t\in[0, T]}\left\|\int_{ \mathbb{R}^{d_2}\times[0,t]} g(\tilde{X}^{n_0}(s),y)h(s)\tilde\nu^n(dy\times ds) - \int_{ \mathbb{R}^{d_2}\times[0,t]} g(\tilde{X}^{n_0}(s),y)h(s)\tilde\nu(dy\times ds)\right\|\right]\\
    &\leq E\left[\sup_{t\in[0, T]}\left\|\int_{ \mathbb{R}^{d_2}\times[0,t]} \left(g(\tilde{X}^{n_0}(s),y)1_{\{\|g(\tilde{X}^{n_0}(s),y)\|\leq C\}}\right.\right.\right.\\
    &\qquad\qquad\qquad\qquad \qquad\left.+C\frac{g(\tilde{X}^{n_0}(s),y)}{\|g(\tilde{X}^{n_0}(s),y)\|} 1_{\{\|g(\tilde{X}^{n_0}(s),y)\|> C\}} \right)h(s)\tilde\nu^n(dy\times ds)\\
    &\quad\quad\qquad \qquad - \int_{ \mathbb{R}^{d_2}\times[0,t]} \left(g(\tilde{X}^{n_0}(s),y)1_{\{\|g(\tilde{X}^{n_0}(s),y)\|\leq C\}}\right.\\
    &\qquad\qquad\qquad\qquad \qquad \quad  \left.\left.\left.+C\frac{g(\tilde{X}^{n_0},y)}{\|g(\tilde{X}^{n_0}(s),y)\|} 1_{\{\|g(\tilde{X}^{n_0}(s),y)\|> C\}} \right)h(s)\tilde\nu(dy\times ds)\right\|\right]\\
    &+E\left[\sup_{t\in[0, T]}\left\|\int_{ \mathbb{R}^{d_2}\times[0,t]} \left(g(\tilde{X}^{n_0}(s),y)-C\frac{g(\tilde{X}^{n_0}(s),y)}{\|g(\tilde{X}^{n_0}(s),y)\|}\right)1_{\{\|g(\tilde{X}^{n_0}(s),y)\|> C\}}h(s)\tilde\nu^n(dy\times ds)\right.\right.\\
    &\qquad \left.\left.- \int_{ \mathbb{R}^{d_2}\times[0,t]} \left(g(\tilde{X}^{n_0}(s),y)-C\frac{g(\tilde{X}^{n_0}(s),y)}{\|g(\tilde{X}^{n_0}(s),y)\|}\right)1_{\{\|g(\tilde{X}^{n_0}(s),y)\|> C\}}h(s)\tilde\nu(dy\times ds)\right\|\right].
\end{align*}
To simplify the expressions, we define
\[
    \varphi_C(x)\doteq x1_{\{\|x\|\leq C\}}+C\frac{x}{\|x\|}1_{\{\|x\|>C\}},
\]
and 
\begin{align*}
    \ell^n(t) &\doteq  \int_{ \mathbb{R}^{d_2}\times[0,t]} \varphi_C(g(\tilde{X}^{n_0}(s),y))h(s)\tilde\nu^n(dy\times ds) \\ & \qquad - \int_{ \mathbb{R}^{d_2}\times[0,t]} \varphi_C(g(\tilde{X}^{n_0}(s),y))h(s)\tilde\nu(dy\times ds).
\end{align*}
We can then write the first expectation in the upper bound in the previous display as, 
\[
E[\sup_{t\in [0,T]}\|\ell^n(t)\|].
\] 
Moreover, using that $h(s) \leq h(0), \forall s \in [0,T]$, and
\[
    \left\|x - C\frac{x}{\|x\|}\right\|  1_{\{\|x\|>C\}}\leq (\|x\|+C)1_{\{\|x\|>C\}}\leq 2\|x\|1_{\{\|x\|>C\}},
\]
for all $x$, combined with the uniform integrability property \eqref{eqn:ui}, we obtain the upper bound,
    \begin{align*}
    &E\Bigg[\sup_{t\in[0, T]}\Bigg\|\int_{ \mathbb{R}^{d_2}\times[0,t]} g(\tilde{X}^{n_0}(s),y)h(s)\tilde\nu^n(dy\times ds) \\ & \qquad \qquad - \int_{ \mathbb{R}^{d_2}\times[0,t]} g(\tilde{X}^{n_0}(s),y)h(s)\tilde\nu(dy\times ds)\Bigg\|\Bigg]\\
    &\leq E[\sup_{t\in [0,T]}\|\ell^n(t)\|] + 4 h(0)  \sup_n E\left[\int_0^T\int_{\|g(\bar X^{n_0}(t),y)\|>C}\|g(x,y)\|\tilde\nu^n(dy\times dt)\right].
\end{align*}
For any $t\in[0,T]$, because  $\tilde\nu(\mathbb{R}^{d_2}\times\{t\})=0$,  $\tilde\nu^n$ converges weakly to $\tilde\nu$ w.\ p.\ 1. Moreover, $\varphi_C$ is bounded and continuous, and thus $ \ell^n(t)\rightarrow 0$ w.\ p.\ 1 as $n\rightarrow\infty$. Without loss of generality, we assume $\ell^n(t,\omega)\rightarrow 0$ for all $\omega\in\Omega$ and $t\in [0,T]$. Now for any fixed $\omega\in \Omega$, it is not hard to see that $\{\ell^n(t,\omega):t\in [0,T]\}_{n\in\mathbb{N}}$ is equicontinuous and uniformly bounded by $2C$. By the Arzel\`{a}-Ascoli theorem, since $\ell^n(t,\omega)\rightarrow 0$, we can conclude that  $\sup_{t\in [0,T]}\|\ell^n(t,\omega)\|\rightarrow 0$ for that given $\omega$; here we use that if every subsequence has a further subsequence which converges uniformly to the same limit, then the whole sequence converges uniformly to the same limit. Because $\omega$ is arbitrary, this means that  $\sup_{t\in [0,T]}\|\ell^n(t)\|\rightarrow 0$ w.\ p.\ 1. Moreover, becasue $\sup_{t\in [0,T]}\|\ell^n(t)\|\leq 2C$, Lebesgue's dominated convergence theorem ensures that $E[\sup_{t\in [0,T]}\|\ell^n(t)\|] \to 0$. Lastly. we have that the second term in the previous display, 
\[
    4 h(0)  \sup_n E\left[\int_0^T\int_{\|g(\bar X^{n_0}(t),y)\|>C}\|g(x,y)\|\tilde\nu^n(dy\times dt)\right],
\]
converges to $0$ by sending $C\rightarrow\infty$. It follows that $D_2(n)$ converges to 0 as well. This concludes the proof.

\end{proof}


\begin{appendix}
\section{Proof that $I$ is a rate function}
In order for Theorem \ref{thm:main} to be a Laplace principle, we must verify that the function $I$ defined in \ref{eqn_rate_function} is indeed a rate function. This is done in the following lemma. 
\begin{lemma}
\label{lem:rateFcn}
Under Conditions \ref{sufficient_conditions}, the function $I:C[0,T]\to [0,\infty]$ defined by,
\[
    I(\varphi) = \int_0^T \frac{1}{h(t)}L(\varphi(t),\dot\varphi(t))dt,
\]
where 
\[
    L(x,\beta) =  \inf_{\mu\in\mathcal{P}(\mathbb{R}^{d_2})}\left\{ \inf_{\eta\in\mathcal{A}(\mu)}\{R(\eta\|\mu \otimes \rho_x(\cdot,\cdot))\}: \beta = \int_{\mathbb{R}^{d_2}} g(x,y)\mu(dy)\right\},
\]
is a rate function, i.e., $I$ has compact level sets.
\end{lemma}
In order to prove Lemma \ref{lem:rateFcn}, we use the following result. Note that the uniqueness proved in Lemma \ref{lem_unique} is also used in Section \ref{sec:conv} for proving Theorem \ref{thm:convX} (specifically in proving Lemma \ref{lem:limit_delta}). 
\begin{lemma}
\label{lem_unique}
The equation 
\[
    \phi(t)  = x_0 + \int_0^t\int g(\phi(s),y)\nu^{\zeta(s),\dot{\zeta}(s)}(dy)ds,
\]
has a unique solution.
\end{lemma}
\begin{proof}
Let $\phi^1$ and $\phi^2$ be solutions to the ODE
\[
    \phi^1(t)  = x_0 + \int_0^t\int g(\phi^1(s),y)\nu^{\zeta(t),\dot{\zeta}(t)}(dy)ds,
\]
\[
    \phi^2(t)  = x_0 + \int_0^t\int g(\phi^2(s),y)\nu^{\zeta(s),\dot{\zeta}(s)}(dy)ds,
\]
and let $\Delta$ be such that $\Delta K<1$, where $K$ is the Lipschitz constant to $g$. Then we will prove that for $t\in [0,\Delta]$, $\phi^1(t) = \phi^2(t)$:
\begin{align*}
    &\sup_{t\in[0,\Delta]}||\phi^1(t)-\phi^2(t)||  \\ & \quad =  \sup_{t\in[0,\Delta]}||\int_0^t\int g(\phi^1(s),y)\nu^{\zeta(s),\dot{\zeta}(s)}(dy)ds - \int_0^t\int g(\phi^2(s),y)\nu^{\zeta(s),\dot{\zeta}(s)}(dy)ds||\\
    &\quad \leq \int_0^\Delta\int \sup_{t\in [0,\Delta]}||g(\phi^1(t),y)-g(\phi^2(t),y)||\nu^{\zeta(s),\dot{\zeta}(s)}(dy)ds \\
    &\quad \leq  \int_0^\Delta\int K\sup_{t\in[0,\Delta]}||\phi^1(t)-\phi^2(t)||\nu^{\zeta(s),\dot{\zeta}(s)}(dy)ds \\ & \quad = K\Delta\sup_{t\in[0,\Delta]}||\phi^1(t)-\phi^2(t)||.
\end{align*}
This is a contraction and the same procedure can be iterated arbitrary number of times, leading to
\begin{equation*}
    \sup_{t\in[0,\Delta]}||\phi^1(t)-\phi^2(t)||\leq (K\Delta)^N \sup_{t\in[0,\Delta]}||\phi^1(t)-\phi^2(t)|| \to 0,\quad N\to \infty.
\end{equation*}
Now we need to extend this to $t\in[0,T]$. For $t\in [0,2\Delta]$ we can use the above argument to obtain
\begin{align*}
    \sup_{t\in[0,2\Delta]}||\phi^1(t)-\phi^2(t)|| &= \sup_{t\in[0,2\Delta]}||\phi^1(t)-\phi^1(\Delta)-(\phi^2(t)-\phi^2(\Delta))|| \\&= \sup_{t\in[\Delta,2\Delta]}||\phi^1(t)-\phi^1(\Delta)-(\phi^2(t)-\phi^2(\Delta))|| \\&
    \leq \Delta K \sup_{t\in[\Delta,2\Delta]}||\phi^1(t)-\phi^2(t)|| \leq \Delta K \sup_{t\in[0,2\Delta]}||\phi^1(t)-\phi^2(t)||
\end{align*}
The same argument as above can now be applied to show that $\phi^1(t) = \phi^2(t)$ for $t\in [0,2\Delta]$. Repeating this procedure yields the result for $t\in [0,T]$.
\end{proof}

\begin{proof}[Proof of Lemma \ref{lem:rateFcn}]
We must show that, for any $K < \infty$, any sequence of functions $\{\phi^j\} \subset C([0,T]:\mathbb{R}^{d_1})$ such that $I(\phi^j)\leq K$ has a convergent subsequence, where the corresponding limit $\phi$ satisfies $I(\phi)\leq K$. To this end, note that for any $\varepsilon$ and $j=1,2,\dots$, there exist a probability measure $ \mu^j(dy\times dt) = \mu^j(dy|t)dt$ and a transition kernel $q^j(y,dz|t)$ such that
\begin{align*}
    \int_0^T \frac{1}{h(t)}\int_{\mathbb{R}^{d_2}} R(q^j(y,\cdot|t)||\rho_{\phi^j(t)}(y,\cdot))\mu^j(dy|t)dt &\leq I(\phi^j) + \varepsilon,\\
    \int_{\mathbb{R}^{d_2}} g(\phi^j(t),y)\mu^j(dy|t) &= \dot{\phi}^j(t),\\
    \mu^jq^j &= \mu^j,
\end{align*}
In order to show tightness for $\{ \mu ^j \}$, we prove that the sequence is uniformly integrable. For this, it is enough to prove that 
\begin{equation*}
    \lim_{m \to \infty} \limsup_{j \to \infty}\int_0^T\int_{\mathbb{R}^{d_2}}\int_{\|z\|>M} \|z\|q^j(y,dz|t)\mu^j(dy\times dt) = 0
\end{equation*}
Using the inequality $ab \leq e^{\sigma a} + \frac{1}{\sigma}(b\log(b)-b+1)$, we have
\begin{align*}
   &\int_0^T\int_{\mathbb{R}^{d_2}}\int_{\|z\|>M} \|z\|q^j(y,dz)\mu^j(dy\times dt) \\ & \quad  \leq \int_0^T \int_{\mathbb{R}^{d_2}} \int_{\|z\|>M}e^{\sigma\|z\|}\rho_{\phi^j(t)}(y,dz)\mu^j(dy|t)dt \\
    & \qquad + \frac{1}{\sigma}\int_0^T R(q^j(y,dz|t)\mu^j(dy|t)||\rho_{\phi^j(t)}(y,dz)\mu^j(dy|t))dt\\ 
   & \quad \leq \int_0^T \sup_y \int_{\|z\|>M}e^{\sigma\|z\|}\rho_{\phi^j(t)}(y,dz)dt + \frac{1}{\sigma}(K+\varepsilon)\\
   & \quad \leq \int_0^T e^{-\sigma M}\sup_x\sup_y \int_{\|z\|>M}e^{2\sigma\|z\|}\rho_{x}(y,dz)dt + \frac{1}{\sigma}(K+\varepsilon).
\end{align*}
Similar to Lemmas \ref{lem:tightness} and \ref{lem:ui}, taking $j\to \infty$, $M\to \infty$ and then $\sigma \to \infty$, in that order, yields the desired limit. This proves that $\mu^j$ is tight. 

To finish the proof, take a convergent subsequence of $\mu^j$, with some limit $\mu$ and we define $\phi$ as the solution to the following ODE:
\begin{equation*}
    \phi(t) = x_0 + \int_0^tg(\phi(s,y))\mu(dy|s)ds.
\end{equation*}
From Lemma \ref{lem_unique} we know that the solution is unique. Moreover, an argument by contradiction using different $\varepsilon$ shows that this solution $\phi$ is independent of $\varepsilon$. Therefore, there exists a subsequence of $\phi^j$ that converges and thus is precompact. The proof is complete if we can show that the limiting $\phi$ satisfies $I(\phi)\leq K$. Using Fatou's lemma, the lower semi-continuity of $R$ and the Feller property of $\rho$, we have,
\begin{align*}
    K&\geq \liminf_{j \to \infty} I(\phi^j) \\
    &\geq \liminf_{j \to \infty} \int_0^T R(q^j(y,dz|t)\mu^j(dy|t)||\rho_{\phi^j(t)}(y,dz)\mu^j(dy|t))dt - \varepsilon\\
    &\geq \int_0^T R(q(y,dz|t)\mu(dy|t)||\rho_{\phi(t)}(y,dz)\mu(dy|t))dts \\
    &\geq I(\phi) - \varepsilon.
\end{align*}
Since $\varepsilon$ is arbitrary, we have $I(\phi)\leq K$. Thus $I$ has compact level sets and is a rate function.
\end{proof}
\end{appendix}
\begin{acks}[Acknowledgments]
The authors thank Prof A.\ Budhiraja for helpful feedback on a first version of the paper in connection with the Ph.D.\ defence of AL. 

The research of HH, AL and PN was supported by the Wallenberg AI, Autonomous Systems and Software Program (WASP) funded by the Knut and Alice Wallenberg Foundation, Sweden.
HH was also supported in part by the Swedish Research Council (VR-2021-05181). PN was also supported in part by the Swedish Research Council (VR-2018-07050, VR-2023-03484). The research of GW was supported in part by the Swedish e-Science Research Centre. 
\end{acks}
\bibliographystyle{imsart-nameyear}
\bibliography{refs}

\end{document}